\documentclass[11pt]{elsarticle}
\usepackage{lineno,hyperref}
\usepackage{amsmath,amsfonts,amsmath,amssymb,amsthm,cases}
\usepackage{amsfonts}
\usepackage{graphics}
\usepackage{caption2}
\usepackage{amssymb}
\usepackage{amsmath}
\usepackage{graphicx}
\usepackage{multirow}
\usepackage{subfigure}
\usepackage{wasysym}
\usepackage{color}
\usepackage[dvipsnames]{xcolor}
\usepackage{fullpage}
\usepackage{algpseudocode}
\usepackage{algorithmicx}
\usepackage[section]{algorithm}
\usepackage{amsthm}
\usepackage{bm}
\makeatletter
\newcommand{\algmargin}{\the\ALG@thistlm}
\makeatother
\newlength{\whilewidth}
\algdef{SE}[parWHILE]{parWhile}{EndparWhile}[1]
{\parbox[t]{\dimexpr\linewidth-\algmargin}{%
                \hangindent\whilewidth\strut\algorithmicwhile\ #1\ \algorithmicdo\strut}}{\algorithmicend\ \algorithmicwhile}%
\algnewcommand{\parState}[1]{\State%
        \parbox[t]{\dimexpr\linewidth-\algmargin}{\strut #1\strut}}

\def\alglongline{\noindent\rule{\textwidth}{0.6pt}}

\numberwithin{equation}{section}

\newtheorem{thm}{Theorem}[section]
\newtheorem{lemmaa}[thm]{\textbf{Lemma}}
\newtheorem{rulee}[thm]{\textbf{Rule}}

\newtheorem{rem}{Remark}

\newtheorem{algorithm_}{Algorithm}[section]
\def\p{\partial}
\def\DD{\displaystyle}
\def\R{\mathbb{R}}
\def\ii{\mathbf{i}}

\def\l{\,l}
\def\bx{{\bf x}}
\def\by{{\bf y}}
\def\bz{{\bf z}}

\def\f{\mbox{field}}
\def\a{\alpha}

\def\B{B}
\def\d{d}                       %not \tilde{d}
\def\P{\mathcal{P}}

\def\L{{\mathcal{L}}}

\def\Nbx{N_1}
\def\Nby{N_2}
\def\Nbz{N_3}
\def\MaxN2{\max\{\Nbx, \Nby\}}
\def\MaxN3{\max\{\Nbx, \Nby, \Nbz\}}

\def\f32{\frac{3}{2}}

\def\s1{{s-1}}
\def\s2{{s-2}}
\def\d{\boldsymbol{d}}
\def\n{\mathbf{n}}

\def\betaH{\mu}
\def\dirThreeD{\Box, \vartriangle, \ocircle}

\def\rangeTwoDRow{i = 1,\ldots,\Nbx \\ j = 1,\ldots,\Nby}

\def\rangeThreeDRow{i = 1,\ldots,\Nbx \\ j = 1,\ldots,\Nby \\ k = 1,\ldots,\Nbz}

\DeclareMathOperator\supp{supp}

\DeclareMathOperator*{\argmin}{arg\,min}

\def\plusd{}
\def\minusd{}

\def\Potential{\Psi}
\def\PotentialVec{\Phi}

\def\eps{{\varepsilon}}

  \def\kdown{\kappa_{\downarrow}}
  \def\kup{\kappa_{\uparrow}}

\def\uprightarrow{
   \uparrow\hspace{-2pt}\raisebox{-1pt}{$_\rightarrow$}   
}
 
\def\abbrLeft{\text{Lf}}
\def\abbrPred{\text{Dr}}
\def\abbrRefl{\text{Rf}}

\def\quadOne{\Omega^{+;+}}
\def\quadTwo{\Omega^{-;+}}
\def\quadThree{\Omega^{-;-}}
\def\quadFour{\Omega^{+;-}}

\def\OmegaXminus{\Omega_{x_1}^-}
\def\OmegaXplus{\Omega_{x_1}^+}
\def\OmegaYminus{\Omega_{x_2}^-}
\def\OmegaYplus{\Omega_{x_2}^+}

\newcommand{\mcol}[1]{\multicolumn{1}{|c|}{#1}}
\newcommand{\mrow}[1]{\multirow{2}{*}{#1}}
\newcommand{\mrowcol}[1]{\multicolumn{1}{|c|}{\multirow{2}{*}{#1}}}   %{\multirow{2}{*}{\multicolumn{1}{c}{#1}}}

\def\Tsolve{T_{\text{slv}}}

\journal{}

\begin{document}

\begin{frontmatter}

\title{Trace Transfer-based Diagonal Sweeping Domain Decomposition Method for \\the Helmholtz Equation: Algorithms and Convergence Analysis}

\author[lsec]{Wei Leng\fnref{hjgfootnote}}
\ead{wleng@lsec.cc.ac.cn}
\address[lsec]{State Key Laboratory of Scientific and Engineering Computing, Chinese Academy
  of Sciences, Beijing 100190, China}
  \fntext[hjgfootnote]{W. Leng's research is partially supported by 
National Key R\&D Program of China under grant number 2020YFA0711904,
National Natural Science Foundation of China under grant number 11771440,
and National Center for Mathematics and Interdisciplinary Sciences of Chinese Academy of Sciences (NCMIS).}
\cortext[mycorrespondingauthor]{Corresponding author}
\author[usc]{Lili Ju\corref{mycorrespondingauthor}\fnref{jllfootnote}}
\address[usc]{Department of Mathematics, University of  South Carolina, Columbia, SC 29208, USA}
\fntext[jllfootnote]{L. Ju's research is partially supported by US National Science Foundation under grant number DMS-1818438.}
\ead{ju@math.sc.edu}

\begin{abstract}
By utilizing the perfectly matched layer (PML) and source transfer techniques,
the diagonal sweeping domain decomposition method (DDM)  was recently developed for solving the high-frequency Helmholtz equation in $\R^n$,  which uses $2^n$ sweeps along respective diagonal directions with  checkerboard domain decomposition. Although this diagonal sweeping DDM is essentially multiplicative,  it is highly suitable for parallel computing of the Helmholtz problem with multiple right-hand sides when combined with the pipeline processing since the number of sequential steps in each sweep is much smaller than the number of subdomains. In this paper, we propose and analyze a trace transfer-based diagonal sweeping DDM. A major advantage of changing from source transfer to trace transfer for information passing between neighbor subdomains is that the resulting diagonal sweeps become  easier to analyze and implement and more efficient, since the transferred traces have only $2n$ cardinal directions between neighbor subdomains while the transferred sources come from a total of $3^n-1$ cardinal and corner directions. We rigorously prove that the proposed diagonal sweeping DDM not only gives the exact solution of the global PML problem  in the constant medium case but also does it with  at most one extra round of diagonal sweeps in the two-layered media case, which lays down the theoretical foundation of the method. Performance and parallel scalability of the proposed DDM as direct solver or preconditioner are also numerically demonstrated through extensive experiments in two and three dimensions.

\end{abstract}

%\ams{65N55, 65F08, 65Y05}

\begin{keyword}
Domain decomposition method, diagonal sweeping, Helmholtz equation, perfectly matched layer, trace transfer, parallel computing
\end{keyword}

\end{frontmatter}

% figure control

\newif\ifdraftFig

\draftFigtrue

%%%%%%%%%%%%%%%%%%%%%%%%%%%%%%%%%%%%%%%%%%%%%%%%%%%%%%%%%%%%%%%%
\section{Introduction}
%%%%%%%%%%%%%%%%%%%%%%%%%%%%%%%%%%%%%%%%%%%%%%%%%%%%%%%%%%%%%%%%

In this paper, we consider the  Helmholtz equation in $\R^n$ ($n=2,3$) with the Sommerfeld radiation condition: \\
\begin{align} 
  \Delta u + \kappa^2 u &= f ,\qquad  \mbox{in} \;\;\; \R^n \label{eq:helm}\\
  r^{\frac{n-1}{2}}\Big(\frac{\p u}{\p r} - \mathbf{i} \kappa u\Big) &\rightarrow 0, \qquad \mbox{as} \;\;\; r = |\bx| \rightarrow \infty, \label{eq:radiation}
\end{align}
where $f$ is the source and $\kappa(\bx)$ is the wave number  defined by $\kappa(\bx) := \frac{\omega}{c(\bx)} $ with $\omega$ denoting the angular frequency and $c(\bx)$ the wave speed.
The Helmholtz equation has  practical applications in diverse areas, such as acoustics, elasticity, electromagnetism and geophysics, in which  the computational costs mainly  come from  solving the  Helmholtz equation numerically. It is a challenging  task to design   efficient and robust solvers for the Helmholtz problem with large wave number $\kappa(\bx)$, since in this situation  the resulting discrete system is  highly indefinite and the Green's function of the Helmholtz operator is very oscillatory \cite{Engquist2016}. Many numerical methods have been proposed to solve  the Helmholtz problem   \eqref{eq:helm}- \eqref{eq:radiation}, including the direct methods \cite{Frontal1, George1973} with the sparsity of the coefficient matrix being exploited \cite{Xia2010}, %Gillman2015, Leng2015}
the multigrid methods with the shifted Laplace used as the preconditioner \cite{Erlangga2006, Erlangga2006b,Erlangga2008,Umetani2009,Sheikh2013,Aruliah2002},  and the domain decomposition methods (DDM) with different types of transmission conditions \cite{Despres1990, Collino2000, SchwarzBdry1, Douglas1998,Boubendir2012, Schadle2007, Toselli1999}.  Our work in this paper is devoted to development and analysis of a new effective  and efficient DDM  for solving the high-frequency Helmholtz equation.

The sweeping type DDM was first proposed by Engquist and Ying in \cite{Engquist2011a, Engquist2011b}, and then further developed in \cite{Stolk2013, Vion2014, Zepeda2014, Chen2013a, GanderReview}.
%where the source transfer method \cite{Chen2013a} and polarized trace method \cite{Zepeda2014}.
The sweeping type DDMs are quite effective for solving the Helmholtz equation due to two ingredients:  one is the employment of  the perfectly matched layer (PML)  boundary condition on  subdomains, the other is the  transmission condition between neighbor subdomains. The latter one is also the main difference among existing sweeping type DDM  approaches. In most of them, the domain is only partitioned into layers along a single direction, and the layered subdomain problems are then solved one after another through the forward and backward sweeps.  The subdomain problems are preferred to be solved with some direct methods, which can greatly reduce the computational cost in sweeps. However, the factorization processes  for subdomain problems in the direct methods are often computationally expensive, particularly for 3D problems.
To overcome such difficulty by the one-directional domain partition in the sweeping type DDMs, the structured domain decomposition along all spatial directions (i.e., checkerboard domain decomposition) naturally comes to consideration. One obvious way to build the sweeping DDM for the checkerboard partition is to use recursion as done in \cite{Liu2015b, Wu2015}, however, the number of sequential steps in  each of such recursive sweeps are proportional to the number of subdomains, thus still not practical for real large-scale applications.

Inspired by the source trace method \cite{Chen2013a}, the corner transfer for the checkerboard domain decomposition was introduced for the first time to design the additive overlapping DDM in \cite{Leng2019}.
Another DDM based on the corner transfer is the ``L-sweeps'' method recently proposed in \cite{Zepeda2020},
in which the sweeps of $3^n-1$ directions (cardinal and corner) with trace transfer are performed to construct the total solution, and the number of sequential steps in each sweep is reduced to be proportional to the $n$-th root of the number of subdomains. Thus this  method becomes very attractive for solving in parallel  high-frequency Helmholtz problems with multiple right-hand sides using the pipelining processing.
One of the major differences between the additive overlapping DDM \cite{Leng2019} and the L-sweeps method \cite{Zepeda2020} is the subdomain solving order. 
More recently, the diagonal sweeping DDM with source transfer was developed in \cite{Leng2020} based on  further re-arranging and improving the subdomain solving order,  in which  only $2^n$  sweeps of diagonal directions are needed and  the reflections in the media is also handled in a more proper way.

In this paper, we design and analyze a new diagonal sweeping DDM by changing the information passing  between neighbor subdomains for diagonal sweeping from source transfer to trace transfer.  Consequently, the proposed DDM is easier to analyze and implement and more efficient since the transferred traces have only $2n$ cardinal directions, while the transferred sources are from $3^n-1$ cardinal and corner directions. 
Moreover,  the overlapping of subdomains in domain decomposition becomes only optional while it is essential for the source transfer approach.   We rigorously prove that the DDM solution is indeed the solution of the Helmholtz problem in the case of constant medium, and such result  also holds with at most one extra round of diagonal sweeps in the  two-layered media case. 
The rest of the paper is organized as follows.  In Section 2, the perfectly matched layer  and trace transfer techniques are first introduced with some useful analysis results,  and  the trace transfer-based diagonal sweeping DDM with is then proposed for solving the Helmholtz equation in $\R^2$. Convergence analysis for the proposed method in the constant medium and two-layered media cases is then presented  in Section  3.
The extension of the method to the Helmholtz equation in $\R^3$ is given and discussed in Section 4.  Through extensive numerical experiments in two and three dimensions, we verify the convergence of the proposed method and demonstrate its efficiency and parallel scalability in Section 5. Finally some concluding remarks are drawn in Section 6.

%%%%%%%%%%%%%%%%%%%%%%%%%%%%%%%%%%%%%%%%%%%%%%%%%%%%%%%%%%%%%%%%
\section{Trace transfer-based diagonal sweeping DDM in $\R^2$}
%%%%%%%%%%%%%%%%%%%%%%%%%%%%%%%%%%%%%%%%%%%%%%%%%%%%%%%%%%%%%%%%
In this section, we first recall  the perfectly matched layer formulation for the Helmholtz problem \eqref{eq:helm}- \eqref{eq:radiation} in $\R^2$, then introduce the trace transfer technique with related lemmas for algorithm design and analysis,  and finally present the trace transfer-based diagonal sweeping DDM in $\R^2$.

%%%%%%%%%%%%%%%%%%%%%%%%%%%%%%%%%%%
\subsection{Perfectly matched layer}
%%%%%%%%%%%%%%%%%%%%%%%%%%%%%%%%%%%

Suppose that the source $f$ is compactly supported in a rectangular box in $\R^2$, defined by $\B = \{\bx=(x_1,x_2) \;|\; a_j \le x_j \le b_j,\, j = 1,2\}$, then the Sommerfeld radiation condition \eqref{eq:radiation} in $\R^2$ can be replaced by the PML boundary outside the box $B$ and the solution to the Helmholtz equation \eqref{eq:helm}   inside the box $\B$ correspondingly can  be obtained by solving the resulting PML problem.
We adopt the uniaxial PML \cite{Berenger, Chew1994, Kim2010, Bramble2013, Chen2013a} and use the PML medium profile defined in the following.
Let $\{\sigma_j\}_{j=1}^2$ be piecewise smooth functions such that
\begin{equation} \label{eq:sigma}
\sigma_j(\bx) = \left\{  
\begin{array}{ll}
\widehat{\sigma}(x_j - b_j), & \text{if} \,\,\, b_j \le x_j,\\
0,                         & \text{if} \,\,\, a_j < x_j < b_j,\\
\widehat{\sigma}(a_j - x_j), & \text{if} \,\,\, x_j \le a_j,\\
\end{array}
\right. 
\end{equation}
where $\widehat{\sigma}(t)$ is a smooth medium profile function such as \cite{Chen2010,Chen2013a,Chen2017,Zheng2020},
then the complex coordinate stretching $\widetilde{\bx}(\bx)$  is defined as
\begin{equation}
\tilde{x}_j(x_j) = c_j+\int_{c_j}^{x_j} \alpha_j(t)\, dt = x_j
+ \ii \int_{c_j}^{x_j} \sigma_j(t) \,dt,\qquad j = 1,2,
\end{equation}
where $c_j = \frac{a_j+b_j}{2}$, for $j = 1,2$, and $\alpha_1(x_1) = 1 + \ii \sigma_1(x_1) $,
$\alpha_2(x_2) = 1 + \ii \sigma_2(x_2) $.

The PML equation  derived from \eqref{eq:helm}  by using the chain rule then can be written as follows:
%with the source $f$:
\begin{equation} \label{eq:PML}
J_{\B}^{-1} \nabla \cdot (A_{\B} \nabla\tilde{u}) + \kappa^2 \tilde{u} = f, \qquad \text{in} \,\, \R^2,
\end{equation}
where $\tilde{u}$ is called the PML solution and
$$\DD A_{\B}(\bx) = \mbox{diag}\left(\frac{\a_2(x_2)}{\a_1(x_1)}, \frac{\a_1(x_1)}{\a_2(x_2)}\right),\quad  J_{\B}(\bx) = \a_1(x_1) \a_2(x_2).$$
It has been shown in \cite{Chen2010,Chen2013a} that the PML equation \eqref{eq:PML} is well-posed and its solution decays exponentially outside the box for both the constant medium and the two-layered media cases.
From now on,   the wave number $\kappa(\bx)$ is assumed to be either constant or two-layered  except stated otherwise. For simplicity,  we also assume that the media interface does not coincide with the subdomain boundaries in the two-layered media case, .

Let us denote the PML problem \eqref{eq:PML} associated with the box  $\B$ and the wave number  $\kappa$ as  $\P_{B,\kappa}$ and the corresponding  linear operator as $$\L_{\B,\kappa}=J_{\B}^{-1} \nabla \cdot (A_{\B} \nabla \, \boldsymbol{\cdot}) + \kappa^2.$$  Denote the fundamental solution of the problem $\P_{\B,\kappa}$ as $G_{B,\kappa}(\bx, \by)$, which satisfies $\L_{B,\kappa} G_{B,\kappa}(\bx, \by) = -\delta_{\by}(\bx)$ for any $\bx\in\R^2$ (since the Helmholtz operator is usually defined as $-\Delta-\kappa^2$). It has been proved in \cite{Lassas2001} that 
\begin{align}
  \label{eq:fund_sol_sym}
\DD G_{B,\kappa}(\by,\bx) := \frac{J_B(\bx)}{J_B(\by)} G_{B,\kappa}(\bx,\by)
\end{align}
is the fundamental solution to the adjoint PML equation.
Following Lemma 2.1 and 4.5 of \cite{Chen2017}, it is easy to obtain the following result with  similar derivations. 
 
\begin{lemmaa} \label{lemma:singularity}
  For continuous $G_{B,\kappa}(\bx,\by)$, $\p_{y_i} G_{B,\kappa}(\bx, \by)$ and $\p^2_{y_i} G_{B,\kappa}(\bx, \by)$, $i=1,2$,
  there exists a constant $C$ depending only on the rectangular box $B$ and the wave number $\kappa$ such that for any $\bx$ and $\by$, it holds that
  \begin{align}
    G_{B,\kappa}(\bx, \by) &\leq C \left(\frac{1}{|\bx - \by|^{1/2}} + \frac{1}{|\bx - \by'|^{1/2}} + 1 \right), \\
    \p_{y_i} G_{B,\kappa}(\bx, \by) &\leq C \left(\frac{1}{|\bx - \by|^{3/2}} + \frac{1}{|\bx - \by'|^{3/2}}  + 1 \right),\qquad i=1,2, \\
    \p^2_{y_i} G_{B,\kappa}(\bx, \by) &\leq C \left(\frac{1}{|\bx - \by|^{5/2}} + \frac{1}{|\bx - \by'|^{5/2}}  + 1 \right),\qquad i=1,2,
  \end{align}
  where $\by'$ is  identical to $\by$ in the constant medium case or the image point of $\by$ with respect to the media interface  in the two-layered media case.
\end{lemmaa}

\begin{figure}[htbp]    
        \centering
        \begin{minipage}[t]{0.34\linewidth}
                \vspace{0pt}
                \centering
                \includegraphics[width=1\textwidth]{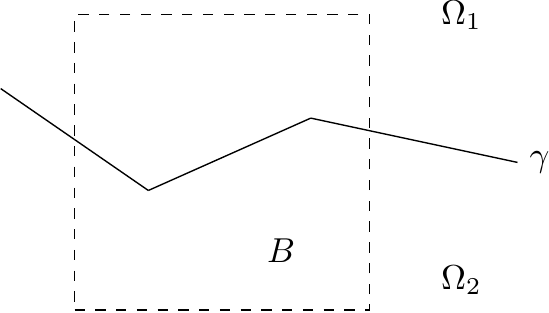}
        \end{minipage}
        \hspace{0.3cm}
        \begin{minipage}[t]{0.28\linewidth}
                \vspace{0pt}
                \centering
                \includegraphics[width=1\textwidth]{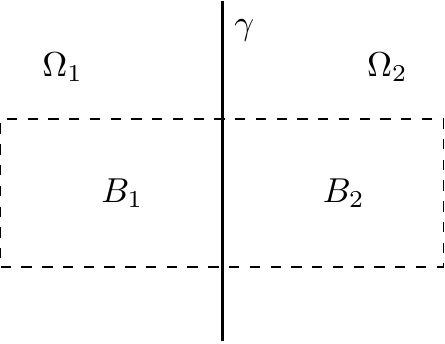}
        \end{minipage}
        \hspace{0.5cm}
        \begin{minipage}[t]{0.26\linewidth}
                \centering
                \vspace{0pt}
                \includegraphics[width=1\textwidth]{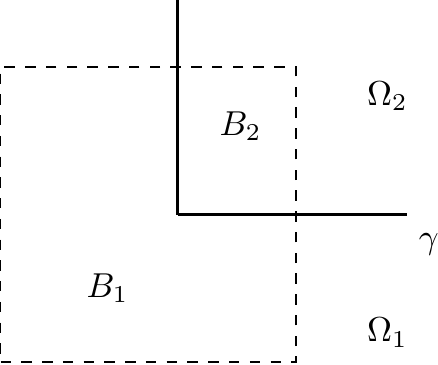}
        \end{minipage}\\
        % \vspace{0.2cm}
        
        \caption{\it Divide $\R^2$ and the box $B$ with  $\gamma$.
          \label{fig:twopart}}
\end{figure}

%%%%%%%%%%%%%%%%%%%%%%%%%%%%%%%%%%%
\subsection{Trace transfer}
%%%%%%%%%%%%%%%%%%%%%%%%%%%%%%%%%%%

The trace transfer technique \cite{Zepeda2014} used in the polarized trace method is  first stated below. Suppose that a piecewise smooth curve $\gamma$ divides $\R^2$ into two parts $\Omega_1$ and $\Omega_2$, and also divides the rectangular box $\B$ into two parts, as shown in Figure \ref{fig:twopart}-(left).
Define the discontinuous cutoff function $\betaH_{\Omega'}$ associated with a given region $\Omega'$  as 
\begin{align*}
        \betaH_{\Omega'}(\bx) = \left\{  
        \begin{array}{ll}
                1,                          & \text{if} \,\,\, \bx \in \Omega' \setminus \p \Omega',\\
                1/2, & \text{if} \,\,\, \bx \in \p\Omega',\\
                0, & \text{if} \,\,\, \bx \notin \Omega' \cup \p\Omega'.
        \end{array}
        \right. 
\end{align*}
which is the counterpart of the smooth cutoff function in the source transfer technique \cite{Chen2013a,Leng2020}. % And it is different from $\chi_{\Omega'}$.
Base on the exponential decay of the solution \cite{Chen2010,Chen2013a} and the fundamental solution to the PML problem,
one can easily get the following result on trace transfer by using integration by parts.
%which is derived from potential theory, 

\begin{lemmaa} \label{lemma:trace} 
  Suppose $\supp(f) \subset \Omega_1 \cap \B $.
  Let $u$ be the solution to the problem $\P_{B,\kappa}$ with  the source $f$ (i.e, $\L_{B,\kappa} u = f$ in $\R^2$).
  % Given $u_1$ as the restriction of $u$ on ${\Omega}_1$, such as $u_1 := u {\chi}_{{\Omega}_1} $, 
  Let $u_2$ be the  potential using the transferred trace of $u$ on ${\gamma}$, defined by
  \begin{align} \label{eq:trace_potential}
    u_2 (\bx) := \DD \int_{{\gamma}} 
    J_B^{-1} G_{B,\kappa}(\bx,\by) \big(A_B \nabla_{\by} u(\by) \cdot \n \big) 
    - u(\by) \big(A_B\nabla_{\by}(J_B^{-1} G_{B,\kappa} (\bx, \by)) \cdot \n \big)\, d\by,
  \end{align}
  where $\n$ is the unit normal of ${\gamma}$ pointing to $\Omega_2$.  Then we have that
  \begin{align} \label{eq:trace_potential_total}
    u \betaH_{{\Omega}_1} + u_2\mu_{\Omega_2}= u, \quad \text{in} \,\, \R^2.
  \end{align} 
   Moreover, $u_2 = 0$ in ${\Omega}_1\setminus{\gamma} $.
\end{lemmaa}

We note that the following convention is used for the potential  \eqref{eq:trace_potential}  throughout this paper: the function $u$ and its derivatives in the integrand over $\gamma$ are defined as their corresponding limits from the positive $\n$ side,
  and the value of the integral over $\gamma$ is also defined as its limits from the positive $\n$ side. A similar result as  Lemma \ref{lemma:trace} also holds for  the corresponding PML problem of the Helmholtz equation in $\R^3$.

The trace transfer will be applied in the diagonal sweeping DDM \cite{Leng2020} for  horizontal or vertical interfaces or quadrant region interfaces, as shown in Figure \ref{fig:twopart}-(middle and right).
Lemma \ref{lemma:trace}  implies a procedure of two subdomains solving: suppose we only know the solution of $\P_{B,\kappa}$ in $\Omega_1$ and smoothly extend it to $\Omega_2$ so that its limit on $\gamma$ from the $\Omega_2$ side is well defined,  then the solution in $\Omega_2$ could be obtained using the potential \eqref{eq:trace_potential}, and together they form the total solution in $\Omega$ with \eqref{eq:trace_potential_total}.
 Building a domain decomposition method further requires changing the global problem solving in the procedure to subdomain problems solving, that is, substitute the fundamental solution $G_{B,\kappa}(\bx,\by)$  in \eqref{eq:trace_potential} to the ones of subdomains. 
 
  Assume that $\Omega_2 \cap B$ is also a rectangular (denoted by $B_2$),
 then a subdomain PML problem associated with  $B_2$ can be built and  the wave number of this subdomain problem, denoted by $\kappa_2$, 
 is an extension of $\kappa|_{\Omega_2}$ to $\R^2$ but could be different from  $\kappa$ outside of $\Omega_2$.
 We then can prove the following lemma regarding such  subdomain PML problem $\P_{B_2,\kappa_2}$.
\begin{lemmaa} \label{lemma:change_domain}
  Suppose that $f$, $u$ and $u_2$ satisfy the conditions given in Lemma \ref{lemma:trace}, and additionally, $f$ is bounded and smooth.
  Assume that the wave number $\kappa_2$ satisfies the following two conditions:
  (i) $\kappa_2 = \kappa$ in $\Omega_2$ and 
  (ii) $\kappa_2$ is either constant or two-layered in $\R^2$.  % in horizontal or vertical direction.
  Let $\widetilde{u}_2$ be the potential associated with the subdomain PML problem $\P_{B_2,\kappa_2}$, defined as
  \begin{align} \label{eq:trace_potential2}
    \widetilde{u}_2 (\bx):= \DD \int_{{\gamma}} 
    J_{B_2}^{-1} G_{B_2,\kappa_2}(\bx,\by) \big(A_{B_2} \nabla_{\by} u(\by) \cdot \n\big) 
    - u(\by) \Big(A_{B_2} \nabla_{\by}(J_{B_2}^{-1} G_{B_2,\kappa_2} (\bx, \by)) \cdot \n \Big) d\by,
  \end{align}
  then $\widetilde{u}_2 = u_2$ in $\R^2$.
\end{lemmaa}
\begin{proof}
  We just need consider the case of horizontal trace transfer  (see Figure \ref{fig:twopart}-(middle)) for the horizontally layered media
  since all other cases are similar.
  By the definition of uniaxial  PML, it is clear that $J_{B_2} |_{\gamma}= J_B |_{\gamma}$  and $A_{B_2} |_{\gamma} = A_B |_{\gamma}$, and also  $G_{B_2,\kappa_2}(\bx,\by) = G_{B,\kappa}(\bx,\by)$ for $\bx \in \Omega_2$ and $\by \in \gamma$ , thus $\widetilde{u}_2 = u_2$ in $\Omega_2$. Since $u_2 = 0$ in $\Omega_1$, we will show that $\widetilde{u}_2 = 0$ in $\Omega_1$ by using the following argument.
  
  As in the source transfer approach \cite{Chen2013a}, an extended region of the subdomain is introduced.
  Let us define $\Omega_2^{\eps} := \{\bx \;|\; {\rm dist}(\bx, \Omega_2) \leq \eps \}$, which is an extended region of $\Omega_2$ by a distance of  $\eps>0$,
  and $\Omega^{\eps} = \Omega_2^{\eps} \setminus \Omega_2$.
  A smooth function $\beta^{\eps}$ exists such that $\beta^{\eps} = 1$ in $\Omega_1 \setminus \Omega^{\eps}$ and  $\beta^{\eps} = 0$ in $\Omega_2 $.
  Denote $u^{\eps}$ as the solution to problem $\P_{B,\kappa}$ with the source $f \cdot \chi_{\Omega_1 \setminus \Omega^{\eps} } $, i.e.,
 $    \L_{B,\kappa} u^{\eps} =  f \cdot \chi_{\Omega_1 \setminus \Omega^{\eps}}$ in $\R^2$,
  and  $v^{\eps}$ as the solution to the  problem
  \begin{align}
    \L_{B,\kappa} v^{\eps} =  - \L_{B,\kappa} (u^{\eps} \beta^{\eps} ) \chi_{\Omega^{\eps}} , \qquad \text{in} \,\, \R^2, \label{eq:lemma2_veps}
  \end{align}
  then it is easy to see  that 
  $  \L_{B,\kappa} v^{\eps} = f - \L_{B,\kappa} (u^{\eps} \beta^{\eps} )$ in $\R^2$,
  which implies
  \begin{align}
    v^{\eps} = u^{\eps} (1 - \beta^{\eps}) , \qquad \text{in} \,\, \R^2. \label{eq:lemma2_veps2}
  \end{align}

\begin{figure}[htbp]    
        \centering
        \begin{minipage}[t]{0.34\linewidth}
                        \centering
                \includegraphics[width=1\textwidth]{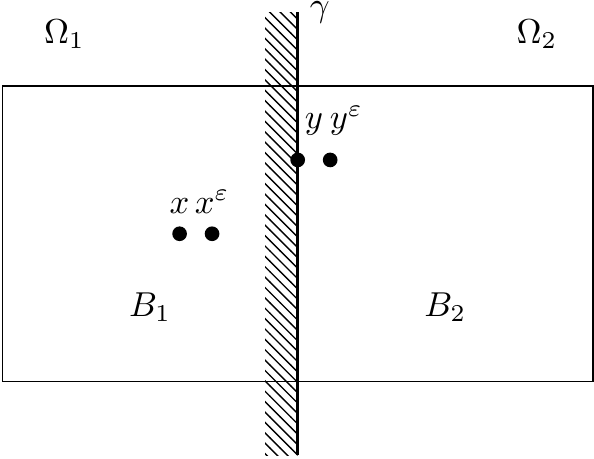}
        \end{minipage}
        \hspace{0.6cm}
        \begin{minipage}[t]{0.42\linewidth}
                \centering
                \includegraphics[width=1\textwidth]{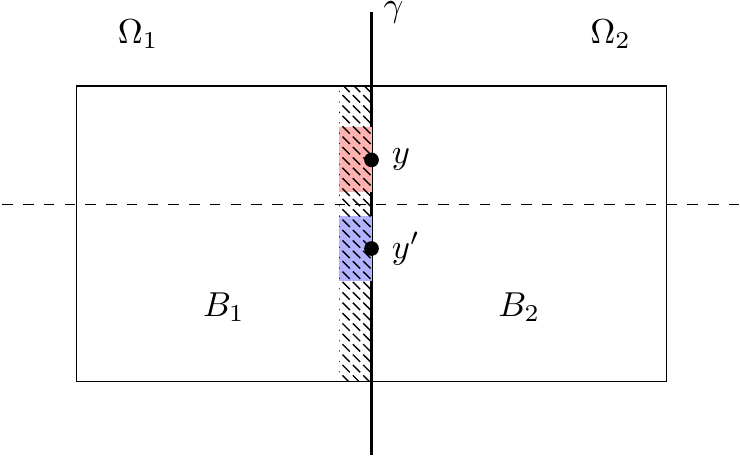}
        \end{minipage}
        \vspace{-0.1cm}
        \caption{\it Trace transfer for two subdomains. Left: the shadowed area denotes $\Omega^{\eps}$;
          right: the shadowed area denotes $\B^{\eps}$,  the dashed line denotes  the media interface,   the red shadowed region denotes  $B_{\by}^{\eps}$, and  the blue shadowed region denotes  $B_{\by'}^{\eps}$.
          \label{fig:lemma}}
\end{figure}

  Let the extended rectangular of $B_2$ be defined as $B_2^{\eps} = \Omega_2^{\eps}  \cap B$, as shown in Figure \ref{fig:lemma},  then by \eqref{eq:lemma2_veps} and \eqref{eq:lemma2_veps2}, we see that
  $v^{\eps}$ is also the solution to the following subdomain  problem associated with $B_2^{\eps}$:
  \begin{align}
    \L_{B_2^{\eps},\kappa_2} v^{\eps} =  - \L_{B,\kappa} (u^{\eps} \beta^{\eps} ) \chi_{\Omega^{\eps}} , \qquad \text{in} \,\, \R^2. \label{eq:lemma2_veps3}
  \end{align}
  Using Lemma \ref{lemma:trace} on \eqref{eq:lemma2_veps3}, we have that
    for any $\bx \in \Omega_1 \setminus \Omega^{\eps} $,
  \begin{align}\label{eqtt}
     \DD \int_{{\gamma^{}}} 
    J_{B_2^{\eps}}^{-1} G_{B_2^{\eps},\kappa_2}(\bx,\by) \big(A_{B_2^{\eps}} \nabla_{\by} u^{\eps}(\by) \cdot \n\big) 
    - u^{\eps}(\by) \Big(A_{B_2^{\eps}} \nabla_{\by}(J_{B_2^{\eps}}^{-1} G_{B_2^{\eps},\kappa_2} (\bx, \by)) \cdot \n \Big) d\by = 0,
  \end{align}
  where the fact of $u^{\eps} |_{\gamma}= v^{\eps}|_{\gamma}$ is used.  Clearly, we also have $J_{B_2^{\eps}} |_{\gamma} = J_{B_2} |_{\gamma} $ and $A_{B_2^{\eps}} |_{\gamma} = A_{B_2} |_{\gamma}$.
  Suppose that for any given $\bx\in \Omega_1$, 
  \begin{align}
    \lim\limits_{\eps \rightarrow 0} G_{B_2 ^{\eps },\kappa_2}(\bx, \by) &= G_{B_2,\kappa_2}(\bx, \by), \label{eq:G_eps} \\
    \lim\limits_{\eps \rightarrow 0} \nabla_{\by} (J_{B^{\eps}_2} G_{B_2 ^{\eps },\kappa_2}(\bx, \by)) &= \nabla_{\by} (J_{B^{}_2} G_{B_2,\kappa_2}(\bx, \by))\label{eq:gradG_eps}
  \end{align}
  uniformly for any $\by \in \gamma$,  and also suppose that
  \begin{align}
    \lim\limits_{\eps \rightarrow 0} u^{\eps }(\by) &= u(\by) \label{eq:u_eps},\\
    \lim\limits_{\eps \rightarrow 0} \nabla_{\by} u^{\eps }(\by) &= \nabla_{\by} u(\by) \label{eq:gradu_eps}
  \end{align}
  uniformly for any $\by \in \gamma$, then we can obtain from \eqref{eqtt} that for any $\bx \in \Omega_1$,
  \begin{align*} 
    \DD \int_{{\gamma}} 
    J_{B_2}^{-1} G_{B_2,\kappa_2}(\bx,\by) \big(A_{B_2} \nabla_{\by} u(\by) \cdot \n\big) 
    - u(\by) \Big(A_{B_2} \nabla_{\by}(J_{B_2}^{-1} G_{B_2,\kappa_2} (\bx, \by)) \cdot \n \Big) d\by = 0,
  \end{align*}
  that is,  $\widetilde{u}_2 = 0$ in $\Omega_1$.
  Hence now we  are left to verify \eqref{eq:G_eps}-\eqref{eq:gradu_eps}.
  
  To prove \eqref{eq:G_eps}-\eqref{eq:gradG_eps}, we first observe that
  \begin{align*}
   |G_{B_2^{\eps},\kappa_2} (\bx, \by) - G_{B_2,\kappa_2} (\bx, \by)| &= |G_{B_2,\kappa_2} (\bx^{\eps}, \by^{\eps}) - G_{B_2,\kappa_2} (\bx, \by)| \nonumber \\
  &\leq  |G_{B_2,\kappa_2} (\bx^{\eps}, \by^{\eps}) - G_{B_2,\kappa_2} (\bx, \by^{\eps})| 
  + |G_{B_2,\kappa_2} (\bx, \by^{\eps}) - G_{B_2,\kappa_2} (\bx, \by)|,
  %  \label{eq:G_eps_1}
  \end{align*}
  where $\bx^{\eps} = \bx + (\eps, 0)$, $\by^{\eps} = \by + (\eps, 0)$, as shown in Figure \ref{fig:lemma}-(left). For fixed $\bx \in \Omega_1$,  let us denote by $d_{\bx}$ the distance between $\bx$ and $\gamma$ and take $\eps < d_{\bx}/4$, 
  then by Lemma \ref{lemma:singularity} we have that
  \begin{align*}
    |G_{B_2,\kappa_2} (\bx^{\eps}, \by^{\eps}) - G_{B_2,\kappa_2} (\bx, \by^{\eps})|
    &= \eps |\partial_{x_1} G_{B_2,\kappa_2} (\bx + \theta (\eps, 0), \by^{\eps})|
    \leq C  \frac{\eps}{d_{\bx}},\\
    |G_{B_2,\kappa_2} (\bx, \by^{\eps}) - G_{B_2,\kappa_2} (\bx, \by)| &= \eps |\partial_{y_1} G_{B_2,\kappa_2} (\bx, \by + \hat\theta (\eps , 0)| \leq C  \frac{\eps}{d_{\bx}},
  \end{align*}
  for some $\theta, \hat \theta \in [0, 1]$, 
  which deduce that \eqref{eq:G_eps}-\eqref{eq:gradG_eps} hold uniformly for any $\by \in \gamma$ .

  %
  %
  %
  % u eps
  %
  %
  %
  To prove \eqref{eq:u_eps}-\eqref{eq:gradu_eps}, we first notice that  
  for any $\by \in \gamma$,
  \begin{align}
    |u^{\eps}(\by) - u(\by)| = \left|\int_{B^\eps} f(\bz) G_{B,\kappa}(\by, \bz) d\bz \right|  
                         \leq C \int_{B^\eps} |G_{B,\kappa}(\by, \bz)| d\bz. \label{eq:u_eps_1}
  \end{align}
 Without loss of generality, let assume  the box $B$ = $[-l_1, l_1] \times [-l_2, l_2]$  and the interface $\gamma = \{\xi\} \times (-\infty, \infty)$.
  We can split $B^\eps = [\xi - \eps, \xi] \times [-l_2, l_2]$
  into three parts (see Figure \ref{fig:lemma}-(right)):
  a  box region in the neighbor of $\by$, namely $B_{\by}^{\eps} = [\xi - \eps, \xi] \times [y_2 - \eps, y_2 + \eps] $,
  a box region $B_{\by'}^{\eps}$ in the neighbor of $\by'$ just as $B_{\by}^{\eps}$, 
  and the left region $B^\eps \setminus \big( B^\eps_{\by} \cup B^\eps_{\by'} \big)$.
  Then by Lemma \ref{lemma:singularity}, we have
  \begin{align}
    \int_{B^\eps_{\by}} |G_{B,\kappa}(\by, \bz)| d\bz \leq C \int_{B^\eps_{\by}} \left( \frac{1}{|\by-\bz|^{1/2}} + 1 \right) d\bz
    \leq C \int_{[0,\sqrt{2} \eps]} \left( \frac{1}{t^{1/2}} + 1 \right) t dt \leq C \eps^{3/2}, \label{eq:u_eps_2} 
  \end{align}
 and similar estimation also holds for the region $B^\eps_{\by'}$, and 
  \begin{align}
   \int_{B^\eps \setminus \big( B^\eps_{\by} \cup B^\eps_{\by'} \big)} |G_{B,\kappa}(\by, \bz)| d\bz &\leq C \int_{[\xi - \eps, \xi]} dz_1 \int_{[-l_2,l_2]\setminus \big( [y_2-\eps,y_2+\eps] \cup [y'_2-\eps,y'_2+\eps] \big)  }  \left( \frac{1}{|\by-\bz|^{1/2}} + 1\right)  dz_2\nonumber \\
   &\leq C \eps \int_{[\eps, 2l_2]} \left( \frac{1}{t^{1/2}} + 1 \right) dt \leq C \eps. 
    \label{eq:u_eps_3}
  \end{align}
  Therefore, it holds that
  \begin{align}
    \int_{B^\eps} |G_{B,\kappa}(\by, \bz)| d\bz \leq C (\eps^{3/2} + \eps).
    \label{eq:u_eps_4}
  \end{align}
  Similarly, we also have that
  \begin{align}
    \int_{B^\eps} |\nabla_{\bz} G_{B,\kappa}(\by, \bz)| d\bz \leq C \eps^{1/2}.
    \label{eq:u_eps_5}
  \end{align}
Based on   \eqref{eq:u_eps_1}, \eqref{eq:u_eps_4} and \eqref{eq:u_eps_5}, it is easy to find
  that \eqref{eq:u_eps}-\eqref{eq:gradu_eps} hold uniformly for any $\by \in \gamma$.
\end{proof}

% ----------------------------------------------------------------------------
%
%Domain decomposition
%
% ----------------------------------------------------------------------------
%%%%%%%%%%%%%%%%%%%%%%%%%%%%%%%%%%%
\subsection{Diagonal sweeping DDM}
%%%%%%%%%%%%%%%%%%%%%%%%%%%%%%%%%%%

We use the following domain decomposition for  the rectangular domain
$\Omega = [-l_1,l_1]\times[-l_2,l_2]$ in $\R^2$. 
Denote $\Delta\xi = 2 l_1 / \Nbx$,
$\xi_i = -l_1 + (i-1) \Delta \xi$ for $i = 1, 2, \ldots, \Nbx+1$, and
$\Delta\eta = 2 \l_2 / \Nby$, $\eta_j = -l_2 + (j-1) \Delta\eta$ for
$j = 1, 2, \ldots, \Nby+1$,
then   a set of $\Nbx\times \Nby$ nonoverlapping rectangular subdomains are given by
$$\Omega_{i,j}
: = [\xi_{i}, \xi_{i+1}] \times [\eta_{j}, \eta_{j+1}],\qquad i = 1, 2,\ldots, \Nbx,\; j = 1, 2,\ldots, \Nby.$$
Define the region $\Omega_{i0,i1;j0,j1}$, $1\leq i0\leq i1\leq N_1+1$,
$1\leq j0\leq j1\leq N_2+1$, as the union of the subdomains as
$$\Omega_{i0,i1;j0,j1} := \bigcup\limits_{i0 \leq i \leq i1 \atop j0 \leq j \leq j1} \Omega_{i, j}.$$
The decomposition of the source $f$ is defined as
$$f_{i,j} = f \cdot \chi_{\Omega_{i,j}} , \qquad i = 1, 2\ldots, \Nbx,\; j = 1, 2, \ldots, \Nby.$$
For the PML problem $\P_{\Omega_{i,j},\kappa_{i,j}}$ associated with  the subdomain $\Omega_{i,j}$,
the corresponding wave number $\kappa_{i,j}$ is defined through  extension of the subdomain interior wave number by
\begin{align}
  \label{eq:wave_number}
  \kappa_{i,j}(\bx) = \kappa(\bx'), 
\end{align}
where $\bx' = \argmin_{\by \in \Omega_{i,j}} |\bx - \by|$. We remark that with this definition the  subdomain  PML problem associated with a subdomain completely contained in the interior of one medium region  always has a constant wave number.
The regional PML problem $\P_{\Omega_{i0,i1;j0,j1},\kappa_{i0,i1;j0,j1}}$ associated with the region $\Omega_{i0,i1;j0,j1}$ uses the wave number $\kappa_{i0,i1;j0,j1}$, which is defined in a similar way as \eqref{eq:wave_number}.

%%%%%%%%%%%%%%%%%%%%%%%%%%%%%%%%%%%%%%%%%%%%%%%%%%%%%%%%%%%%%%%%%%%%%%%%%%
%
%   dDDM with trace transfer
%
%%%%%%%%%%%%%%%%%%%%%%%%%%%%%%%%%%%%%%%%%%%%%%%%%%%%%%%%%%%%%%%%%%%%%%%%%%%

The following notations are introduced. Let $H$ denote the Heaviside function and define the following  one-dimensional cutoff functions: 
\begin{align*}
\mu^{(1)}_{\Box; i}(x_1) &= \left\{
\begin{array}{ll}
  \DD H( x_1 - \xi_i ),  & \, \Box = -1 \,\text{and}\, i \neq 1, \\
  \DD H( \xi_{i+1} - x_1 ),  & \, \Box = 1 \,\text{and}\, i \neq N_1, \\
  1,                         & \, \text{otherwise}, 
\end{array}
\right. &  \mu^{(2)}_{\vartriangle; j}(x_2) &= \left\{
\begin{array}{ll}
\DD H( x_2 - \eta_j),  & \, \vartriangle = -1 \,\text{and}\, j \neq 1,\\
  \DD H(\eta_{j+1}  - x_2),  & \, \vartriangle = 1 \,\text{and}\, j \neq N_2, \\
  1,                         & \, \text{otherwise}, 
\end{array}
\right.
\end{align*}
for $\Box, \vartriangle = \pm 1$, $i = 1, \ldots \Nbx$, $j = 1, \ldots, \Nby$.
Then the two-dimensional cutoff function corresponding to the subdomain $\Omega_{i,j}$ is defined as
$$\betaH_{i, j}(x_1, x_2) = \betaH^{(1)}_{-1; i}(x_1)
\betaH^{(1)}_{+1; i}(x_1) \betaH^{(2)}_{-1; j}(x_2)
\betaH^{(2)}_{+1; j}(x_2),$$
for $i = 1, \ldots \Nbx$, $j = 1, \ldots, \Nby$.
These cutoff functions could be regarded as the discontinuous counterpart of the smooth cutoff functions used in the diagonal sweeping DDM with source transfer \cite{Leng2020}.
%
% \gamma and n, and G
%
For brief notations, set $G_{i,j} = G_{\Omega_{i,j},\kappa_{i,j}}$ the fundamental solution to the subdomain PML problem $\P_{\Omega_{i,j},\kappa_{i,j}}$,  denote the boundaries of $\Omega_{i,j}$ as
\begin{align*}
        \gamma_{\Box,\vartriangle; i, j} &= \left\{ 
        \begin{array}{ll}
                \{(x_1,x_2)\;|\;x_1 = \xi_i \minusd\},  & \, \text{if} \,\, (\Box, \vartriangle) = (-1,0) \,\,\text{and}\,\, i > 1, \\
                \{(x_1,x_2)\;|\; x_1 = \xi_{i+1} \plusd\},  & \, \text{if} \,\, (\Box, \vartriangle) = (+1,0) \,\,\text{and}\,\, i < \Nbx,\\
                \{(x_1,x_2)\;|\; x_2 = \eta_j \minusd\},  & \, \text{if} \,\,  (\Box, \vartriangle) = (0,-1) \,\,\text{and}\,\, j > 1,\\
                \{(x_1,x_2)\;|\; x_2 = \eta_{j+1} \plusd\},  & \, \text{if} \,\,  (\Box, \vartriangle) = (0,+1)  \,\,\text{and}\,\, j < \Nby, \\
                \varnothing , &    \text{otherwise,}        
        \end{array}
        \right. 
\end{align*}
and the corresponding unit normal vectors associated with them are then $\n_{\Box,\vartriangle} =  -(\Box, \vartriangle)$.
Then we are able to define the potential operators  as follows:
\begin{align}
\Potential_{\Box, \vartriangle;\, i, j} (v)
= \PotentialVec_{\Box, \vartriangle;\, i, j} \left( (v, A_{\Omega_{i,j}}  \nabla v \cdot \n_{\Box, \vartriangle})^T \right),
\end{align}
where
\begin{align*}
        \PotentialVec_{\Box, \vartriangle;\, i, j} ( (v, w)^T ):=& 
        \DD  \int_{\gamma_{\Box, \vartriangle;\, i, j}}
         J_{\Omega_{i,j}}^{-1} G_{{i, j}}(\bx,\by )
                w(\by)
                - v(\by) \Big( A_{\Omega_{i,j}} \nabla_{\by} \big( J_{\Omega_{i,j}}^{-1} G_{{i, j}}(\bx, \by) \big) \cdot \n_{\Box, \vartriangle} \Big)  d\by,
\end{align*}
for $(\Box, \vartriangle) = (\pm 1, 0), (0, \pm 1)$. 
%And we have the additive overlapping DDM with trace transfer in $\R^2$ as follows,
%which is the counterpart for the  additive source transfer algorithm \ref{alg:add2D}.

% Sweeping order:
 There are totally $2^2=4$ diagonal directions in $\R^2$: $(+1,+1)$, $(-1,+1)$, $(+1,-1)$, $(-1,-1)$, and the sweep along each of the directions contains a total of $(\Nbx-1)+(\Nby-1)+1 = \Nbx+\Nby-1$ steps.
 The sweeping order is of top-right, top-left, bottom-right, and bottom-left, denoted respectively by
\begin{align} \label{Sorder2D}
\begin{array}{llll}
(+1,+1),& (-1,+1),& (+1,-1),& (-1,-1).
\end{array}
\end{align}
In the $s$-th step of the sweep of direction $(d_1, d_2)$, the group of subdomains $\{ \Omega_{i,j} \}$ satisfying
  \begin{equation}
    \label{eq:step_subd}    
    \widehat{i}(i)+\widehat{j}(j)+1 = s 
  \end{equation}
with
\begin{align} \label{eq:step_subd_ij}    
\widehat{i}(i) &= \left\{ 
\begin{array}{ll}
i-1,  &\,\,\, \text{if} \,\,\,\, d_1 = 1,\\
\Nbx-i, & \,\,\, \text{if} \,\,\,\, d_1 = -1,
\end{array}
\right. &  
\widehat{j}(j) &= \left\{
\begin{array}{ll}
j-1,  &\,\,\, \text{if} \,\,\,\, d_2 = 1,\\
\Nby-j, & \,\,\, \text{if} \,\,\,\, d_2 = -1,
\end{array}
\right. 
\end{align}
are to be solved.
Define that two vectors $\d_1$ and $\d_2$ in $\R^2$ are in the similar direction if $\d_1 \cdot \d_2 > 0$, the following rules on the transferred traces during sweeps in $\R^2$ are introduced, which are the same with those in \cite{Leng2020}.
\begin{rulee}{\rm (Similar directions in $\R^2$)} \label{rule2d_a}
A transferred trace  which is not in the similar direction of one sweep in $\R^2$, should not be used in that sweep.
\end{rulee}
\begin{rulee}{\rm (Opposite directions in $\R^2$)}      \label{rule2d_b}
 The transferred trace generated in one sweep in $\R^2$ should not be used in a later sweep if these two sweeps have opposite directions.
\end{rulee}

%With the above notations and rules,
Our trace transfer-based diagonal sweeping DDM in $\R^2$ is then stated in the following.

\alglongline
\begin{algorithm_}[Trace transfer-based Diagonal sweeping DDM in $\R^2$]~ 
        
        \label{alg:diag2D_t}
        \begin{algorithmic}[1]
                \parState {Set the sweep order as $(+1,+1)$,  $(-1,+1)$,  $(+1,-1)$, $(-1,-1)$ } 
                \parState {Set the local trace of each subdomain for each sweep as
                        $\mathbf{g}^{l}_{\Box,\vartriangle; i,j} = 0$, $l=1,2,3,4$}
                \For{Sweep $l = 1,\ldots,4$}
                \For{Step $s = 1, \ldots, \Nbx+\Nby-1$ } 
                \For {each of the subdomains $\{\Omega_{i,j}\}$ defined by \eqref{eq:step_subd} in Step $s$ of Sweep $l$} 
                  % the subdomain $\Omega_{i,j}$ in Step $s$ of Sweep $l$}    
                \parState {If $f_{i,j} \neq 0$, then
                        solve the local solution %$u_{i, j}^{l}$ with the source $f_{i,j}$}
                        %\If {$f_{i,j} \neq 0$}
                        %\begin{eqnarray*}
                        $\L_{\Omega_{i,j},\kappa_{i,j}} (u_{i, j}^{l}) =  f_{i,j}$,\\
                        %\end{eqnarray*}
                        %\Else
                        %\State $u_{i, j}^{l} = 0$   
                        %\EndIf
                        %\State Set $f_{i,j} = 0$
                        else set $u_{i, j}^{l} = 0$ without solving}
                \State Set $f_{i,j} = 0$
                \State Add potentials to the local solution
                \begin{eqnarray} \label{eq:subd_solve}
                        u_{i, j}^{l} \gets u_{i, j}^{l} + \sum\limits_{\substack{(\Box,\vartriangle) = (\pm 1, 0), (0, \pm 1)} }  \PotentialVec_{\Box, \vartriangle;\, i, j} (\mathbf{g}^{l}_{\Box,\vartriangle; i, j})
                \end{eqnarray}
                %\parState {Compute new transferred trace 
                %$\Potential_{\Box,\vartriangle,\ocircle; i, j} (u_{i, j, l}^{l})$, 
                
                %$(\Box,\vartriangle) = (\pm 1, 0)$, $(0, \pm 1)$} 
                \For {each direction $(\Box,\vartriangle)$ =  $(\pm 1, 0)$, $(0, \pm 1)$ }
                \State  Compute new transferred trace 
                %$\DD (u_{i, j}^{l}, \frac{\p u_{i, j}^{l}}{\p n_{\Box,\vartriangle}})^T \Big|_{\gamma_{\Box,\vartriangle; i, j}}$
                $\left( u_{i, j}^{l}, 
                        %\frac{\p u_{i, j}^{l}}{\p n_{\Box,\vartriangle}}
                        A_{\Omega_{i,j}} \nabla u_{i,j}^{l} \cdot  \n_{\Box,\vartriangle}
                        \right)^T \Big|_{\gamma_{\Box,\vartriangle; i, j}} $
                \parState{Find the smallest sweep number $l' \geq l$, such that the transferred trace could be used in Sweep $l'$, according to Rules \ref{rule2d_a} and \ref{rule2d_b}} 
                \parState {Add the transferred trace to the $l'$-th local trace of the corresponding  neighbor subdomain}
                \begin{eqnarray*}
                 \qquad\qquad\qquad \mathbf{g}^{l'}_{-\Box,-\vartriangle; i+\Box, j+\vartriangle} \gets
                  \mathbf{g}^{l'}_{-\Box,-\vartriangle; i+\Box, j+\vartriangle} +
                        \left( u_{i, j}^{l}, 
                        %\frac{\p u_{i, j}^{l}}{\p n_{\Box,\vartriangle}}
                        A_{i,j} \nabla u_{i,j}^{l} \cdot  \n_{\Box,\vartriangle}
                        \right)^T \Big|_{\gamma_{\Box,\vartriangle; i, j}} 
                \end{eqnarray*}
                \EndFor
                \EndFor
                \EndFor
                \EndFor
                \State The DDM solution for  $\P_{\Omega,\kappa}$ with the source $f$ is then given by
                $$ u_{\text{DDM}} = \sum\limits_{\substack{l = 1, \ldots, 4 }} %\sum\limits_{}
                \sum\limits_{\substack{\rangeTwoDRow}} u_{i,j}^{l} \mu_{i, j} $$
                
        \end{algorithmic}
\end{algorithm_}
\alglongline   
\vspace{0.2cm}

We note that the subdomain problem \eqref{eq:subd_solve} in Algorithm \ref{alg:diag2D_t} is solved with several potentials, however, in the practical numerical discretization, these potentials are not computed individually then summed up. Instead, the line integrals of the potentials for one subdomain are discretized as solutions to the subdomain problem with particular sources, as is done in the polarized trace method \cite{Zepeda2014}, then the sources are summed up as one RHS for the subdomain and solved, therefore the subdomain problem \eqref{eq:subd_solve} is solved only once in one step.

%%%%%%%%%%%%%%%%%%%%%%%%%%%%%%%%%%%%%%%%%%%%%%%%%%%%%%%%%%%%%%%%
\section{Convergence analysis}
%%%%%%%%%%%%%%%%%%%%%%%%%%%%%%%%%%%%%%%%%%%%%%%%%%%%%%%%%%%%%%%%
In this section we carry out  convergence analysis to the 
proposed diagonal sweeping DDM with  trace transfer in $\R^2$ (Algorithm \ref{alg:diag2D_t}) in both  the constant medium and two-layered media cases.

%%%%%%%%%%%%%%%%%%%%%%%%%%%%%%%%%%%%%%%%%%%%%%%%%%%%%%%%%%%%%%%%
\subsection{The constant medium case}
%%%%%%%%%%%%%%%%%%%%%%%%%%%%%%%%%%%%%%%%%%%%%%%%%%%%%%%%%%%%%%%%

We will show that the  DDM solution $u_{DDM}$ of Algorithm \ref{alg:diag2D_t}  gives exactly the solution $u$ of  \eqref{eq:PML} (the PML problem  $\P_{\Omega,\kappa}$  with the source $f$) in the constant medium  case. 
Note that in this case $\kappa_{i,j}\equiv \kappa$ and $\kappa_{i_1,i_2;j_1,j_2}\equiv \kappa$ where $\kappa$ is a constant function.
Let us start by assuming that the source lies in only one subdomain, i.e., $\supp(f) \subset \Omega_{i_0,j_0}$ for some $(i_0,j_0)$.  Define the regional solution $u_{i_1,i_2;j_1,j_2}$ as the solution to the PML problem $\P_{\Omega_{i_1,i_2;j_1,j_2},\kappa}$ with the source $f$, where $i_1 \leq i_0 \leq i_2$ and $j_1 \leq j_0 \leq j_2$, obviously, $u_{i_1,i_2;j_1,j_2} = u$ in $\Omega_{i_1,i_2;j_1,j_2}$.

%--------------------------------------------------------------------------------
%
%  On transfers
%
%--------------------------------------------------------------------------------

Let us check the trace transfers for the subdomain solution construction in the \textbf{first sweep}, i.e., the sweep in the direction of $(+1,+1)$. 
%The cardinal and corner directional trace transfers in the subdomain solution construction of the \textbf{first sweep} is considered. 
 The nonzero subdomain solve starts at step $i_0+j_0-1$, the subdomain problem $\P_{\Omega_{i_0,j_0},\kappa}$ is solved with the source $f$, and in the following steps, the subdomain problem $\P_{\Omega_{i,j},\kappa}$, $i \geq i_0$, $j \geq j_0$, is solved at step $ i + j -1$.
% Taking the first sweep for example, the construction of the total
% solution in each subdomain of region $\Omega_{i_0,\Nbx;j_0,\Nby}$ is
% considered using the cardinal and corner directional trace
% transfers.
% --------------------------------------------------------------------------------
%
% Cardinal directions  
For those subdomains to the right of $\Omega_{i_0,j_0}$, i.e., $\Omega_{i;j_0}$, $i=i_0+1,\ldots,\Nbx$, the subdomain solutions $u_{i,j_0}^1$ are obtained sequentially at one per step from left to right.  By applying the horizontal trace transfer repeatedly on the region $\Omega_{i_0,i;j_0,j_0}$, $i=i_0+1,\ldots,\Nbx$, one can easily see that for any $N_1\geq i' > i_0$,
\begin{align}
  %   u_{i_0,j_0}^{1}  \mu^{(1)}_{+1; i}
  % + u_{i',j_0}^{1} \mu^{(1)}_{-1; i'} 
  % + \sum\limits_{i = i_0+1, \ldots, i'-1} u_{i,j_0}^{1} \mu^{(1)}_{-1; i} \mu^{(1)}_{+1; i}
  % \sum\limits_{i = i_0, \ldots, i'} u_{i,j_0}^{1} \mu^{(1)}_{-1; i} \mu^{(1)}_{+1; i}
  %   =  u_{{i_0, i'; j_0,j_0}} . \label{eq:horz}  
  \sum\limits_{i = i_0, \ldots, i'} u_{i,j_0}^{1} \mu_{i,j_0}
    =  u_{{i_0, i'; j_0,j_0}}, \quad \text{in} \,\, \Omega_{i_0, i'; j_0,j_0}. \label{eq:horz}  
  \end{align} 
Similarly, for the subdomains to the top of $\Omega_{i_0,j_0}$, we have that for any $N_2\geq j' > j_0$,
\begin{align}
  % u_{i_0,j}^{1} \mu^{(2)}_{+1; j}
  % + u_{i_0,j'}^{1} \mu^{(2)}_{-1; j'} 
  % +  \sum\limits_{j = j_0+1, \ldots, j'-1} u_{i_0,j}^{1} \mu^{(2)}_{-1; j} \mu^{(2)}_{+1; j}
  % \sum\limits_{j = j_0, \ldots, j'} u_{i_0,j}^{1} \mu^{(2)}_{-1; j} \mu^{(2)}_{+1; j}
  %   =  u_{{i_0, i_0; j_0,j'}}. \label{eq:vert}
  \sum\limits_{j = j_0, \ldots, j'} u_{i_0,j}^{1} \mu_{i_0, j}
    =  u_{{i_0, i_0; j_0,j'}}, \quad \text{in} \,\, \Omega_{i_0, i_0; j_0,j'}. \label{eq:vert}
\end{align}

% --------------------------------------------------------------------------------
%   Corner directions

%For those subdomains on the upper right corner of
%  $\Omega_{i_0,j_0}$, i.e., $\Omega_{i,j}$, $i > i_0$, $j> j_0$, we
%  have the following lemma,

The following result holds for the solutions of the subdomains in the upper-right direction of the subdomain $\Omega_{i_0,j_0}$.
\begin{lemmaa} \label{lemma:quadr}
  For the constant medium problem, suppose the source satisfies $\supp(f) \subset \Omega_{i_0,j_0}$, 
        then for any $N_1\geq i'>i_0$ and $N_2\geq j'>j_0$, we have
        \begin{align} \label{eq:quadr}
          \sum\limits_{\substack{i = i_0, \ldots, i'\\ j = j_0, \ldots, j' }}
          u_{i,j}^{1} \mu_{i,j}
          = u_{{i_0, i'; j_0,j'}}
          \quad \text{in} \,\, \Omega_{i_0, i'; j_0,j'}. 
        \end{align}
\end{lemmaa}
\begin{proof}
  First we consider the subdomain $\Omega_{i_0+1,j_0+1}$, of which the local problem is solved at step $i_0+j_0+1$ in the first sweep, as shown in Figure \ref{fig:ddm_2x2t}.
  The local problem of its lower left neighbor subdomain $\Omega_{i_0,j_0}$ is solved two steps before, as shown in Figure \ref{fig:sf22_corner_trace}-(a).
  The local problems of its lower neighbor subdomain $\Omega_{i_0+1,j_0}$ and  left neighbor subdomain $\Omega_{i_0,j_0+1}$ are solved one step before, as shown in Figure \ref{fig:sf22_corner_trace}-(b) and (d), respectively.
  Consequently, for the region $\Omega_{i_0,i_0+1;j_0,j_0+1}$ with $2 \times 2$ subdomains, three subdomain problems have been solved and the solution within them has been constructed, while only the  subdomain problem of $\Omega_{i_0+1,j_0+1}$  is left to be solved.
  For convenience, let us define
  $\betaH_{\rightarrow} = \betaH_{+1;i_0}^{(1)}$,
  $\betaH_{\uparrow} = \betaH_{+1;j_0}^{(2)}$,
  $\betaH_{\leftarrow} = \betaH_{-1;i_0+1}^{(1)}$,
  $\betaH_{\downarrow} = \betaH_{-1;j_0+1}^{(2)}$,
    $\widehat{\Omega} = \Omega_{i_0,i_0+1;j_0,j_0+1}$, $\widehat{u} = u_{i_0,i_0+1;j_0,j_0+1}$,
  and omit the superscripts of the subdomain solutions which indicate the first sweep, i.e., $u_{i,j} = u_{i,j}^{1}$.

%\ifdraftFig
%
\def\wdff{0.4}
\begin{figure*}[!ht]    
        \centering
        \begin{minipage}[t]{\wdff\linewidth}
                \centering
                \includegraphics[width=0.9\textwidth]{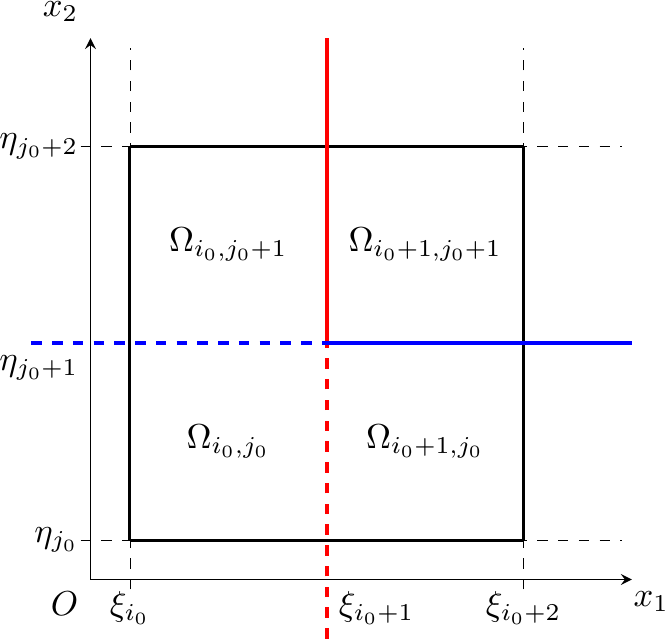}
        \end{minipage}
        \caption{
          The trace transfer for solving the subdomain problem of ${\Omega_{i_0+1,j_0+1}}$.
          The interface represented by the dotted and solid red lines is used for horizontal trace transfer on $\Omega_{i_0,i_0+1;j_0,j_0}$ at step $i_0+j_0$,
          the interface represented by the dotted and solid blue line is used for vertical trace transfer on $\Omega_{i_0,i_0;j_0,j_0+1}$ at step $i_0+j_0$,
          and the interface $\gamma_{\uprightarrow}$ represented by the solid blue and red lines is used for the trace transfer on $\Omega_{i_0,i_0+1;j_0,j_0+1}$ at step $i_0+j_0+1$.
                \label{fig:ddm_2x2t}}\vspace{0.25cm}
\end{figure*}
%\else
%\textcolor{red}{------------ Figs disabled ----------------}\\
%\fi % draftFig

%\ifdraftFig
%
\def\wdff{0.22}
\begin{figure*}[!ht]    
        \centering
        \begin{minipage}[t]{\wdff\linewidth}
                \centering
                \includegraphics[width=0.9\textwidth]{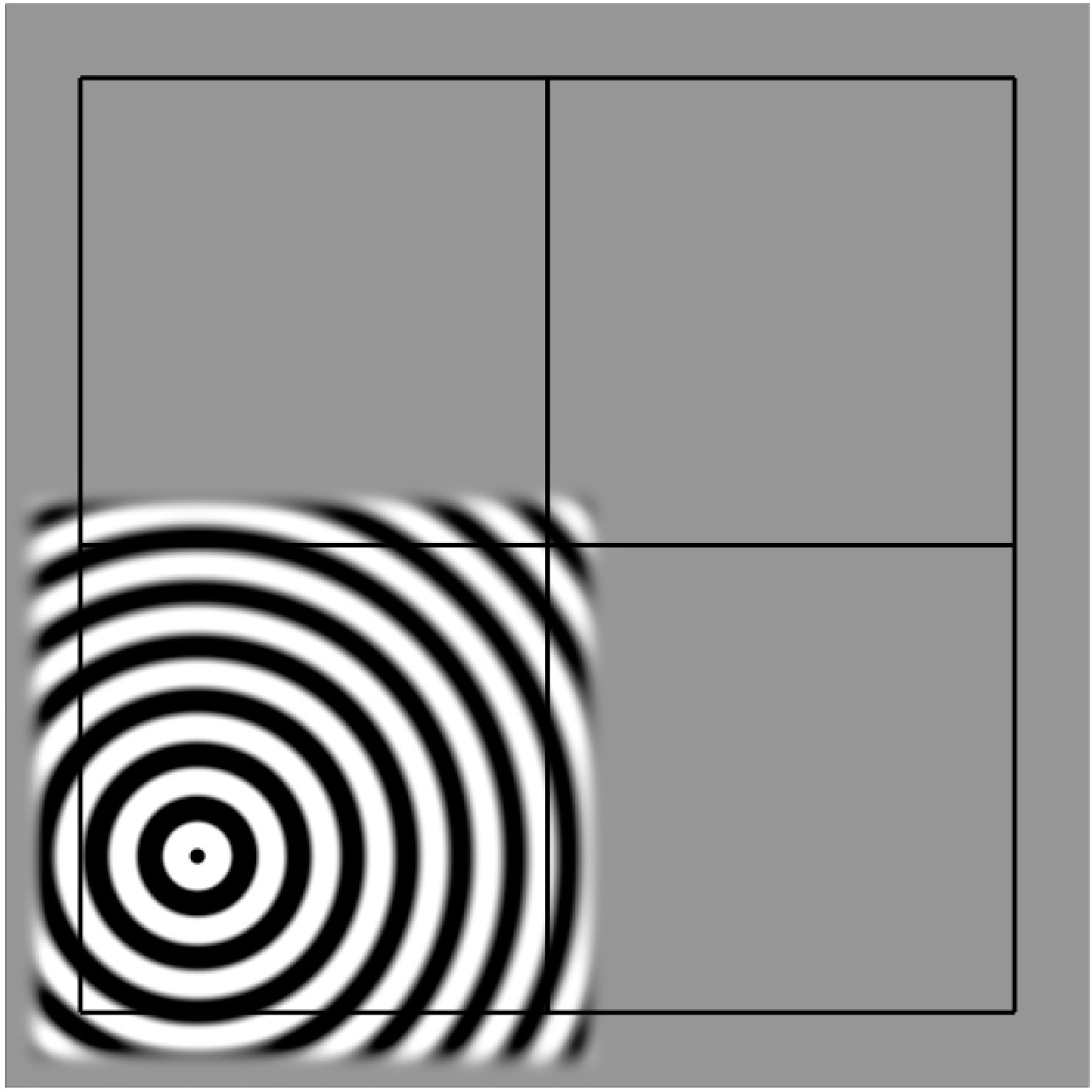}\\
                (a) $u_{i_0,j_0}$
        \end{minipage}
        \begin{minipage}[t]{\wdff\linewidth}
                \centering
                \includegraphics[width=0.9\textwidth]{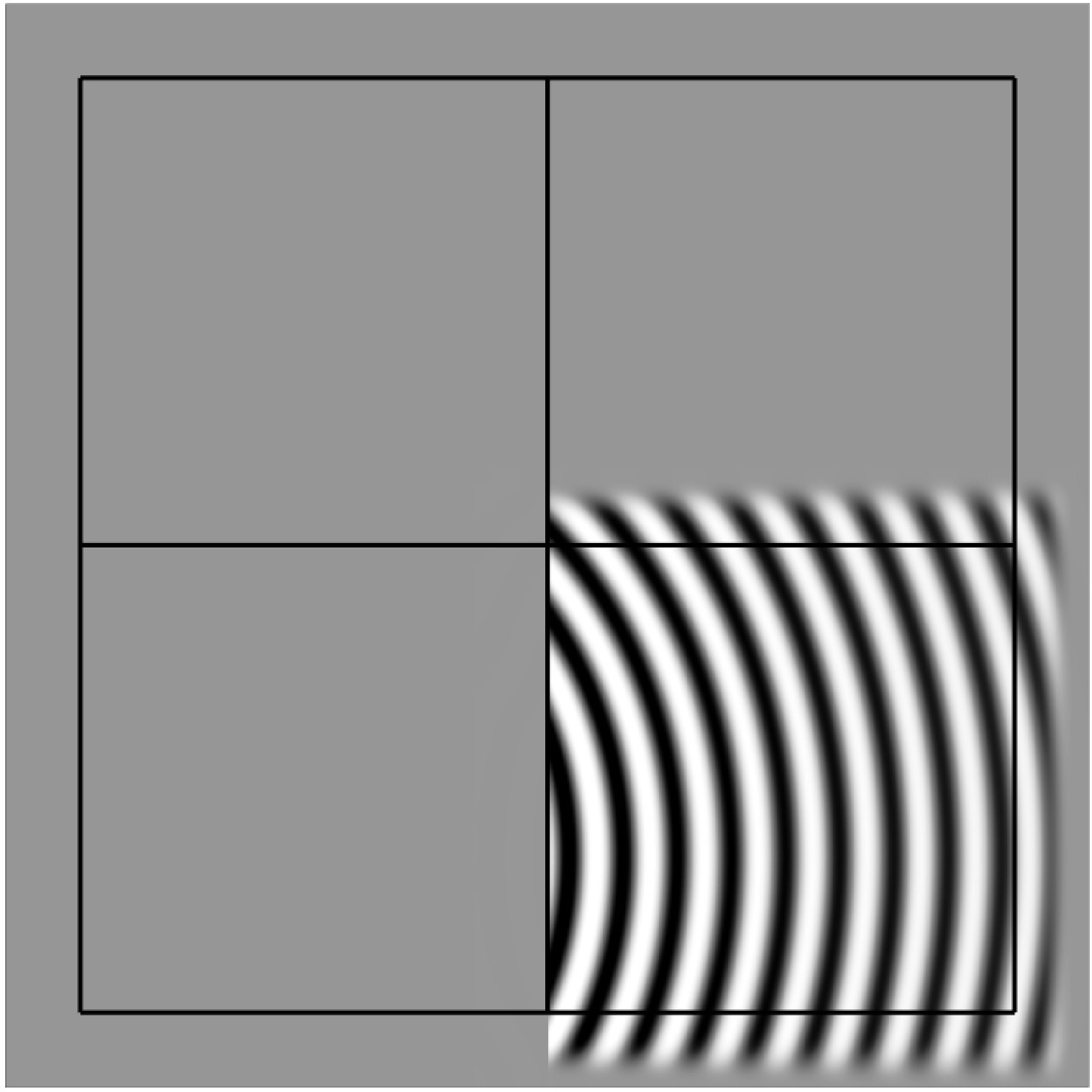}\\
                (b) $u_{i_0+1,j_0}$
        \end{minipage}
        \begin{minipage}[t]{\wdff\linewidth}
                \centering
                \includegraphics[width=0.9\textwidth]{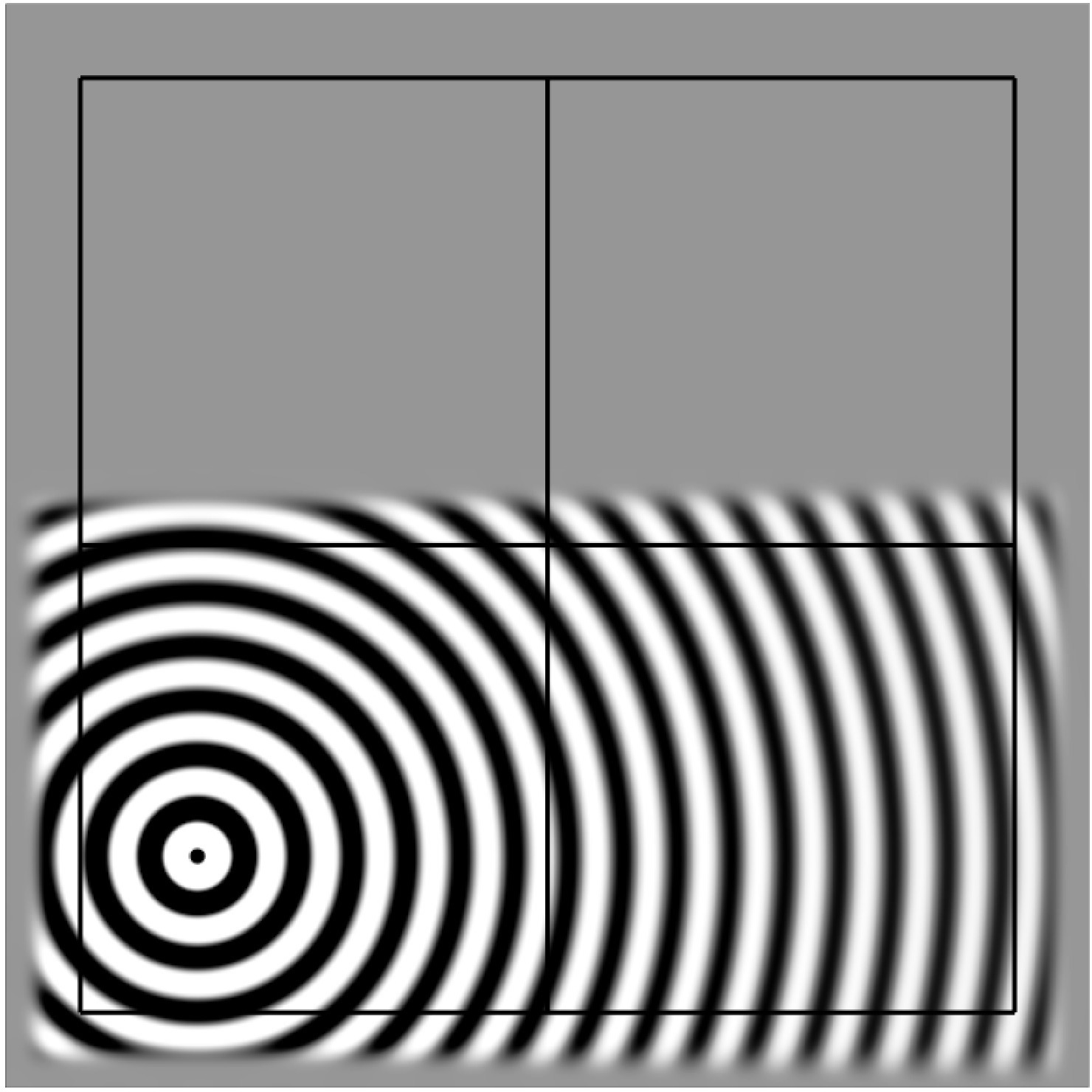}\\
                (c) $u_{i_0,j_0} \mu_{\rightarrow} + u_{i_0+1,j_0}$
        \end{minipage}
        \begin{minipage}[t]{\wdff\linewidth}    
                \centering
                \includegraphics[width=0.9\textwidth]{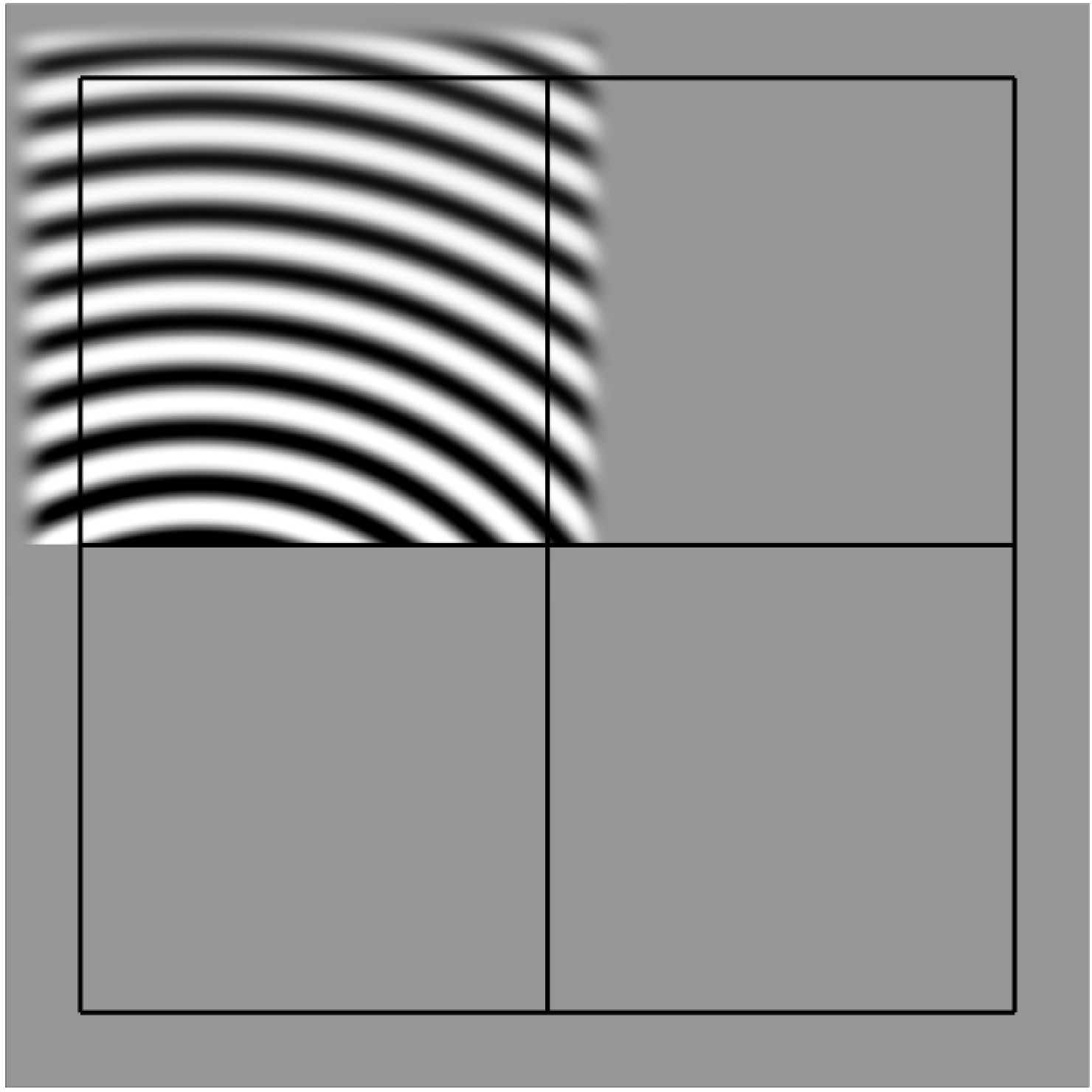}\\
                (d) $u_{i_0,j_0+1}$
        \end{minipage}  \\      
        \vspace{0.2cm}
        
        \begin{minipage}[t]{\wdff\linewidth}
                \centering
                \includegraphics[width=0.9\textwidth]{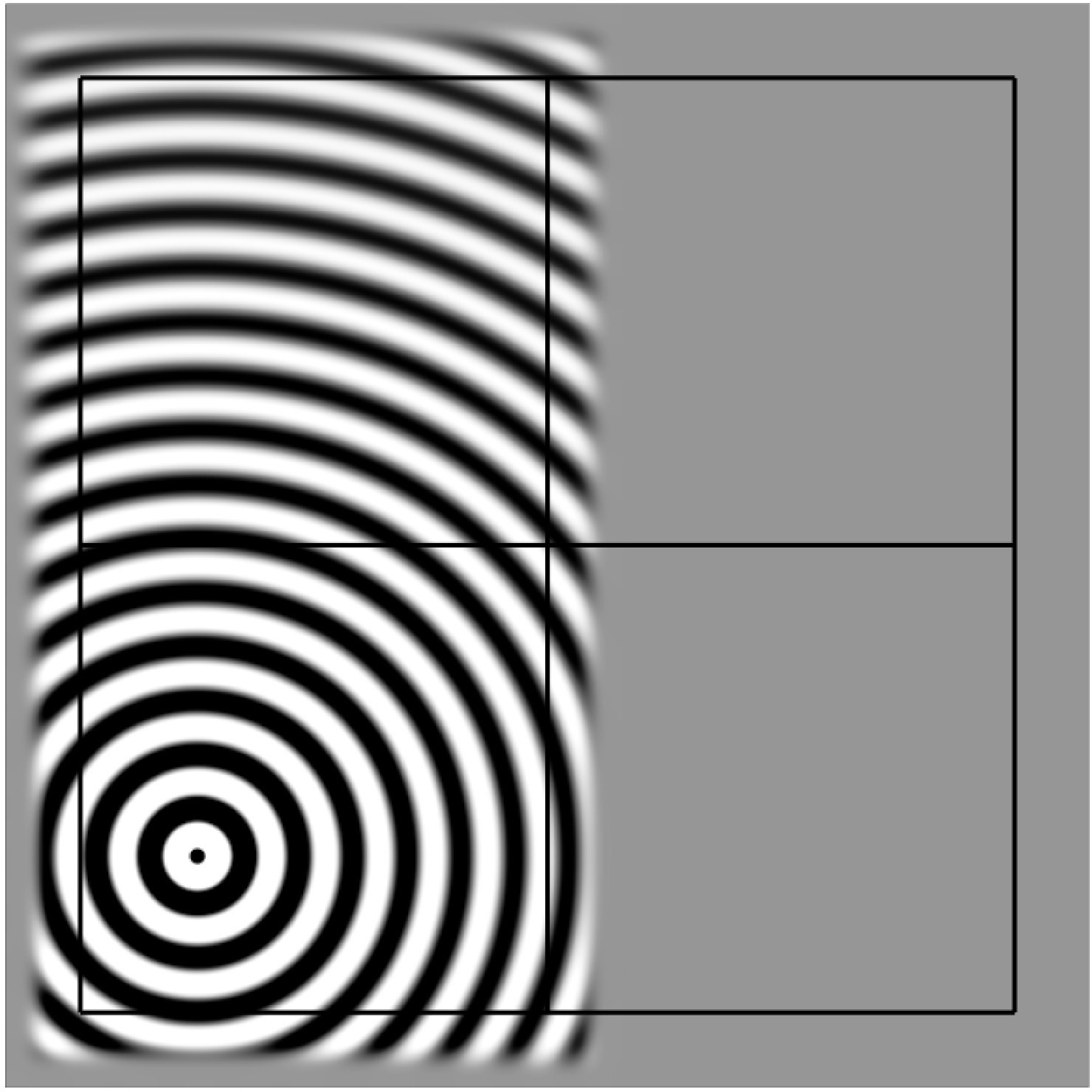}\\
                (e) $u_{i_0,j_0} \mu_{\uparrow} + u_{i_0,j_0+1}$
        \end{minipage}
        \begin{minipage}[t]{\wdff\linewidth}
                \centering
                \includegraphics[width=0.9\textwidth]{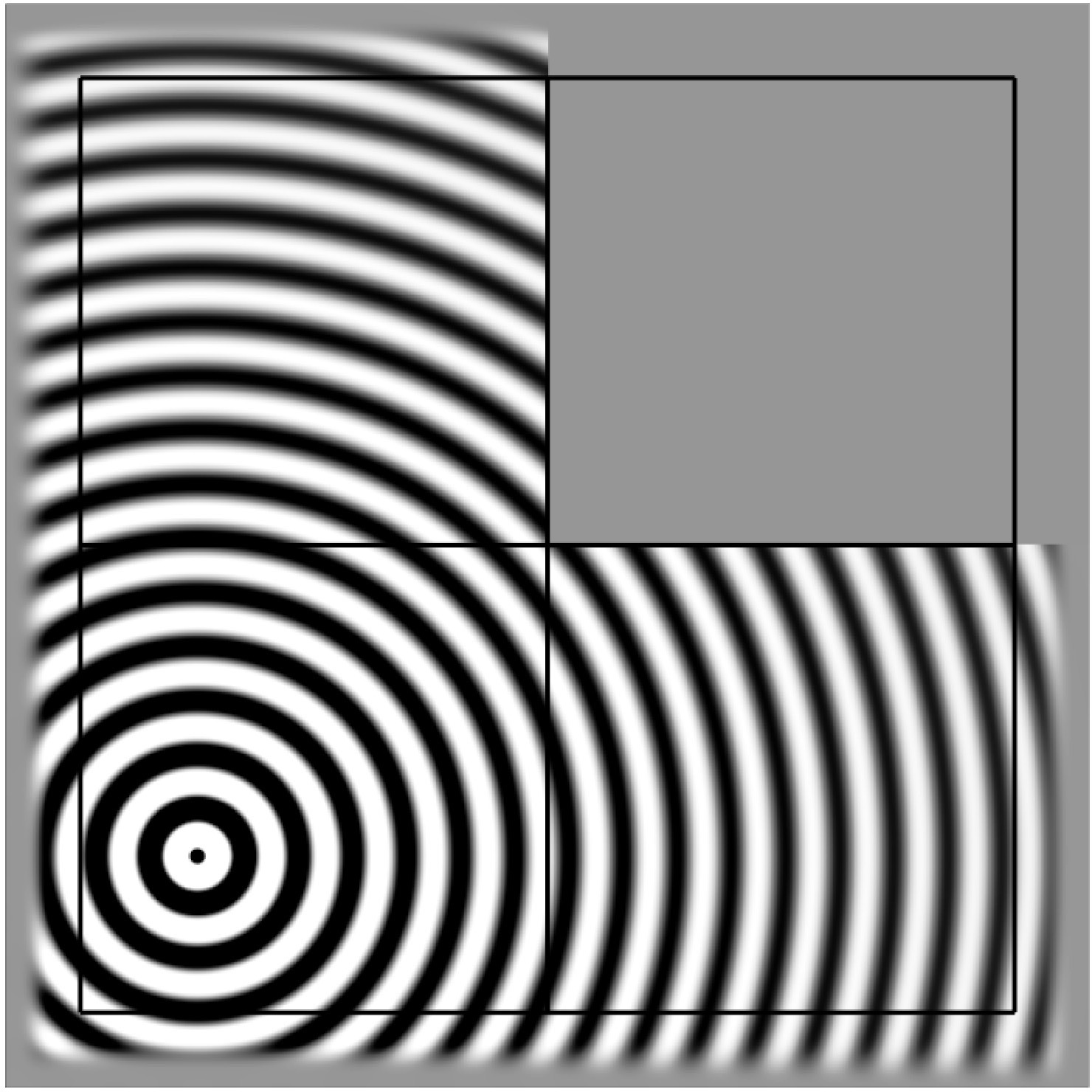}\\
                (f) $\widehat{u} {\mu}_{\uprightarrow} $
        \end{minipage}
        \begin{minipage}[t]{\wdff\linewidth}    
                \centering
                \includegraphics[width=0.9\textwidth]{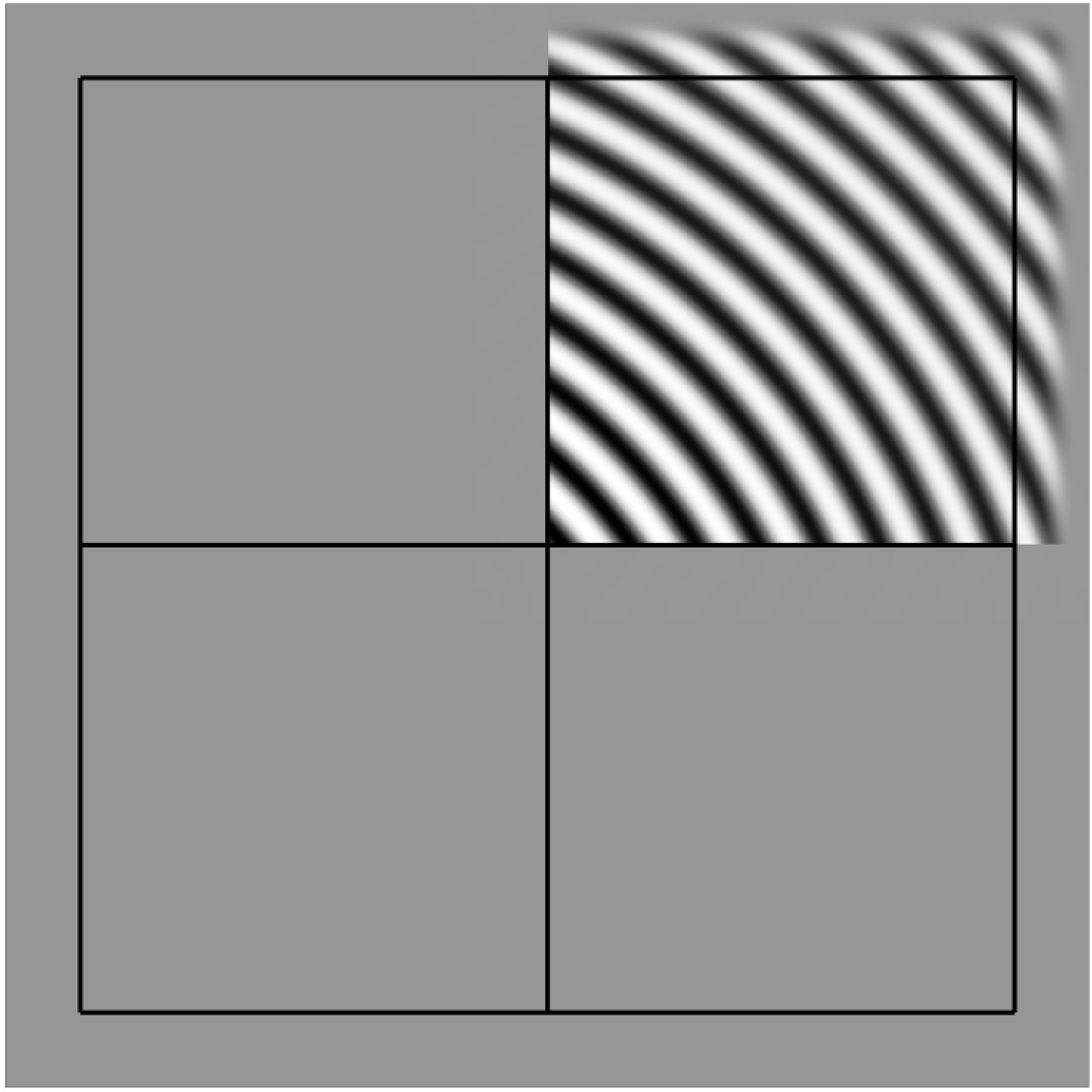}\\
                (g) $ u_{i_0+1,j_0+1}$ 
        \end{minipage}
        \begin{minipage}[t]{\wdff\linewidth}
                \centering
                \includegraphics[width=0.9\textwidth]{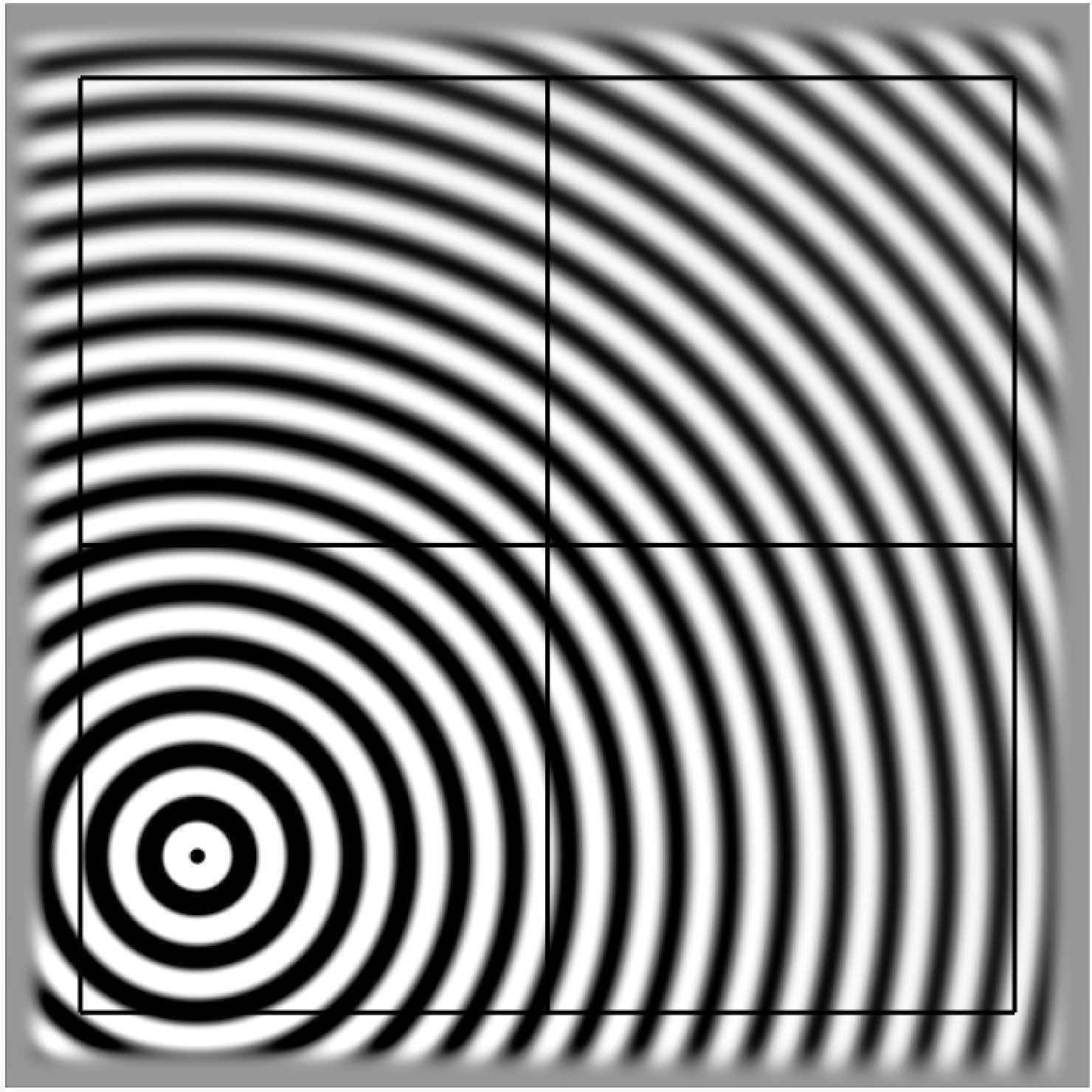}\\
                (h) $\widehat{u}$
        \end{minipage}          
        \vspace{-0.0cm}
        \caption{{Illustration of the process of trace transfer for $2\times 2$ subdomains. } \label{fig:sf22_corner_trace}}
\end{figure*}
%\else
%\textcolor{red}{------------ Figs disabled ----------------}\\
%\fi % draftFig

By applying the horizontal trace transfer on $\Omega_{i_0,i_0+1;j_0,j_0}$, we know that by Lemma \ref{lemma:change_domain},
\begin{equation} \label{eq:u2x2lower_t}
        u_{i_0,j_0}\betaH_{\rightarrow} + u_{i_0+1,j_0} \betaH_{\leftarrow} = \widehat{u}, \qquad \qquad \text{in} 
        \quad (-\infty,+\infty)\times(-\infty,\eta_{j_0+1}\plusd)
\end{equation}
as shown in Figure \ref{fig:sf22_corner_trace}-(c). 
Similarly, by applying the vertical trace transfer on $\Omega_{i_0,i_0;j_0,j_0+1}$,  we know that by Lemma \ref{lemma:change_domain},
\begin{equation} \label{eq:u2x2left_t}
        u_{i_0,j_0}\betaH_{\uparrow} + u_{i_0,j_0+1}\betaH_{\downarrow} = \widehat{u}, \qquad \qquad \text{in} 
        \qquad (-\infty, \xi_{i_0+1}\plusd)\times(-\infty,+\infty)
\end{equation}
as shown in Figure \ref{fig:sf22_corner_trace}-(e).

Base on \eqref{eq:subd_solve}, the subdomain problem of $\Omega_{i_0+1,j_0+1}$ is solved with the upward transferred trace
 %$u_{i_0+1,j_0}$ on $\gamma_{\uparrow}$ 
from  $\Omega_{i_0+1,j_0}$, and the rightward transferred trace 
 %$u_{i_0,j_0+1}$ on $\gamma_{\rightarrow}$ 
from  $\Omega_{i_0,j_0+1}$.
By using \eqref{eq:u2x2lower_t}, \eqref{eq:u2x2left_t} and the fact that $u_{i_0+1,j_0} = 0$ on $(-\infty,\xi_{i_0+1} \plusd) \times \{\eta_{j_0+1} \plusd\}$, and $u_{i_0,j_0+1} = 0$ on $ \{\xi_{i_0+1} \plusd\} \times (-\infty,\eta_{j_0+1} \plusd)$, we see that
\begin{align*}
  u_{i_0+1,j_0+1} = \;& \Potential_{0,-1;i_0+1,j_0+1}(u_{i_0+1,j_0}) + \Potential_{-1,0;i_0+1,j_0+1}(u_{i_0,j_0+1})  \\
  = & \int_{\gamma_{\uprightarrow}}
  J_{\Omega_{i_0+1,j_0+1}}^{-1}G_{i_0+1,j_0+1}(\bx,\by)\big(A_{\Omega_{i_0+1,j_0+1}}\nabla_{\by} \widehat{u}(\by)\cdot \n_{\uprightarrow} \big) \\
   &\qquad-\widehat{u}\Big(A_{\Omega_{i_0+1,j_0+1}}\nabla_{\by}\big(J_{\Omega_{i_0+1,j_0+1}}^{-1}G_{i_0+1,j_0+1}(\bx,\by)\big)\cdot \n_{\uprightarrow}  \Big)d\by,
\end{align*}
where $\gamma_{\uprightarrow} = \{(x_1,x_2) \;|\; x_1=\xi_{i_0+1}\plusd, x_2\geq\eta_{j_0+1}\plusd \,\,\text{or}\,\, x_1\geq\xi_{i_0+1}\plusd, x_2=\eta_{j_0+1}\plusd\}$, and $\n_{\uprightarrow}$ is the unit normal of $\gamma_{\uprightarrow}$ pointing to $\Omega_{i_0+1,j+0+1}$,  as shown in Figure \ref{fig:ddm_2x2t}.

By Lemma \ref{lemma:change_domain} and applying the trace transfer on $\widehat{\Omega}$, we know that
\begin{align} \label{eq:u2x2corner_t}
  \widehat{u} {\betaH}_{\uprightarrow} + u_{i_0+1,j_0+1} \betaH_{\leftarrow} \betaH_{\downarrow}= \widehat{u},
  \end{align}
where ${\betaH}_{\uprightarrow} = 1 -
(1-\betaH_{\rightarrow})(1-\betaH_{\uparrow})$,
as is  shown in Figure \ref{fig:sf22_corner_trace}-(f) to (h).
{\em Note that this also indicates that the corner directional trace transfer is implicitly carried out by combining horizontal and vertical trace transfers, which is 
the essential difference with the source transfer-based method used in \cite{Leng2020}}.
From \eqref{eq:u2x2corner_t} we know that
\begin{equation}
u_{i_0,j_0}\betaH_{\rightarrow}\betaH_{\uparrow} +
u_{i_0+1,j_0}\betaH_{\leftarrow}\betaH_{\uparrow} + u_{i_0,j_0+1} \betaH_{\rightarrow}\betaH_{\downarrow} +
u_{i_0+1,j_0+1}\betaH_{\leftarrow}\betaH_{\downarrow} = \widehat{u}.
\end{equation}
Thus we obtain that \eqref{eq:quadr} holds for the case $i' = i_0+1$
and $j'=j_0+1$.  After solving the subdomain problem $\Omega_{i_0+1,j_0+1}$, the upward and rightward transferred traces
are then generated from the subdomain $\Omega_{i_0+1,j_0+1}$, while
both the leftward and downward transferred traces are zero.

Next we will prove by induction that for any $s > 0$, \eqref{eq:quadr} holds for $i'$ and $j'$ that $i' + j' \leq i_0 + j_0 + s$, $N_1\geq i' \geq i_0$, $N_2\geq j' \geq j_0$.
The cases for $s = 1,2$ has been proved in the above, now assume \eqref{eq:quadr} holds for $s \leq t$ , we will prove it also holds for $s = t+1$.
For the subdomain $\Omega_{i',j'}$ that $i' + j' = i_0 + j_0 + t + 1$, $i' \geq i_0$, $j' \geq j_0$, we can divide the region $\Omega_{i_0,i';j_0,j'}$ into four regions, namely $\Omega_{i_0,i'-1;j_0,j'-1}$, $\Omega_{i_0,i'-1;j',j'}$, $\Omega_{i',i';j_0,j'-1}$ and $\Omega_{i',j'}$, apply a similar argument for the $2\times 2$ subregions as the above for  $\Omega_{i_0,i_0+1;j_0,j_0+1}$, and we can obtain
\begin{align*}
  u_{i_0,i';j_0,j'} =\;& 
  U_{i_0,i'-1;j_0,j'-1}\betaH_{+1;i'-1}^{(1)} \betaH_{+1;j'-1}^{(2)} 
  + U_{i',i';j_0,j'-1}  \betaH_{-1;i'}^{(1)} \betaH_{+1;j'-1}^{(2)} \\
  &+ U_{i_0,i'-1;j',j'} \betaH_{+1;i'-1}^{(1)} \betaH_{-1;j'}^{(2)} 
    + u_{i',j'} \betaH_{-1;j'}^{(1)} \betaH_{-1;j'}^{(2)},   
\end{align*}
where $U_{i_1, i_2; j_1, j_2} = \sum_{\substack{i_1 \leq i \leq i_2 \\ j_1 \leq j \leq j_2  }} u_{i, j} \mu_{i,j} $,
which completes the proof.
\end{proof}

%--------------------------------------------------------------------------------
%
%
%  On Sweeps
%
%
%--------------------------------------------------------------------------------

The application of cardinal trace transfers have been demonstrated in the construction of the subdomain solutions in the first sweep through Lemma
\ref{lemma:quadr}, and it is similar for all remaining sweeps.
However, we still  need to ensure that each subdomain problem is solved with the right transferred traces passed from its neighbor subdomains, and the following result
can be obtained.

\begin{lemmaa} \label{lemma:single_src}
  Suppose that $\supp(f) \subset \Omega_{i_0,j_0}$, then the DDM solution $u_{\text{DDM}}$ of Algorithm \ref{alg:diag2D_t} is indeed the exact solution to the problem $\P_{\Omega,\kappa}$ with the source $f$ in the constant medium case.
\end{lemmaa}

\begin{proof}
        
%The verification that the Algorithm \ref{alg:diag2D_t} produces the
%exact solution for the constant medium case

The proof  is similar to the verification for the source transfer method in \cite{Leng2020}, and the main difference is that  the transferred traces only come from neighbor subdomain in cardinal directions at each step in all sweeps, while the transferred sources come  from neighbor subdomains in both cardinal and corner directions.
%The case of the source lying in only one subdomain is verified, and the general
%case follows since the solving procedure of the algorithm does not
%depend on any specific subdomain.
%
Without loss of generality, we use a $5\times 5$ ($N_1=N_2=5$) domain partition and a source lying in $\Omega_{3,3}$ ($i_0=j_0=3$) for verification, and the movement of transferred traces in each sweep is carefully checked in the following.

%\ifdraftFig
\def\wdff{0.3}
\begin{figure*}[!ht]
        \centering
        \begin{minipage}[t]{\wdff\linewidth}
                \centering
                \includegraphics[width=0.9\textwidth]{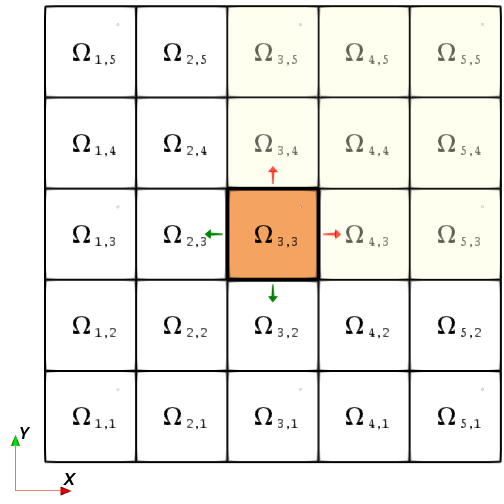}\\
                (a) First sweep: step 5 
        \end{minipage}
        \begin{minipage}[t]{\wdff\linewidth}
                \centering
                \includegraphics[width=0.9\textwidth]{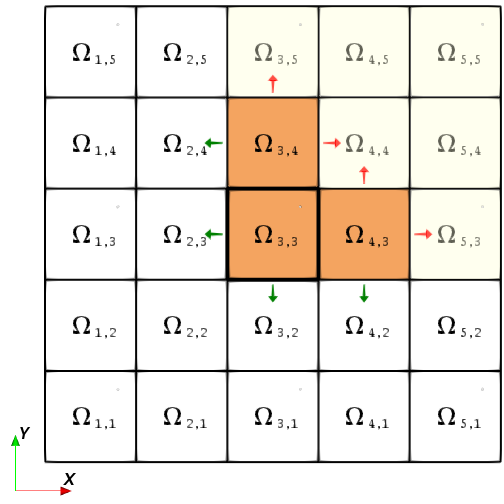}\\
                (b) First sweep: step 6
        \end{minipage}
        \begin{minipage}[t]{\wdff\linewidth}
                \centering
                \includegraphics[width=0.9\textwidth]{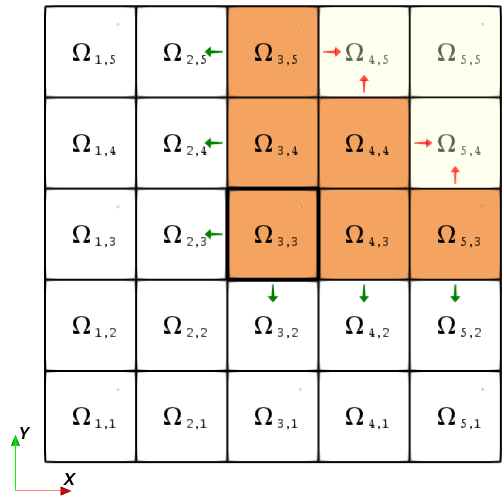}\\
                (c) First sweep: step 7 
        \end{minipage}
        
        \vspace{0.25cm}
        \begin{minipage}[t]{\wdff\linewidth}
                \centering
                \includegraphics[width=0.9\textwidth]{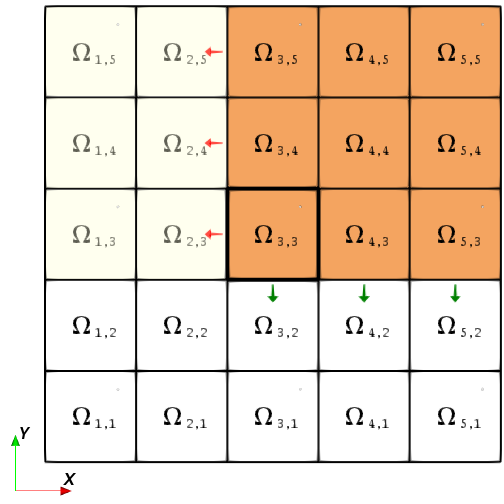}\\
                (d) After first sweep 
        \end{minipage}
        \begin{minipage}[t]{\wdff\linewidth}
                \centering
                \includegraphics[width=0.9\textwidth]{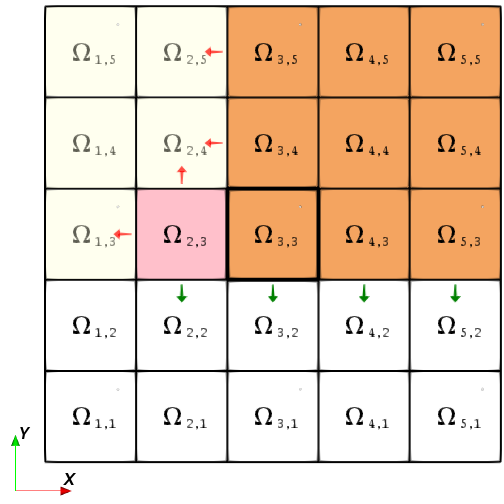}\\
                (e) Second sweep: step 6 
        \end{minipage}
        \begin{minipage}[t]{\wdff\linewidth}    
                \centering
                \includegraphics[width=0.9\textwidth]{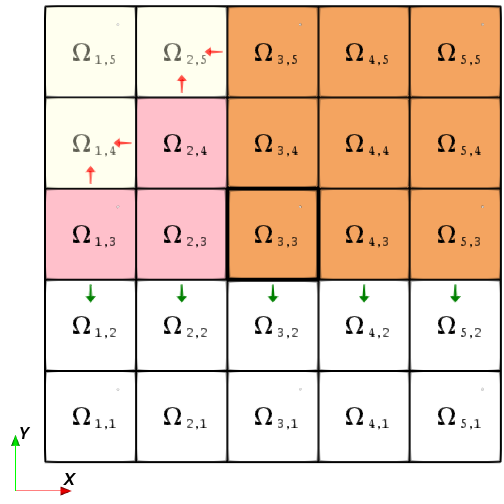}\\
                (f) Second sweep: step 7 
        \end{minipage}
        \vspace{-0.0cm}
        \caption{The first sweep $(+1,+1)$ and the second sweep $(-1,+1)$ of the trace transfer-based diagonal sweeping DDM in $\R^2$, where the source  lies in $\Omega_{3,3}$. 
                The arrows denote the  transferred traces with their directions, the red ones are in the similar direction of the current sweep,
                while the green ones are not, thus only the red ones are used in the current sweep due to   Rule \ref{rule2d_a} (similar directions).  
                 \label{fig:sweep2d_12}}        
\end{figure*}
%\fi % draftFig 

%
In the {\bf first sweep} of direction $(+1,+1)$, the solution in the upper-right region $\Omega_{i_0,N_1;j_0,N_2}=\Omega_{3,5;3,5}$ is to be constructed.
% and we focus on the subdomains in each of the sweep steps.
The sweep contains $N_1+N_2-1 = 5+5-1=9$ steps, and the group of subdomain problems of $\{\Omega_{i,j}\}$ with $(i-1)+(j-1)+1=s$  are solved at the $s$-th step of the sweep.
The subdomain solutions are all zeros in the first $(i_0-1)+(j_0-1)=4$ steps since the local sources are all zeros.
Then at step $(i_0-1)+(j_0-1)+1=5$, the subdomain problem of $\Omega_{i_0,j_0}=\Omega_{3,3}$ is solved with the source $f_{3,3}=f$, and four transferred source in cardinal directions are generated, as shown in Figure  \ref{fig:sweep2d_12}-(a).
At step 6, the subdomain problems in $\Omega_{4,3}$ and $\Omega_{3,4}$ are solved,
% and the horizontal and vertical source transfer applies, where one
% transferred trace are used as source and
each of them takes one transferred trace from $\Omega_{3,3}$, and generates three new transferred traces, as shown in \ref{fig:sweep2d_12}-(b).
At step 7, the subdomain problems of $\Omega_{5,3}$ and $\Omega_{3,5}$ are solved in the similar way to the previous step, while the subdomain problem of $\Omega_{4,4}$ is solved with the transferred traces from $\Omega_{3,4}$ and $\Omega_{4,3}$,
% and the corner directional source transfer applies, where
% and two transferred traces are summed as source 
and two new ones are generated, as shown in \ref{fig:sweep2d_12}-(c).
In the following steps, the subdomain problems are solved similarly, and after $\Nbx+\Nby-1=9$ steps the exact solution in the upper-right region is obtained.

%
% Since the sweeping procedure is similar for all sweeps, the details of
% the sweep steps are skipped for the following steps, and the focus is
% changed to whether the needed transferred sources are correctly
% selected for each of the octants.
%

In the {\bf second sweep} of direction $(-1,+1)$, the solution in the upper-left region $\Omega_{1,i_0-1;j_0,N_2}={\Omega_{1,2;\, 3,5}}$ is to be constructed, and the group of subdomain problems of $\{\Omega_{i,j}\}$ with $(N_1-i)+(j-1)+1=5-i+j=s$ are solved at the $s$-th step of the sweep.  Among all the unused transferred traces from the first sweep, the ones in the similar direction with this sweep are needed while the others are not (Rule \ref{rule2d_a}), as shown in Figure \ref{fig:sweep2d_12}-(d).  Note that  Rule \ref{rule2d_b} doesn't apply here since the first and second sweeps are not in the opposite direction.  The steps of this sweep are similar to the steps of the first sweep, hence the details are omitted. Steps 6 and 7 are shown in Figure \ref{fig:sweep2d_12}-(e) and (f), and after $\Nbx+\Nby-1=9$ steps the exact solution in the upper-left region is obtained.

%\ifdraftFig
\def\wdff{0.3}
\begin{figure*}[ht!]
  \centering
  \begin{minipage}[t]{\wdff\linewidth}
    \centering
    \includegraphics[width=0.9\textwidth]{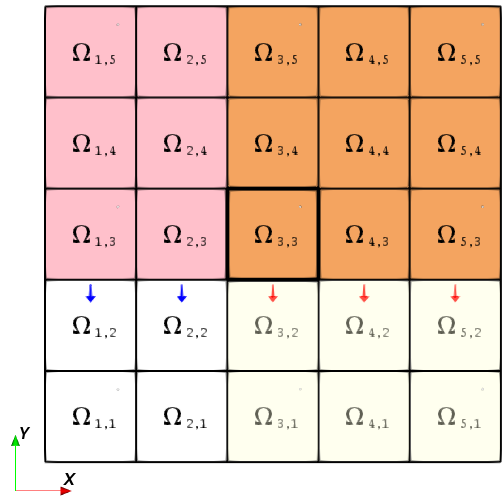}\\
    (a) Before third sweep 
  \end{minipage}
  \begin{minipage}[t]{\wdff\linewidth}
    \centering
    \includegraphics[width=0.9\textwidth]{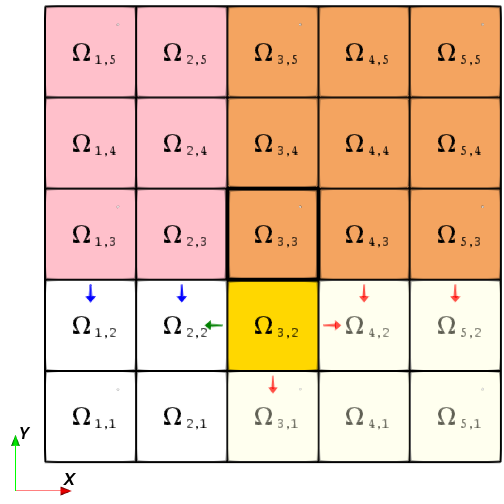}\\
    (b) Third sweep: step 6 
  \end{minipage}
  \begin{minipage}[t]{\wdff\linewidth}
    \centering
    \includegraphics[width=0.9\textwidth]{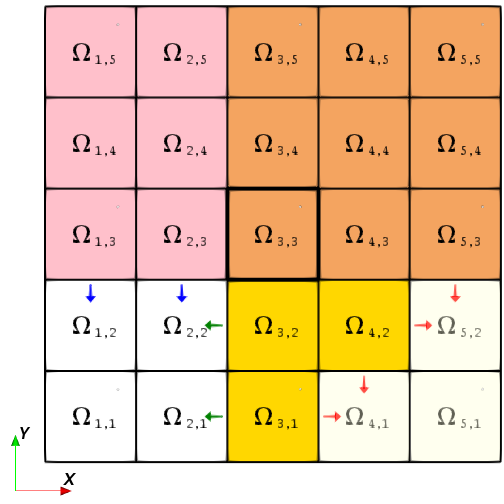}\\
    (c) Third sweep: step 7 
  \end{minipage}
  
  \vspace{0.25cm}
  \begin{minipage}[t]{\wdff\linewidth}
    \centering
    \includegraphics[width=0.9\textwidth]{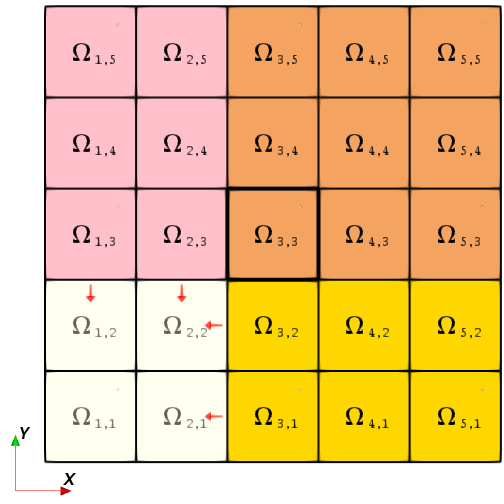}\\
    (d) After third sweep 
  \end{minipage}
  \begin{minipage}[t]{\wdff\linewidth}
    \centering
    \includegraphics[width=0.9\textwidth]{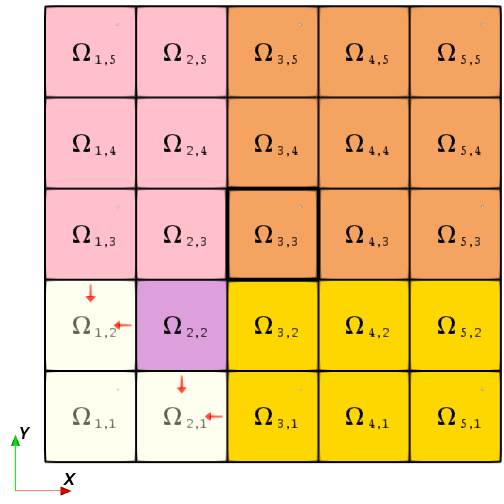}\\
    (e) Fourth sweep:  step 7 
  \end{minipage}
  \begin{minipage}[t]{\wdff\linewidth}    
    \centering
    \includegraphics[width=0.9\textwidth]{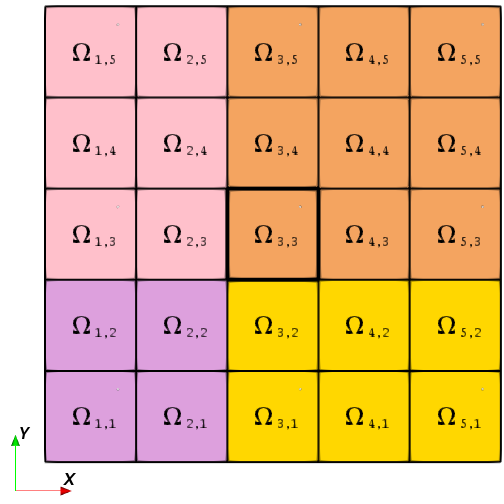}\\
    (f) After fourth sweep 
  \end{minipage}
  \vspace{-0.0cm}
  \caption{ The third sweep $(+1,-1)$ and the fourth sweep $(-1,-1)$ of the  trace transfer-based diagonal sweeping DDM in $\R^2$, where the source  lies in $\Omega_{3,3}$. The arrows denote the transferred traces with their directions, the red ones are in the similar direction to the current sweep and not against  Rule \ref{rule2d_b} (opposite directions), the blue ones are in the similar direction but against Rule \ref{rule2d_b}, and the green ones are not in the similar direction, thus only the red ones are used in the current sweep.
  \label{fig:sweep2d_34} }       
\end{figure*}
%\fi % draftFig 

In the {\bf third sweep} of direction $(+1,-1)$, the solution in the lower-right region $\Omega_{i_0,N_1;1,j_0-1}={\Omega_{3,5; 1,2 }}$ is to be constructed.
%and the local problems of subdomains in the region $\{\Omega_{i,j}\}$ with $(i-1)+(N_2-j)+1=5+i-j=s$ are solved at the $s$-th step of the sweep.
as shown in Figure \ref{fig:sweep2d_34}-(a),
among all the unused transferred traces from the previous two sweeps, the ones from the upper-right region are needed since they satisfy both Rules \ref{rule2d_a} and \ref{rule2d_b}, while the ones from the upper-left region are excluded by Rule \ref{rule2d_b}, since they are from the second sweep which is in the opposite direction of the current sweep.
The steps of this sweep are similar to the steps of previous sweeps, and steps 6 and 7 are shown in Figure \ref{fig:sweep2d_34}-(b) and (c). After $\Nbx+\Nby-1=9$ steps the exact solution in the lower-right region is obtained.

In the {\bf fourth sweep} of direction $(-1,-1)$, the solution in the lower-left region $\Omega_{1,i_0-1;1,j_0-1}={\Omega_{1,2; 1,2 }}$ is to be constructed.
% and the local problems of subdomains in the region $\{\Omega_{i,j}\}$ with $(N_1-i)+(N_2-j)+1=9-i-j=s$ are solved at the $s$-th step of the sweep.
All the unused transferred traces from the previous sweeps are needed, as shown in Figure \ref{fig:sweep2d_34}-(d).  The reason is that they are all 
from the second and third sweeps (thus Rule \ref{rule2d_b} does not apply to them) and satisfy  Rule \ref{rule2d_a}.
The steps of the sweep are again similar to the steps of the previous sweeps, and  step 7 of the sweep are shown in Figure \ref{fig:sweep2d_34}-(e). After $\Nbx+\Nby-1=9$ steps, the exact solution in the whole region $\Omega$  is finally obtained, as shown in Figure \ref{fig:sweep2d_34}-(f).
\end{proof}

For the case of  general source $f$, clearly 
$u(f) = \sum_{i,j} u(f_{i,j})$, where $u(g)$ denotes the solution to the problem $\P_{\Omega,\kappa}$ with the source $g$, 
%and the solving procedure of the algorithm does not depend on any specific subdomain.
thus  we have the following theorem based on Lemma \ref{lemma:single_src}:
\begin{thm} \label{thm:sweep2d}
For a bounded smooth source $f$,  the DDM solution $u_{\text{DDM}}$ of Algorithm  \ref{alg:diag2D_t} is  the exact solution to the PML problem $\P_{\Omega,\kappa}$ with the source $f$ in the constant medium case. 
\end{thm}
%\begin{proof}
%        
%\end{proof}

%%%%%%%%%%%%%%%%%%%%%%%%%%%%%%%%%%%%%%%%%%%%%%%%%%%%%%%%%%%%%%%%
\subsection{The two-layered media case}
%%%%%%%%%%%%%%%%%%%%%%%%%%%%%%%%%%%%%%%%%%%%%%%%%%%%%%%%%%%%%%%%

The layered media problems  play an important role in many applications involving the Helmholtz equation, for example, the seismic wave traveling in the layered structure is extensively studied for imaging the earth's interior. 
 Although the result of  Theorem \ref{thm:sweep2d}  in general does not hold for DDMs in the inhomogeneous media case,  
 the diagonal sweeping DDM handle the layered media more properly than  the L-sweeps method  \cite{Zepeda2020} and such advantage has been numerically  demonstrated   in \cite{Leng2020}. Part of the reason is that
the forward sweep for the incident and the backward sweep for the reflection could be accomplished in one iteration by the diagonal sweeping in the layered media case, while the L-shaped sweeping only deals the forward sweep for the incident within one iteration. 
Here we carry out rigorous analysis  on the trace transfer-based diagonal sweeping DDM in the two-layered media case and 
 show that some nice results still can be achieved  for Algorithm  \ref{alg:diag2D_t} under relaxed conditions.
Let us assume that  the media has an interface $\gamma_L$, which could be either horizontal or vertical. In the case that the media   interface is horizontal, it could be defined as $\gamma_L=\{(x_1,x_2)\;|\; x_2 = \eta_{L}\}$, and the wave number $\kappa$ satisfies 
\begin{align} \label{eq:2L-medium}
\kappa(x_1,x_2) &= \left\{
\begin{array}{ll}
\kup,  & \, x_2 > \eta_L,\\
\kdown,  & \, x_2 \leq \eta_L,
\end{array}
\right. 
\end{align}
for two  constants $\kup\ne\kdown$, as shown in Figure \ref{fig:2layer}-(left) where  $\eta_{L}> 0$ is assumed for illustration. 
%where $x_2 = \eta_L$ is the interface between two media. 
The well-posedness and convergence of  the resulting global PML problem in two-layered media case have been proved in \cite{Chen2010}. In the following, the  one-dimensional strip domain partition is first studied  to establish some basic properties of the trace transfer in the 
two-layered media case, then the  general checkerboard partition is considered based on them.
%
% Note that we always assume the medium
% interface %$(-\infty, \infty) \times \{\eta_L\}$
% is away from any of the subdomain interfaces, i.e.,
% $|\eta_L - \eta_{j}| \geq C$, $j = 1, \ldots, \Nby+1$, where $C$ is a
% positive constant.

%We can similarly prove that
%\begin{itemize}
%       \item If the source lies below the media interface, the DDM solution of the  Algorithm  \ref{alg:diag2D_t} is the exact solution.
%       \item If the source lies above the media interface, the DDM solution of two iterations of the Algorithm  \ref{alg:diag2D_t} is the exact solution.
%\end{itemize}

\def\wdff{0.3}
\begin{figure*}[!ht]    
        \centering
        \begin{minipage}[t]{\wdff\linewidth}
                \centering
                \includegraphics[width=0.9\textwidth]{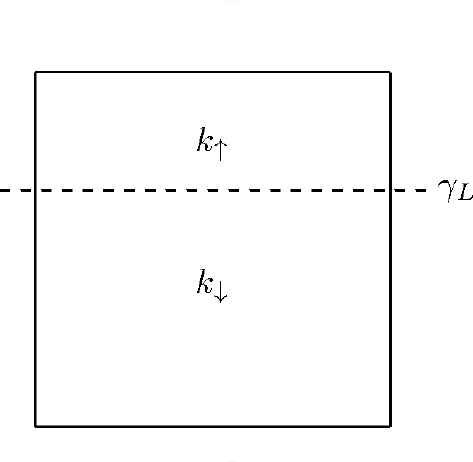}\\
                %(a) Medium
        \end{minipage}
        \begin{minipage}[t]{\wdff\linewidth}
                \centering
                \includegraphics[width=0.9\textwidth]{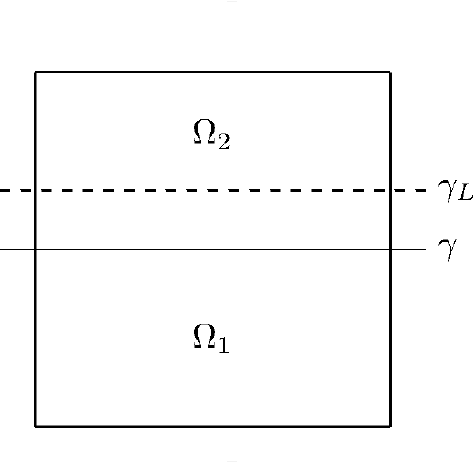}\\
                %(b) $1 \times 2$ partition
        \end{minipage}
        \begin{minipage}[t]{\wdff\linewidth}
                \centering
                \includegraphics[width=0.9\textwidth]{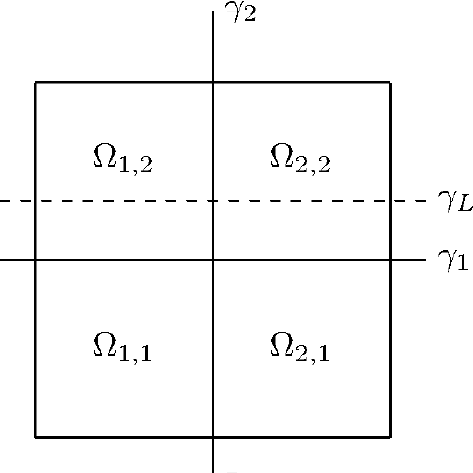}\\
                %(c) $2 \times 2$ partition
        \end{minipage}
        \vspace{-0.2cm}
        \caption{ Two-layered media  (left) and  its $1\times 2$ (middle) or $2\times 2$ (right) domain partitions.  
                 \label{fig:2layer}}
\end{figure*}

\subsubsection{With  one-dimensional strip partitions}

The trace transfer-based one-dimensional sweeping DDM  for the $1\times N_2$ domain partition consists of one forward sweep and one backward sweep, which could be viewed as a degeneration of the diagonal sweeping DDM in $\R^2$ (Algorithm \ref{alg:diag2D_t}).
To illustrate this DDM, the simplest case $1\times 2$ is  first
described below.
Suppose the domain $\Omega = [-l_1, l_1] \times [-l_2, l_2]$ is partitioned into two subdomains $\Omega_1 = [-l_1, l_1] \times [-l_2, 0]$ and $\Omega_2 = [-l_1, l_1] \times [0, l_2]$, by the $x$-axis, which is denoted by $\gamma$, as shown in Figure \ref{fig:2layer}-(middle). 
%denoted $\Gamma = (-\infty, \infty) \time $
%The upper half-plane is denoted $\Omega^+$, while the lower half is $\Omega^-$.
The upper half-plane is defined as $\Omega^+ = \{(x_1,x_2) \;|\; x_2 > 0\}$, and the lower half-plane $\Omega^- = \{(x_1,x_2) \;|\; x_2 < 0\}$.
The two subdomain PML  problems are $\P_{\Omega_1, \kappa_1}$ and $\P_{\Omega_2, \kappa_2}$ with the wave numbers 
$\kappa_1$ and $\kappa_2$ determined according  to \eqref{eq:wave_number}, 
and we denote $G_1$ and $G_2 $ as the fundamental solution to  $\P_{\Omega_1, \kappa_1}$ and $\P_{\Omega_2, \kappa_2}$, respectively.
The potentials associated with the subdomains $\Omega_1$ and $\Omega_2$ are then given by
\begin{align*} 
\Potential_{i} (\lambda) (\bx) := 
\int_{\gamma_{}}J_{\Omega_{i}}^{-1}G_{i}(\bx,\by)%\frac{\p\lambda(\by)}{\pn_{i}}
\big(A_{\Omega_{i}}\nabla_{\by}\lambda(\by)\cdot \n_{i}\big)%-\lambda(\by)(\frac{\pG_{i}(\bx,\by)}{\pn_{i}})d\by,
-\lambda(\by)\Big(A_{\Omega_{i}}\nabla_{\by}\big(J_{\Omega_{i}}^{-1}G_{i}(\bx,\by)\big)\cdot \n_{i}\Big)d\by,
\end{align*}
for $i = 1,2$, where  $\n_1 = (0,-1)$ and  $\n_2 =(0,+1)$.

In the \textbf{first sweep} of upward direction, at step 1 the subdomain problem $\P_{\Omega_1,\kappa_1}$ is solved with the source $f \cdot \chi_{\Omega_1}$, and denote the solution as $u_1=\L_{\Omega_1, \kappa_1}^{-1} (f \cdot \chi_{\Omega_1}) $. At step 2, the subdomain problem $\P_{\Omega_2,\kappa_2}$ is solved with the transferred trace from $\Omega_1$ in addition to the source $f \cdot \chi_{\Omega_2}$, and denote the solution as $u_2$, which is given by 
\begin{align} \label{eq:1x2_u2}
u_2 = \L_{\Omega_2, \kappa_2}^{-1} (f \cdot \chi_{\Omega_2}) + \Potential_{2} (u_1).
\end{align}
In the \textbf{second sweep} of downward direction, at step 1,
since the source $f$ is no longer used and there is no transferred trace passed to $\Omega_2$ from upside,
the subdomain solution of $\P_{\Omega_2,\kappa_2}$ is zero. Then at step 2   the subdomain problem of $\P_{\Omega_1,\kappa_1}$ is solved with the transferred trace from $\Omega_2$, and denote the solution as $u'_1$, which is given by 
\begin{align} \label{eq:1x2_up1}
u'_1 = \Potential_{1} (u_2).
\end{align}
After these two sweeps, the one-dimensional sweeping DDM terminates and generates a DDM solution for  $\P_{\Omega,\kappa}$ with the source $f$, which is given by
\begin{align} \label{eq:1x2_sol}
  u_{\text{DDM}} = (u_1 + u'_1) \mu_{\Omega^-} + u_2 \mu_{\Omega^+}.
  % \overline{\chi}_{\Omega^-} + u_2 \overline{\chi}_{\Omega^+}.
\end{align}

\def\wdff{0.21}
\begin{figure*}[!ht]    
        \centering
        \begin{minipage}[t]{\wdff\linewidth}
                \centering
                \includegraphics[width=0.9\textwidth]{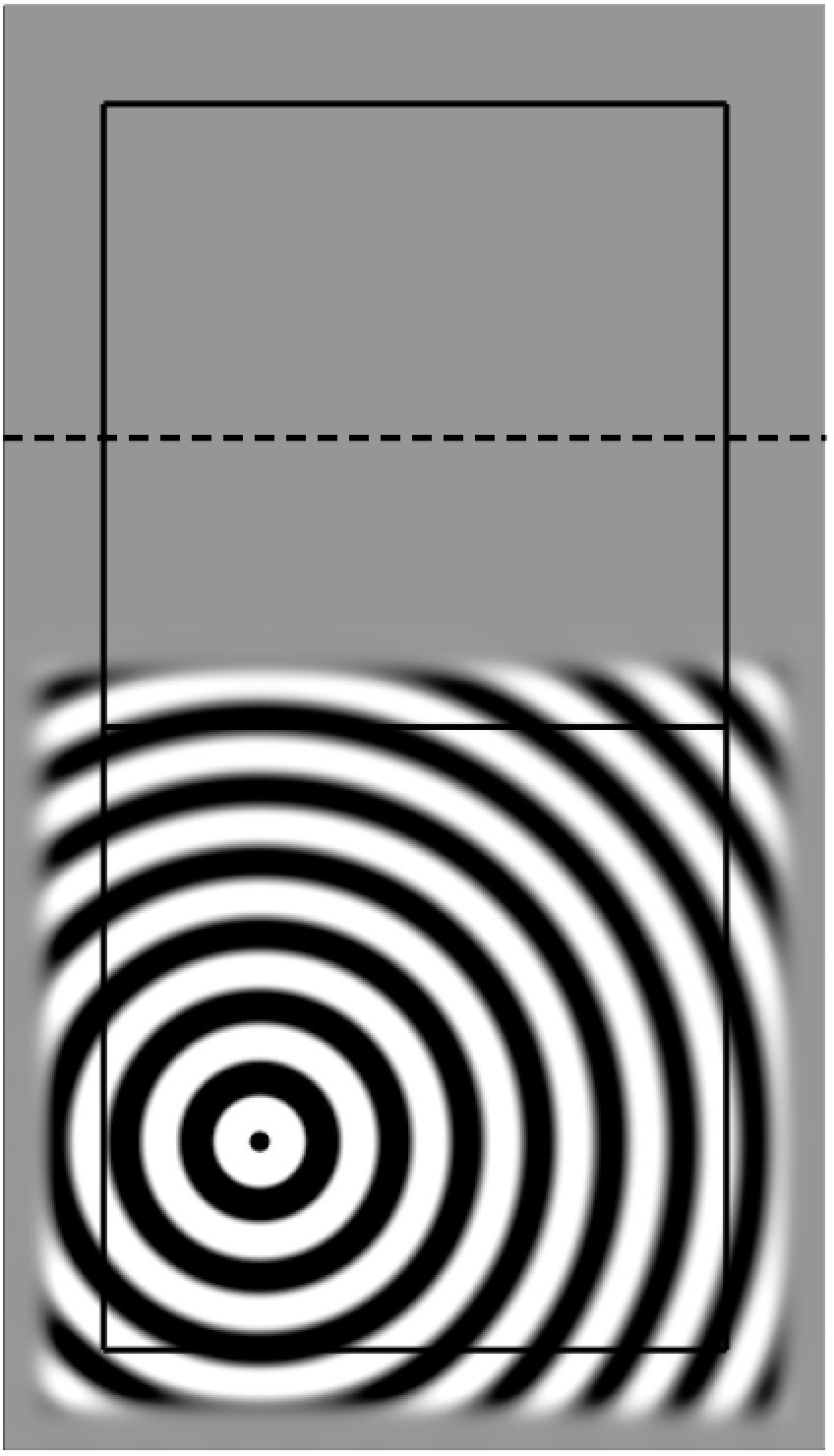}\\
                (a) $u_1$
        \end{minipage}
        \begin{minipage}[t]{\wdff\linewidth}
                \centering
                \includegraphics[width=0.9\textwidth]{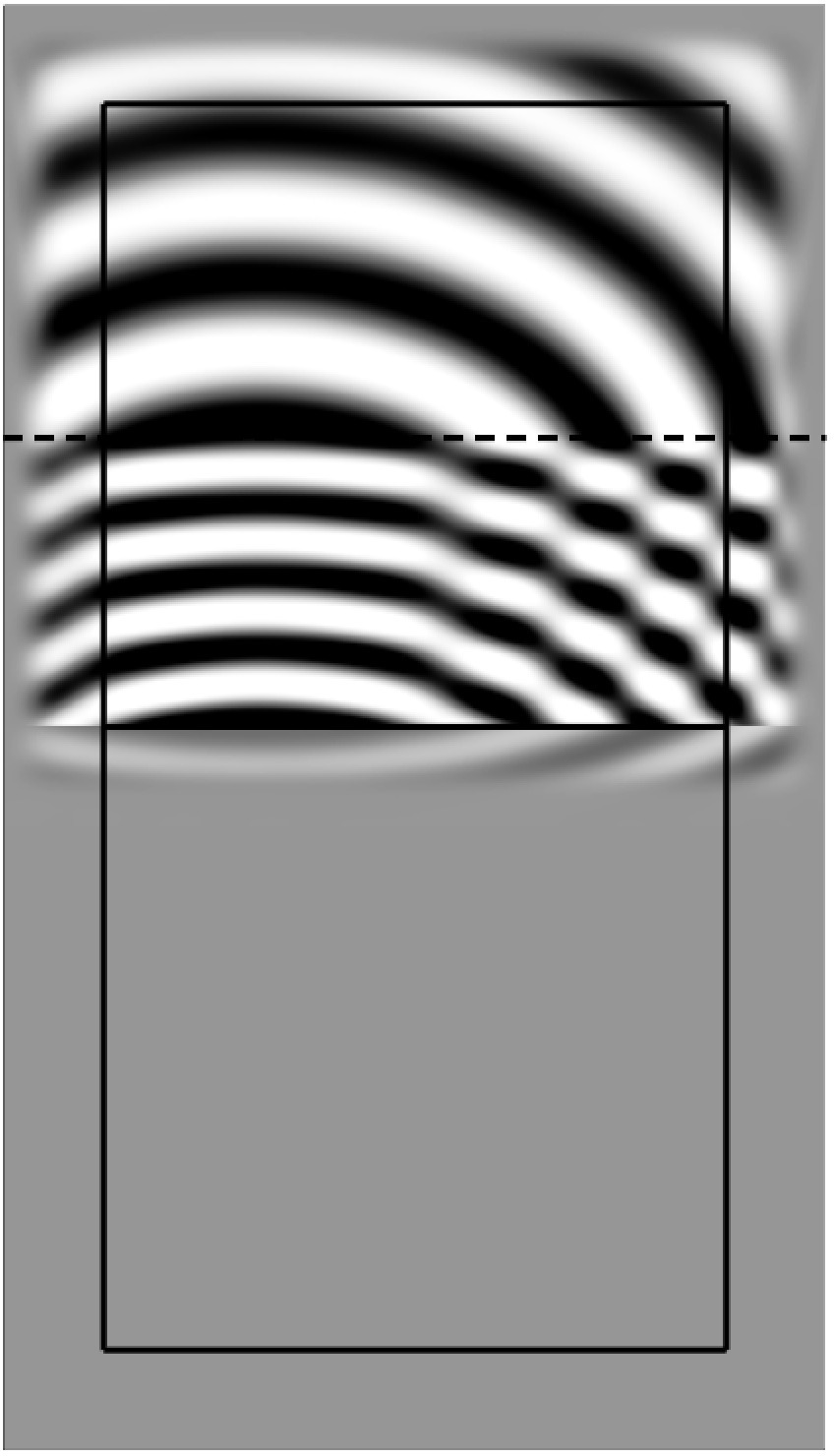}\\
                (b) $u_2$
        \end{minipage}
        \begin{minipage}[t]{\wdff\linewidth}
                \centering
                \includegraphics[width=0.9\textwidth]{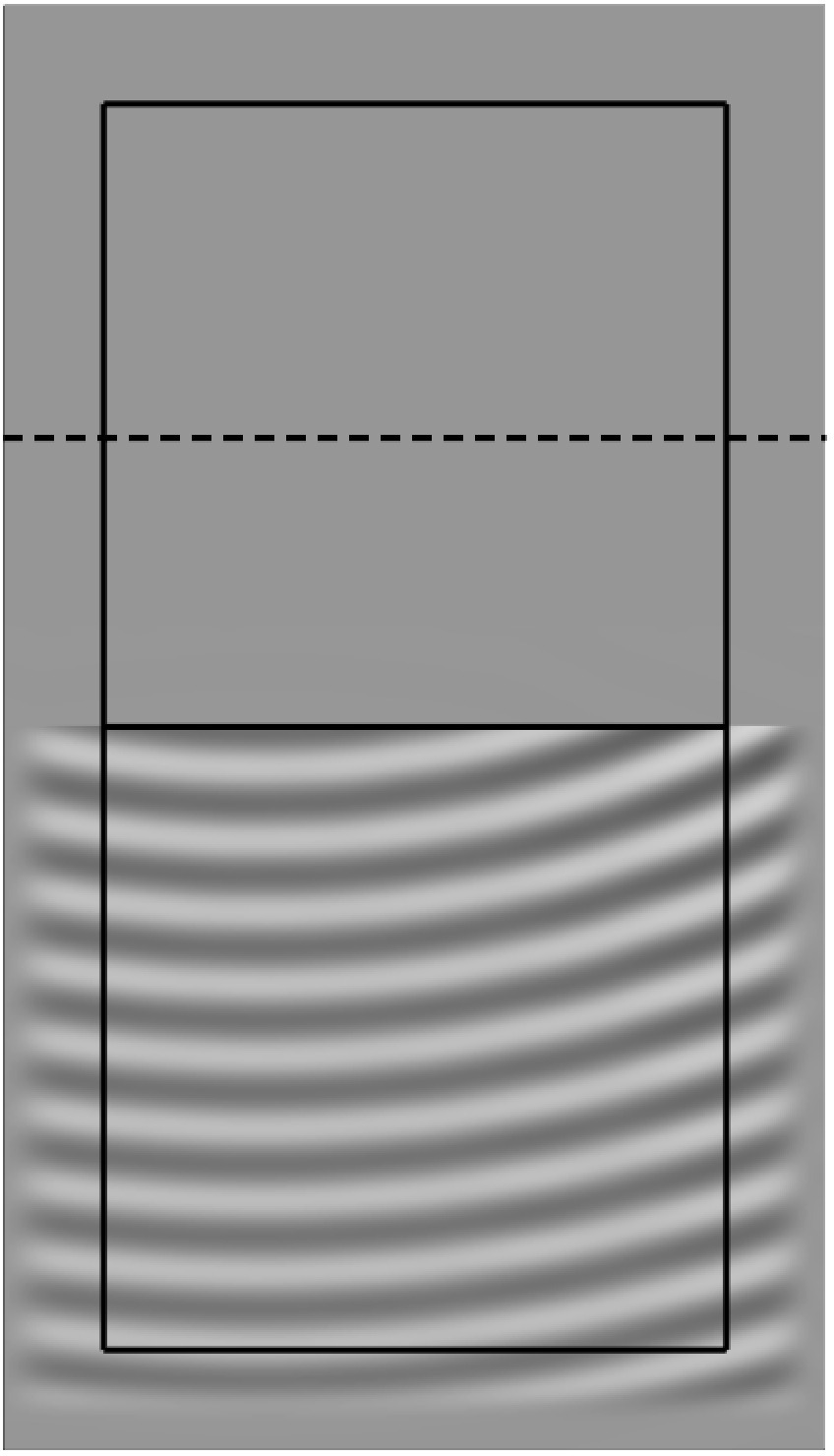}\\
                (c) $u'_1$
        \end{minipage}
        \begin{minipage}[t]{\wdff\linewidth}    
                \centering
                \includegraphics[width=0.9\textwidth]{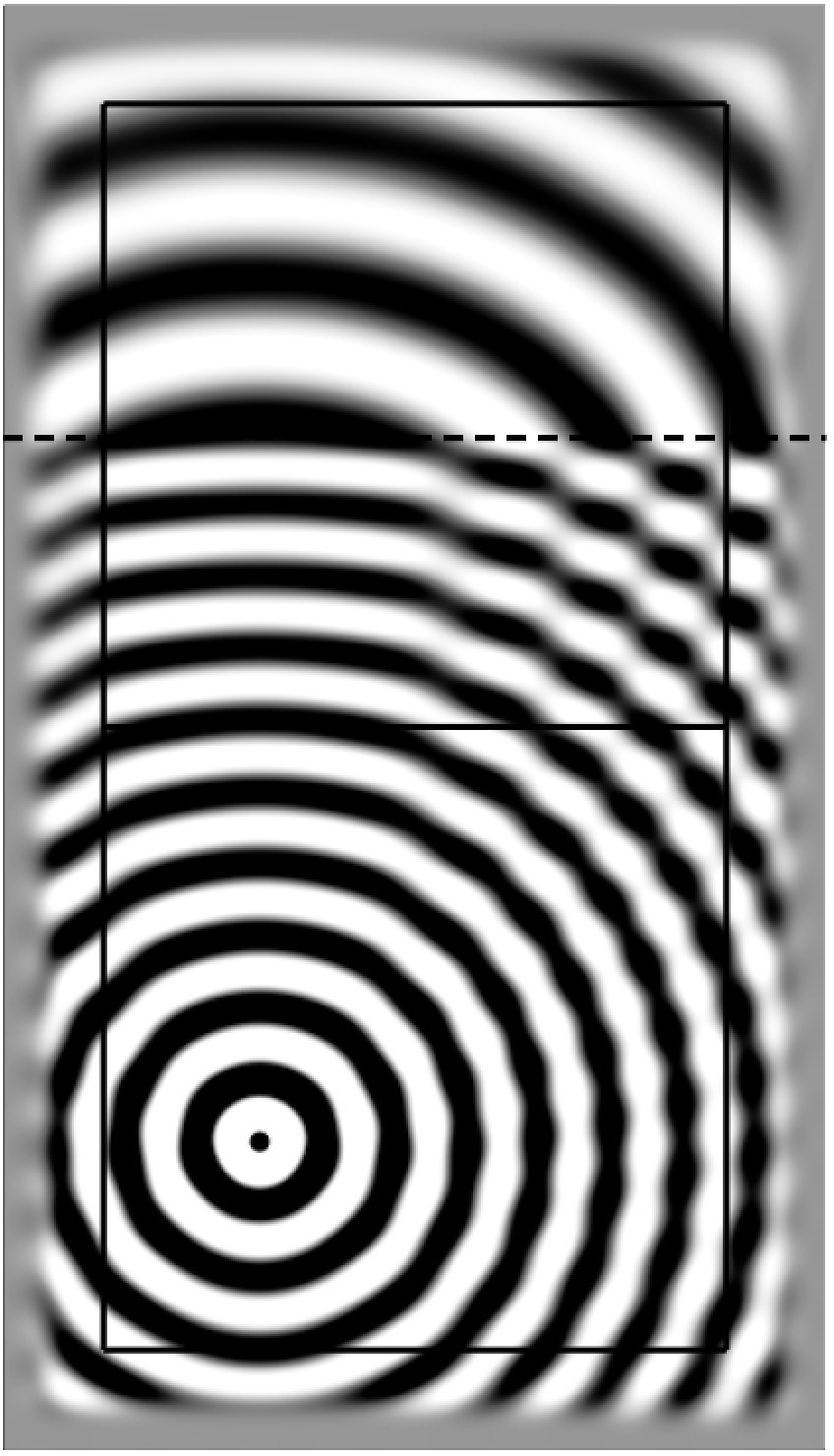}\\
                (d) $u$
        \end{minipage}   
        
        \caption{{Illustration of the  sweeping process of Algorithm \ref{alg:diag2D_t} in the two-layered media case with the $1 \times 2 $ domain partition, $\eta_L > 0$ and $\supp(f) \subset \Omega_{1}$.   
                } \label{fig:sf12R}}
\end{figure*}

On the other hand, the sweeping procedure could still continue so that it results in an iterative solver. For example, 
one more upward sweep could follow the second (downward) sweep, 
and we then have the subdomain solution of $\Omega_{2}$, denoted as $u'_2$ and given by
%\begin{align} \label{eq:1x2_up2}
	$u'_2 = \Potential_{2} (u'_1).$
%\end{align}
Consequently the DDM solution with such three sweeps, denoted $u^3_{\text{DDM}}$, is formed as
\begin{align} \label{eq:1x2_sol2}
	u^3_{\text{DDM}} = (u_1 + u'_1) \mu_{\Omega^-} + (u_2 + u'_2) \mu_{\Omega^+}.
	% \overline{\chi}_{\Omega^-} + u_2 \overline{\chi}_{\Omega^+}.
\end{align}

By Lemma \ref{lemma:change_domain}, it is obvious that the one-dimensional sweeping DDM could produce the exact solution in two sweeps for the two-layered media case with vertical media interface.  In the case of horizontal media interface and $\eta_L > 0$, the exact solution could also be obtained in two sweeps and we have the following lemma.
\begin{lemmaa} \label{lemma:1x2_up}
        Assume that $f$ is a bounded smooth source, the media interface is horizontal and $\eta_L > 0$. Then the one-dimensional sweeping DDM solution $u_{\text{DDM}}$ with the $1\times 2$ domain partition (i.e., the DDM solution \eqref{eq:1x2_sol})  is the exact solution to $\P_{\Omega,\kappa}$ with the source $f$ in the two-layered media case.        
\end{lemmaa}
\begin{proof}

  % In the case that the media interface coincides with the subdomain interface, i.e., $\eta_L = \eta_2$, it is obvious that the DDM solution  \eqref{eq:1x2_sol} is exact due to the wave number definition \eqref{eq:wave_number} and Lemma \ref{lemma:trace}.
  Since $\eta_L > 0$, we have $\kappa_1 = \kdown$ and $\kappa_2 = \kappa$ according to \eqref{eq:wave_number}.
  The result obviously holds in the situation of $\supp(f) \subset \Omega_{2}$,
  thus we only need to consider the situation of $\supp(f) \subset \Omega_{1}$.
  Figure \ref{fig:sf12R}  gives the diagonal sweeping DDM solution process for this situation. The subdomain solution $u_1$ contains the up-going wave of $\Omega^-$, the subdomain solution $u_2$ is already exact in $\Omega^+$, the subdomain solution $u'_1$ contains the reflection in $\Omega^-$ from $\Omega^+$, and 
  $u_{\text{DDM}}$  gives the exact PML solution $u$, as shown in Figure \ref{fig:sf12R}-(a) to (d), respectively.   
  
  % . it is easy to see that the lemma also holds in the case that $\supp(f) \subset \Omega_{2}$, therefore we only consider the case that
  % The solution to the problem  $\P_{\Omega,\kappa}$ of wave number $\kappa$ with source $f$ is denoted by $u$.
We start the proof  with first introducing some notations. Define the following two PML problems associated with $\Omega$: $\P_{\Omega, \kappa}$ is the one with the two-layered media $\kappa$ and $\P_{\Omega, \kdown}$ is the one with the constant medium $\kdown$.
  The solution to the problem  $\P_{\Omega, \kappa}$ with source $f$ is denoted by $u$.
  Let $G^{\kappa}$ and $G^{\kdown}$ be the fundamental solutions to $\P_{\Omega, \kappa}$ and $\P_{\Omega, \kdown}$, respectively.  Define the potential associated with $\P_{\Omega, \kappa}$ as
  \begin{align}
    \Potential_{\pm}^{\kappa}(\lambda)(\bx):=\DD\int_{\gamma}J_{\Omega}^{-1}G^{\kappa}(\bx,\by)\big(A_{\Omega}\nabla_{\by} \lambda(\by)\cdot \n_{\pm}\big)-\lambda(\by)\Big(A_{\Omega}\nabla_{\by} \big(J_{\Omega}^{-1}G^{\kappa}(\bx,\by)\big)\cdot \n_{\pm}\Big)d\by,\label{eq:1x2_Potentialpm}
    % {\Potential}_{\pm}^{\kappa}(\lambda)(\bx)=\DD\int_{\gamma}{G}^{\kappa}(x,y)\Big(A^{\kappa}\nabla\big(J_{\kappa}^{-1}\lambda(\by)\big)\cdot \n_{\pm}\Big)-J_{\kappa}^{-1}\lambda(\by)\Big(A\nabla {G}^{\kappa}(x,y)\cdot \n_{\pm}\Big)dy,
  \end{align}
  where $ \n_{-}=\n_1$ and $ \n_{+}=\n_2$.
  The potentials $\Potential_{\pm}^{\kdown}$  associated with $\P_{\Omega, \kdown}$ are defined similarly.
  %
  % They are different from the potential $\Potential_{1}$ and $\Potential_{2}$, the former involves the Green's functions of total domain problem, namely $G^\kappa$ and $G^{\kdown}$, while the latter involves the Green's function of the subdomain problems, namely $G_{1}$ and $G_{2}$.
  %Note that the above potentials are associated with the problem of $\Omega$, not the subdomain.

  First we prove that the diagonal  sweeping DDM solution $u_{\text{DDM}}$ is exact in the upper half-plane:
    \begin{align} \label{eq:1x2_Omg2}
    u_2 = u,  \qquad &\text{in} \,\, \Omega^+. 
  \end{align}
  To show this, we start by considering the case that $f$ is a delta function $\delta(\bx - \bx')$ with $\bx' \in \Omega_{1}$. In this case, the subdomain solution of $\Omega_1$ at step 1 is $G^{\kdown} (\bx, \bx') $ in $\Omega^-$, and define $G^{\kdown}_{\bx'}(\bx) = G^{\kdown} (\bx, \bx')$ for any $\bx\in\Omega^{-}$.  Then at step 2, the subdomain solution of $\Omega_2$ is $\Potential_2(G_{\bx'}^{\kdown})$ according to \eqref{eq:1x2_u2}.
  In the upper plane $\Omega^+$, by the property of uniaxial PML, we have that for any $\bx'' \in \Omega^+$ and $\by \in \gamma$, $G_2(\bx'', \by) = G^\kappa(\bx'', \by)$, thus
   $\Potential_2(G_{\bx'}^{\kdown}) = \Potential^{\kappa}_{+} (G_{\bx'}^{\kdown})$.
 Then by \eqref{eq:1x2_Potentialpm} and \eqref{eq:fund_sol_sym}, we have that for 
   \begin{align}
     \Psi_{+}^{\kappa}(G_{\bx'}^{\kdown})(\bx'')
    = & \DD\int_{\gamma}J_{\Omega}^{-1}(\by)G^{\kappa}(\bx'',\by)\big(A_{\Omega}\nabla_{\by} G^{\kdown}(\by,\bx')\cdot \n_{+}\big) 
    \nonumber\\
    &\qquad -G^{\kdown}(\by,\bx')\Big(A_{\Omega}\nabla_{\by}\big(J_{\Omega}^{-1}(\by)G^{\kappa}(\bx'',\by)\big)\cdot \n_{+}\Big)d\by\nonumber \\
    = &\; \DD\frac{J_{\Omega}(\bx')}{J_{\Omega}(\bx'')}\int_{\gamma}J_{\Omega}^{-1}(\by)G^{\kdown}(\bx',\by)\big(A_{\Omega}\nabla_{\by} G^{\kappa}(\by,\bx'')\cdot \n_{-}\big)\nonumber\\
    & \qquad\qquad\quad-G^{\kappa}(\by,\bx'')\Big(A_{\Omega}\nabla_{\by}\big(J_{\Omega}^{-1}(\by)G^{\kdown}(\bx',\by)\big)\cdot \n_{-}\Big)d\by\label{eq:1x2_adj} \nonumber\\
    = &\; \frac{J_{\Omega}(\bx')}{J_{\Omega}(\bx'')}\Psi_{-}^{\kdown}(G_{\bx''}^{\kappa})(\bx'),
  \end{align}
  where $G_{\bx''}^{\kappa}$ is the fundamental solution to the problem $\P_{\Omega,\kappa}$ with the delta source $\delta(\bx - \bx'')$.
  %Moreover, by the Lemma \ref{lemma:change_wave_number}, we can show that
  %
  Moreover, using the same argument as done in Lemma \ref{lemma:change_domain}, we  obtain ${\Potential}^{\kdown}_{-} ({G}_{\bx''}^{\kappa}) (\bx') = {\Potential}^{\kappa}_{-} ({G}_{\bx''}^{\kappa}) (\bx')$,
  and by Lemma \ref{lemma:trace}, ${\Potential}^{\kappa}_{-} ({G}_{\bx''}^{\kappa}) (\bx') = {G}^\kappa_{\bx''} (\bx')$,
  thus \eqref{eq:1x2_adj} implies that for any $\bx'' \in \Omega^+$,
  \begin{equation}
    \Potential^{\kappa}_{+} (G_{\bx'}^{\kdown}) (\bx'') 
    = \frac{J_{\Omega}(\bx')}{J_{\Omega}(\bx'')}  G^\kappa_{\bx''} (\bx')
    = G^\kappa_{\bx'} (\bx''), \label{eq:1x2_delta}
    % \Potential^{1} (G_{\bx'}^{\kdown}) (\bx'') = G^\kappa_{\bx'} (\bx'').
  \end{equation}
  which implies that \eqref{eq:1x2_Omg2}  holds in the case of $f$ being a delta function $\delta(\bx-\bx')$ in $\Omega_1$.
  
  Next let us consider a general smooth bounded source $f$ located in $\Omega_1$. By Fubini's theorem, for any $\by \in \gamma$,
  \begin{align} 
    \nabla_{\by} {u_1} (\by) =&\; \nabla_\by \int_{\Omega_1} f(\bx') G^{\kdown} (\by, \bx') d\bx' = \int_{\Omega_1} f(\bx') \nabla_\by G^{\kdown} (\by, \bx') d\bx'. \label{eq:1x2_pu1}
  \end{align}
  % Then the subdomain solution $u_2$ can by \eqref{eq:1x2_delta} and \eqref{eq:1x2_pu1}
  % which implies by that \eqref{eq:1x2_delta} for any $x'' \in \Omega^+$,
 Combining \eqref{eq:1x2_delta} and \eqref{eq:1x2_pu1}, we then  have that for any $\bx'' \in \Omega^+$,
  \begin{align}
    u_{2}(\bx'')= &\; \Potential_{+}^{\kappa}(u_{1})(\bx'')\nonumber \\
    = & \DD\int_{\gamma_{}}J_{\Omega}^{-1}(\by)G^{\kappa}(\bx'',\by)\big(A_{\Omega}(\by)\nabla_{\by}u_{1}(\by)\cdot \n_{2}\big)-u_{1}(\by)\Big(A_{\Omega}(\by)\nabla_{\by}\big(J_{\Omega}^{-1}(\by)G^\kappa(\bx'',\by)\big)\cdot \n_{2}\Big)d\by\nonumber \\
    = & \DD\int_{\gamma} \bigg[ J_{\Omega}^{-1}(\by)G^{\kappa}(\bx'',\by)\Big(A_{\Omega}(\by)\int_{\Omega_{1}}f(\bx') \nabla_\by G^{\kdown}(\by,\bx')  d\bx'\cdot \n_{2}\Big) \nonumber\\
                & \qquad -\Big(\int_{\Omega_{1}}f(\bx')G^{\kdown}(\by,\bx')d\bx'\Big) \Big(A_{\Omega}(\by)\nabla_{\by}\big(J_{\Omega}^{-1}(\by)G^{\kappa}(\bx'',\by)\big)\cdot \n_{2}\Big) \bigg] d\by. \label{eq:1x2_doubleint}
  \end{align}
  The fundamental solution $G^{\kdown}$, $G^{\kappa}$ and their derivatives are singular at $\bx = \bx'$, but still integrable on ${\Omega_1}$, furthermore, they decay exponentially outside of $\Omega_1$, consequently the order of integration in \eqref{eq:1x2_doubleint} can be changed according to Fubini's theorem, therefore we have that  for any $\bx'' \in \Omega^+$,
  \begin{align*}
    u_{2}(\bx'')= & \DD\int_{\Omega_{1}}f(\bx')  \bigg( \int_{\gamma_{}}J_{\Omega}^{-1}(\by)G^{\kappa}(\bx'',\by)\big(A_{\Omega}(\by)\nabla G^{\kdown}(\by,\bx')\cdot \n_{2}\big)\\
                &\qquad\qquad\qquad -G^{\kdown}(\by,\bx')\Big(A_{\Omega}(\by)\nabla_{\by}\big(J_{\Omega}^{-1}(\by)G^{\kappa}(\bx'',\by)\big)\cdot \n_{2}\Big)d\by \bigg) d\bx' \\
    = & \DD\int_{\Omega_{1}}f(\bx')G^{\kappa}(\bx'',\bx')d\bx'
    \;=\; u(\bx'').
  \end{align*}

  % -------------------------------------------
  % 
  % 
  % 1x2 Lower region
  % 
  % 
  % -------------------------------------------

  Next we prove that the diagonal  sweeping DDM solution $u_{\text{DDM}}$ is also exact in the lower half-plane $\Omega^-$:
  \begin{align} \label{eq:1x2_Omg1}
    u_1 + u'_1 = u,  \qquad \text{in} \,\, \Omega^-.
  \end{align}
%
  % From \eqref{eq:1x2_up1} we know the limit value of subdomain solution $u_2 = \Potential_2(u_1)$ on $\gamma$ from the $\Omega^{-}$ side is needed to evaluate $u'_1$.
 Let us investigate the  limit value of the subdomain solution $u_2$ on $\gamma$ from the $\Omega^{-}$ side,
  then study the value of $u'_1$ using \eqref{eq:1x2_up1}.
  Since the media interface $\gamma_{L}$ is away from the subdomain interface, we  denote the distance as $d_*>0$ and  introduce a smooth cutoff function $\beta_* (x_2) = \beta(\frac{x_2 - d_*}{d_*})$, which is 0 for $x_2\geq d_k$, 1 for $x_2\leq0$, and smooth in $x_2 \in (0, d_*)$.
  %Denote region $\Omega_{\beta}= (-\infty, \infty) \times (d_k/2, d_k)$.
  %is nonzero only in the region $\Omega_{\beta}= (-\infty, \infty) \times (d_k/2, d_k)$.
  Define
  \begin{align} 
    v = (u - u_1) \beta_{*}, \qquad \text{in} \,\,  \R^2, \label{eq:1x2_v}
  \end{align} 
and it is clear that 
$ \L_{\Omega,\kappa} v = 0$ in $\R^2 / \Omega_{\beta}$, where
$\Omega_{\beta}= (-\infty, \infty) \times (0, d_*)$.
  This indicates that $v$ is the solution to the problem $\P_{\Omega,\kappa}$ with the source lying in $\Omega_{\beta}$.
  Let $\widetilde{v}$ be the solution to the problem $\P_{\Omega_2,\kappa}$ with the same source, that is,
 $
    \L_{\Omega_2,\kappa} \widetilde{v} = \L_{\Omega, \kappa} v$ in  $\R^2$, 
  which implies 
  \begin{align}
    \Psi_2(\widetilde{v}) = -\widetilde{v}, \quad \text{in} \,\, \Omega^-, \label{eq:1x2_tildev} 
  \end{align}
according to Lemma \ref{lemma:trace}.  Since the problem  $\P_{\Omega_2,\kappa}$ is a subdomain problem of   $\P_{\Omega,\kappa}$ and $v$ and $\widetilde{v}$ shares the same source, we know that $v = \widetilde{v}$ on $\gamma$, which further implies from \eqref{eq:1x2_v} and \eqref{eq:1x2_tildev} that
   $ \Psi_2(u - u_1) = -\widetilde{v}$ in $\Omega^-$.
   Moreover, by Lemma \ref{lemma:change_domain}, it holds $\Psi_2(u) = 0$ in $\Omega^-$.
   Consequently we obtain that
 % \begin{align}
    $\Psi_2(u_1) = \widetilde{v}$, in $\Omega^-$.  %\label{eq:1x2_u2_negtive} 
  %\end{align}
  Therefore it holds that
\begin{align}
  % \lim\limits_{x_2 \rightarrow 0-}\Psi_2(u_1) = u - u_1. \label{eq:1x2_u2_interface}
    \lim\limits_{x_2 \rightarrow 0-}u_2 = u - u_1. \label{eq:1x2_u2_interface}
\end{align}

On the other hand, $v$ also satisfies that 
  $ \L_{\Omega,\kdown} v = 0$ in $\R^2 / \Omega_{\beta},$  thus by Lemma \ref{lemma:trace} we get 
  \begin{align}
    \Psi_1(v) = \Psi_{-}^{\kdown}(v) = v, \quad \text{in} \,\, \Omega^-.
  \end{align}
 Together with  \eqref{eq:1x2_v} and \eqref{eq:1x2_u2_interface}, we have
  \begin{align}
    \Psi_1(u_2) = u - u_1, \quad \text{in} \,\, \Omega^-, 
  \end{align} 
  and thus \eqref{eq:1x2_Omg1} is true. 
\end{proof}

In the above analysis for the $1 \times 2$ partition, we have proved that  the one-dimensional sweeping DDM  produces the exact solution in two sweeps (one upward and one downward) in the case of horizontal media interface and $\eta_L > 0$. However, in the case of $\eta_L < 0$, one more upward sweep is needed to construct the exact solution. {Note that in this case $\supp(f)$ could be  contained in the subdomains above the media interface $\gamma_L$.} 

\begin{lemmaa} \label{lemma:1x2_down}
	Assume that $f$ is a bounded smooth source,  the media interface is horizontal and $\eta_L < 0$. Then the DDM solution $u^3_{\text{DDM}}$  defined by \eqref{eq:1x2_sol2} with the $1\times 2$ domain partition is the exact solution to $\P_{\Omega,\kappa}$ with the source $f$ in the two-layered media case.        
\end{lemmaa}
\begin{proof}
	It is obvious that  the exact solution can be constructed by the first (upward) sweep in the situation of $\supp(f) \subset \Omega_1$.
	In the  situation of $\supp(f) \subset \Omega_2$, the direct wave in $\Omega_2$ is obtained in the first (upward) sweep, the reflection and refraction wave in $\Omega_1$ is obtained in the second (downward) sweep, and the reflection wave in $\Omega_2$ is obtained in the third (upward) sweep. The details are omitted here since the case of $\eta_L < 0$ is a mirror of the case $\eta_L > 0$, and all the differences are due to the order of sweepings in our DDM algorithm.   
\end{proof}

For the general one-dimensional strip partition $1\times N_2$, 
%the exact solution could be constructed with the sweeping DDM, 
the number of sweeps of the one-dimensional sweeping DDM to construct the exact solution for the  two-layered media problem also depends on the position of the layer interface. For example, for the partitions $1 \times N_2$, assume $f = f_{1,N_2}$ and the horizontal media interface satisfies that $\gamma_L$ lies in $\Omega_{1,j_L}$ with $j_L<N_2$,
then the first (upward) sweep constructs the direct wave in $\Omega_{1,N_2}$,
the second  (downward) sweep constructs the direct wave in $\Omega_{1,j}$, $j=N_2-1,\ldots,j_L$,
the reflection in $\Omega_{1,j_L}$ and the refraction in $\Omega_{1,j}$, $j=j_L,\ldots,1$.
An extra upward sweep is still needed to construct the reflection in $\Omega_{1,j}$, $j={j_L+1,\ldots,N_2}$.
By using Lemma \ref{lemma:change_domain} for the trace transfer and similar arguments in Lemmas \ref{lemma:1x2_up} and \ref{lemma:1x2_down} for the $1\times 2$ partition,
we can obtain the following result.
\begin{thm} \label{lemma:1xN}
	Assume that $f$ is a bounded smooth source and the domain $\Omega$ is decomposed into $1\times N_2$ subdomains, then
	the DDM solution generated by the one-dimensional sweeping DDM  with one extra upward sweep is the exact solution to $\P_{\Omega,\kappa}$ with the source $f$ in the two-layered media case.        
\end{thm}

% -------------------------------------------
%
%
%   2x2 
%
%
% -------------------------------------------
\subsubsection{With  general checkerboard partitions}

For a general checkerboard partition of $\Nbx \times \Nby$, we first show that the DDM solution of Algorithm \ref{alg:diag2D_t} %diagonal sweeping DDM with trace transfer
is exactly the solution to $\P_{\Omega, \kappa}$ in the case of two-layered media and horizontal media interface when $\supp(f)$ is not contained in any of the subdomains above the media interface $\gamma_L$.
%Moreover, the exact solution can be constructed with one extra round of sweeps for the case of general source $f$. 
%
Without loss of generality, the $2 \times 2$ partition case is considered for demonstration.
%Again, the solution of diagonal sweeping DDM is exact in the case that the media interface coincides with the subdomain interface, i.e., $\eta_L = \eta_2$, the proof is quite similar to the case of constant medium and thus omitted.
%We assume that the $\eta_L > \eta_2$, otherwise the case would be obvious.
For simplicity, the following notations are used for  the $2\times 2$ partition of the domain $\Omega = [-l_1, l_1] \times [-l_2, l_2]$.  Let $\OmegaXminus$ and $\OmegaXplus$ be the left and right half-plane, while $\OmegaYminus$ and $\OmegaYplus$ be the lower and upper half-plane, respectively.  The characteristic functions for half-planes are defined as $\chi_{i}^{\pm} = \chi_{\Omega^{\pm}_{x_i}}$, $i = 1,2$.  The four quadrant planes are denoted as $\quadOne$, $\quadTwo$, $\quadThree$ and $\quadFour$, and the $x$- and $y$-axes are denoted as $\gamma_1$ and $\gamma_2$, respectively, as shown in Figure \ref{fig:2layer}-(right). Assume the media interface is horizontal and $\eta_{L}>0$,
similar to Lemma \ref{lemma:1x2_up} for the $1\times 2$ partition, %the condition of the media interface being horizontal and $\eta_L > 0$ 
such media interface condition implies that the source can not lies in any of the subdomains above the media interface, hence no matter where the source locates,  four diagonal sweeps is enough to produce the exact solution.
% eta_y > 0
\begin{lemmaa} \label{lemma:2x2_up}
	Assume that $f$ is a bounded smooth source, the media interface is horizontal and $\eta_L > 0$, and the domain $\Omega$ is decomposed into $2 \times 2$ subdomains. Then the DDM solution $u_{\text{DDM}}$ of Algorithm \ref{alg:diag2D_t}  is the exact  solution to $\P_{\Omega,\kappa}$ with the source $f$ in the two-layered media case.        
\end{lemmaa}
\begin{proof}

\def\wdff{0.22}
\begin{figure*}[!ht]    
	\centering
	\begin{minipage}[t]{\wdff\linewidth}
		\centering
		\includegraphics[width=0.9\textwidth]{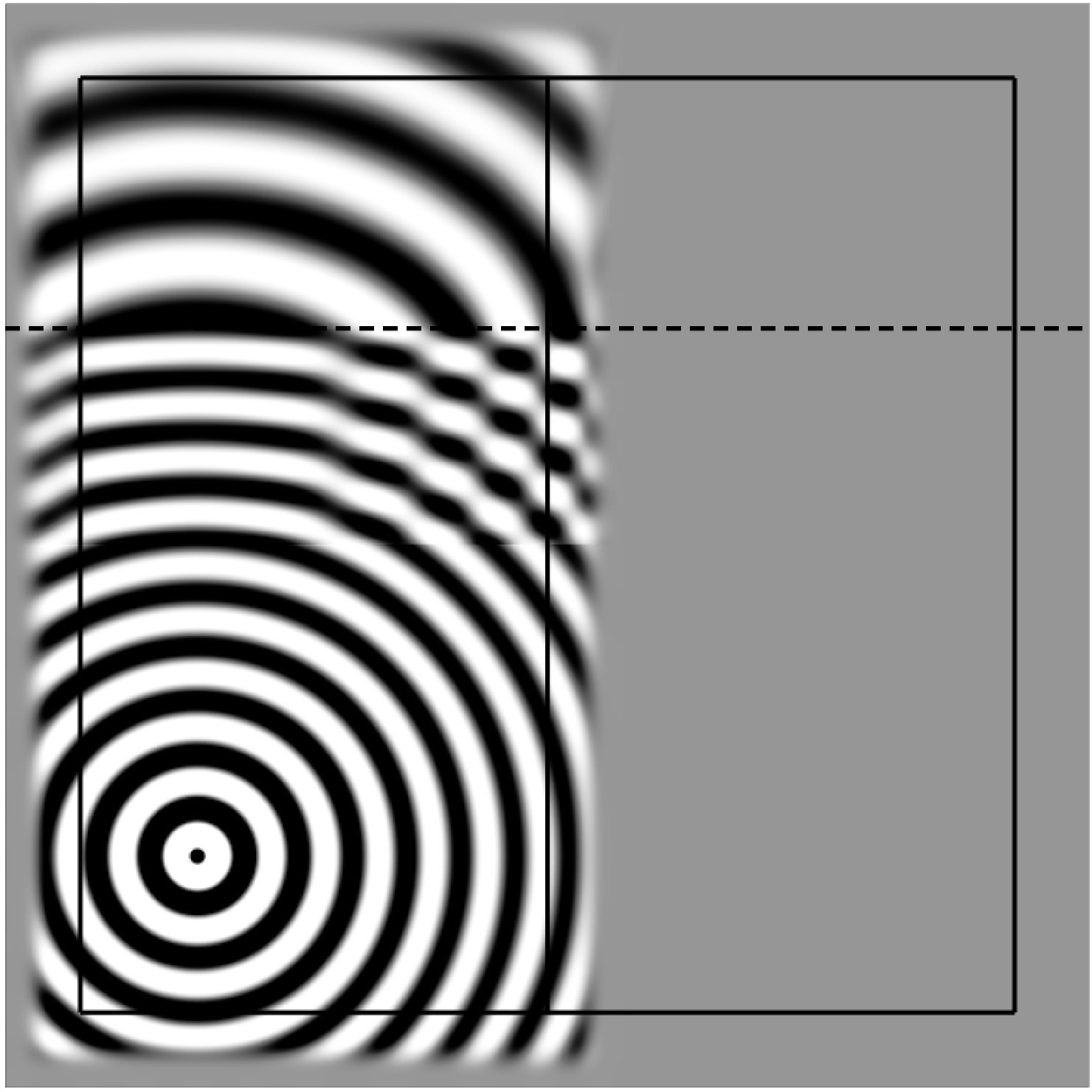}\\
		(a) $u_{\abbrLeft}$
	\end{minipage}
	\begin{minipage}[t]{\wdff\linewidth}
		\centering
		\includegraphics[width=0.9\textwidth]{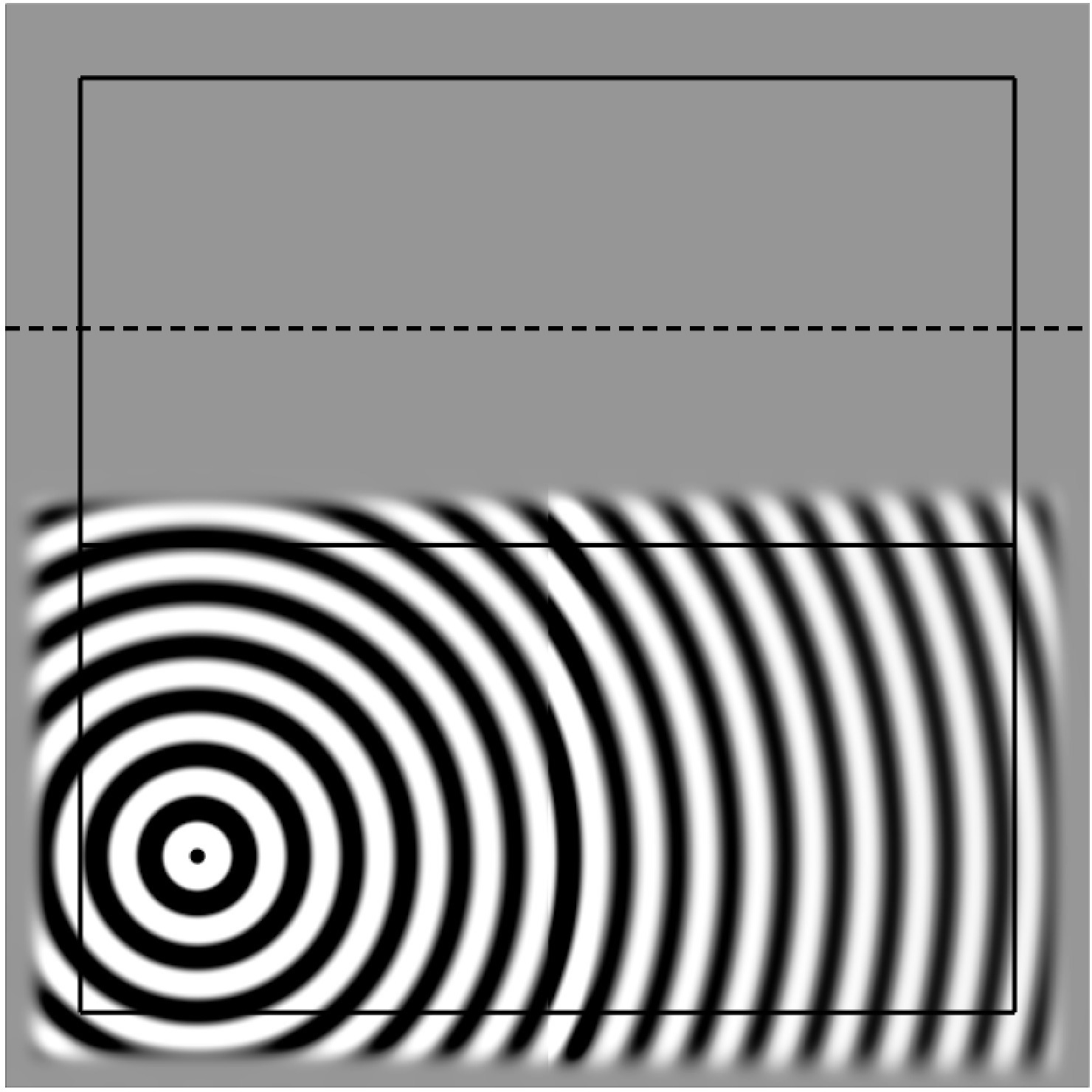}\\
		(b) $u_{\abbrPred}$
	\end{minipage}
	\begin{minipage}[t]{\wdff\linewidth}
		\centering
		\includegraphics[width=0.9\textwidth]{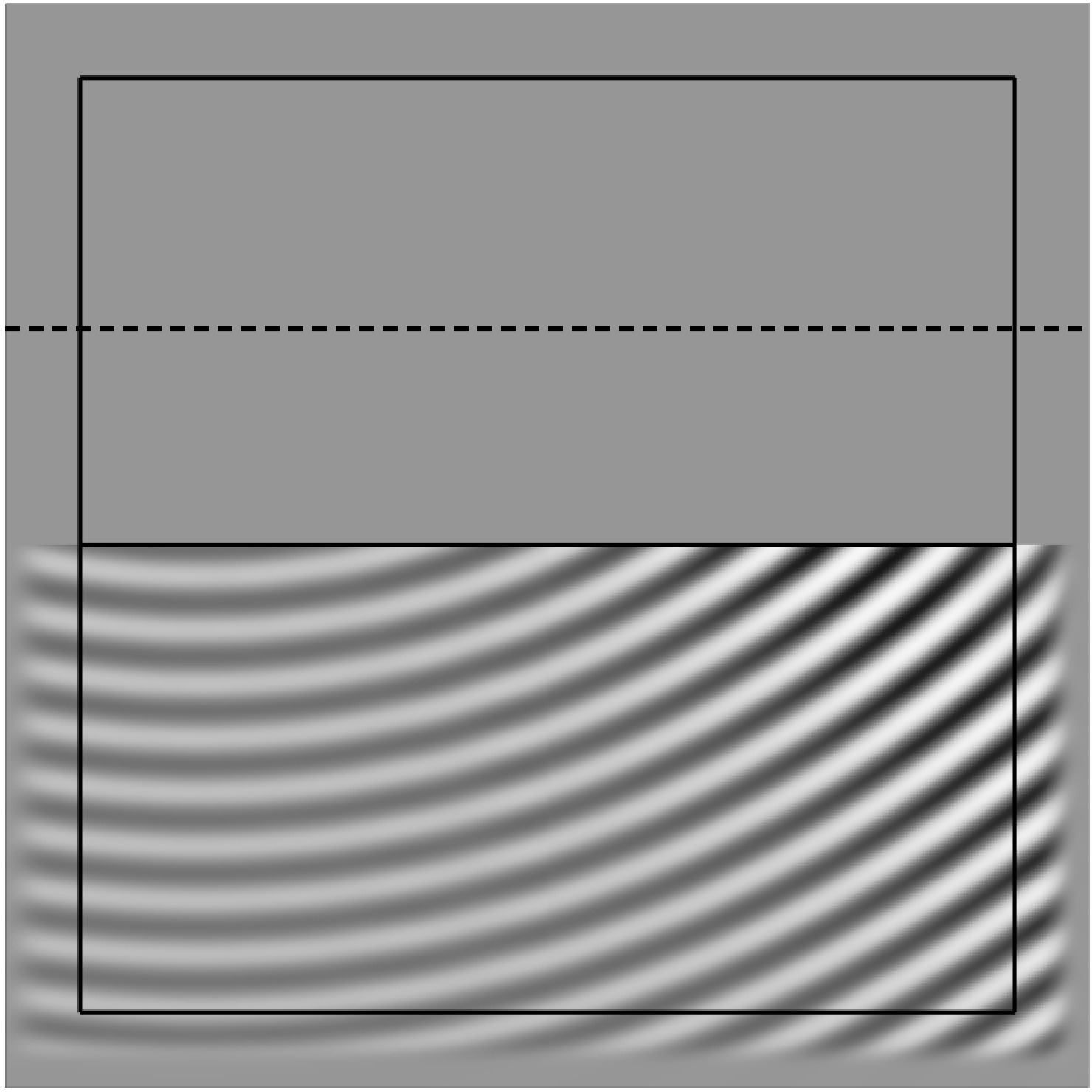}\\
		(c) $u_{\abbrRefl}$
	\end{minipage}
	\begin{minipage}[t]{\wdff\linewidth}    
		\centering
		\includegraphics[width=0.9\textwidth]{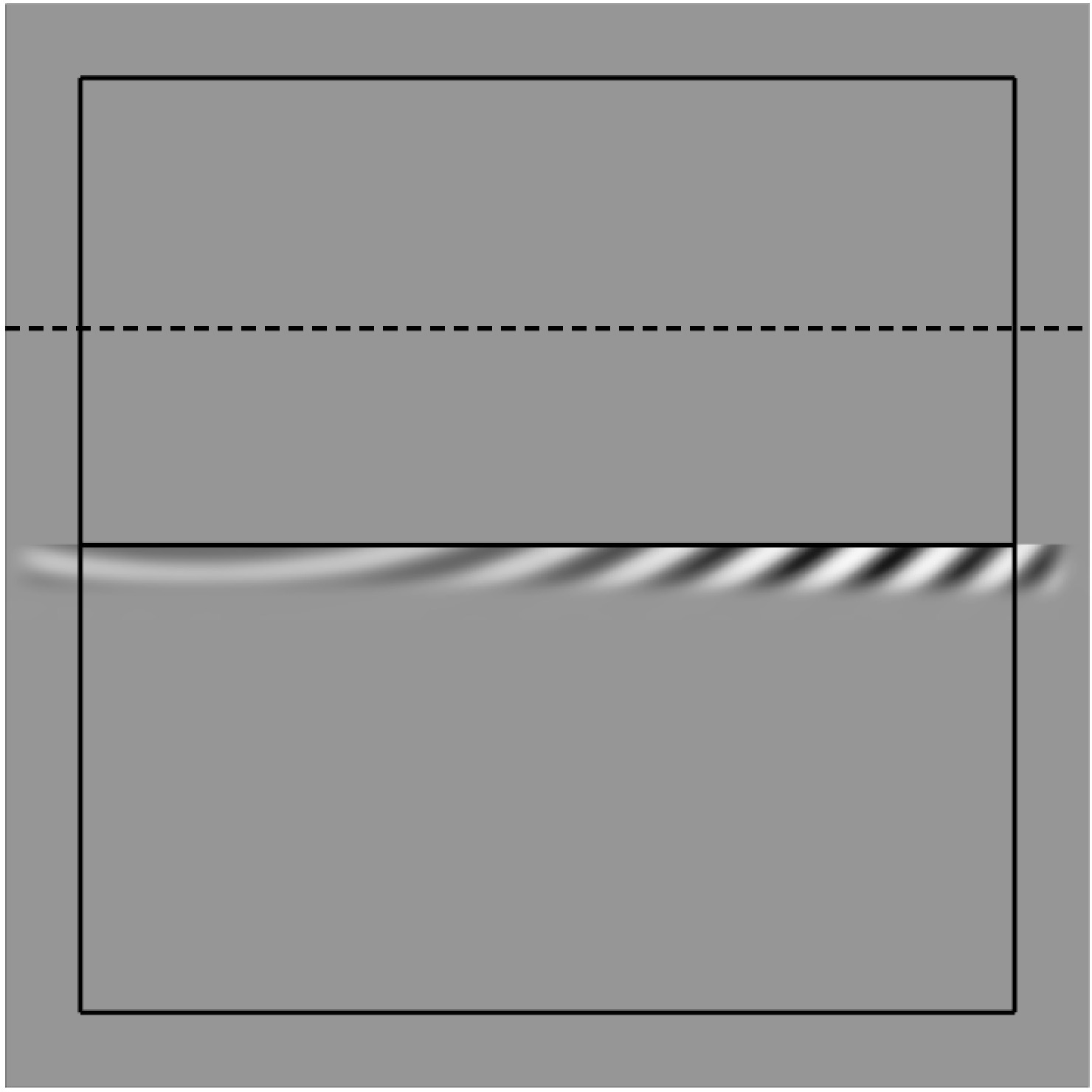}\\
		(d) $\widetilde{u}_{\abbrRefl}$
	\end{minipage}  
	\caption{{Illustration of the partial solutions in the two-layered media case  with the $2 \times 2 $ domain partition, $\eta_L>0$ and $\supp(f) \subset \Omega_{1,1}$.} \label{fig:sf22r_part}}
\end{figure*}

Following the similar line of augments used in Theorem \ref{thm:sweep2d} for constant medium, and using  Lemma \ref{lemma:change_domain}, 
it is easy to prove the lemma holds in the case that the source lies within the upper subdomains $\Omega_{1,2}$ and $\Omega_{2,2}$.
Now we only consider the case that the source lies within subdomain $\Omega_{1,1}$, since lying in any of the two lower subdomains makes no essential difference.

It is clear that the  PML problem $\P_{\Omega_{1,1;1,2},\kappa}$ of the {\bf left} half region is a two-layered media problem using the $1\times 2$ partition, of which the solution with the source $f$ is denoted as $u_\abbrLeft$, as shown in Figure \ref{fig:sf22r_part}-(a).
If we regard the regions $\Omega_{1,2;1,1}$ and $\Omega_{1,2;2,2}$ as two subdomains, then the problem $\P_{\Omega, \kappa}$ still can be treated as a two-layered media problem with the $1\times 2$ partition  as discussed in the preceding subsection, of which the lower subdomain solution in the upward sweep is the \textit{direct} wave,  denoted as $u_{\abbrPred}$ (shown in Figure \ref{fig:sf22r_part}-(b)). The lower subdomain solution in the downward sweep is the \textit{reflection} from the media interface, denoted as $u_{\abbrRefl}$ (shown in Figure \ref{fig:sf22r_part}-(c)). Since the reflection in the lower subdomain is caused by the medium change in the upper subdomain, we can define the PML solution of the reflection for the upper subdomain problem as $\widetilde{u}_{\abbrRefl}$, as illustrated in Figure \ref{fig:sf22r_part}-(d), which satisfies $u_{\abbrRefl} = \widetilde{u}_{\abbrRefl}$ on $\gamma_1$ and $\widetilde{u}_{\abbrRefl}$ decays exponentially in the lower half-plane. %$\OmegaYminus$.

\def\wdff{0.22}
\begin{figure*}[!ht]    
        \centering
        \begin{minipage}[t]{\wdff\linewidth}
                \centering
                \includegraphics[width=0.9\textwidth]{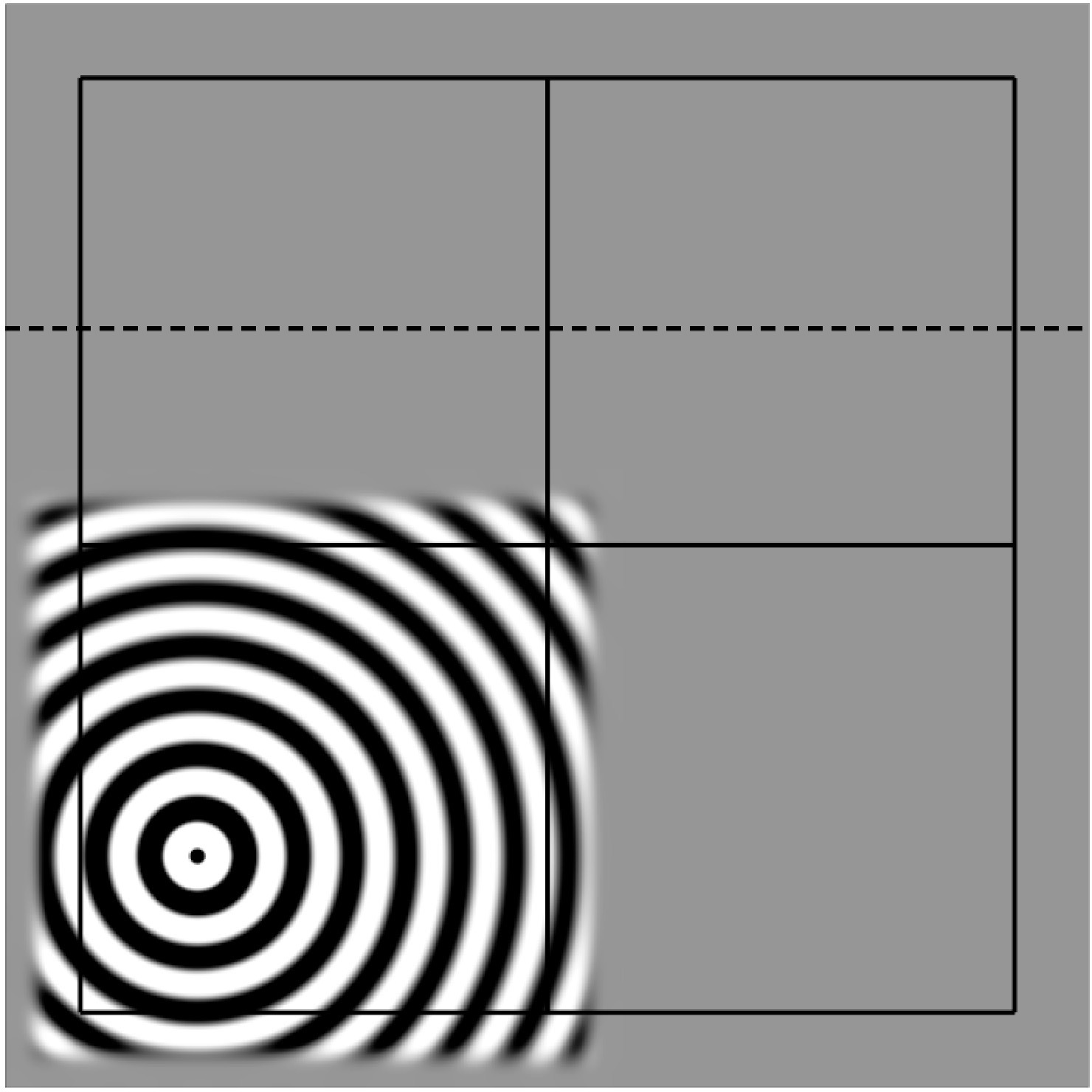}\\
                (a) $u_{1,1}^{1}$
        \end{minipage}
        \begin{minipage}[t]{\wdff\linewidth}
                \centering
                \includegraphics[width=0.9\textwidth]{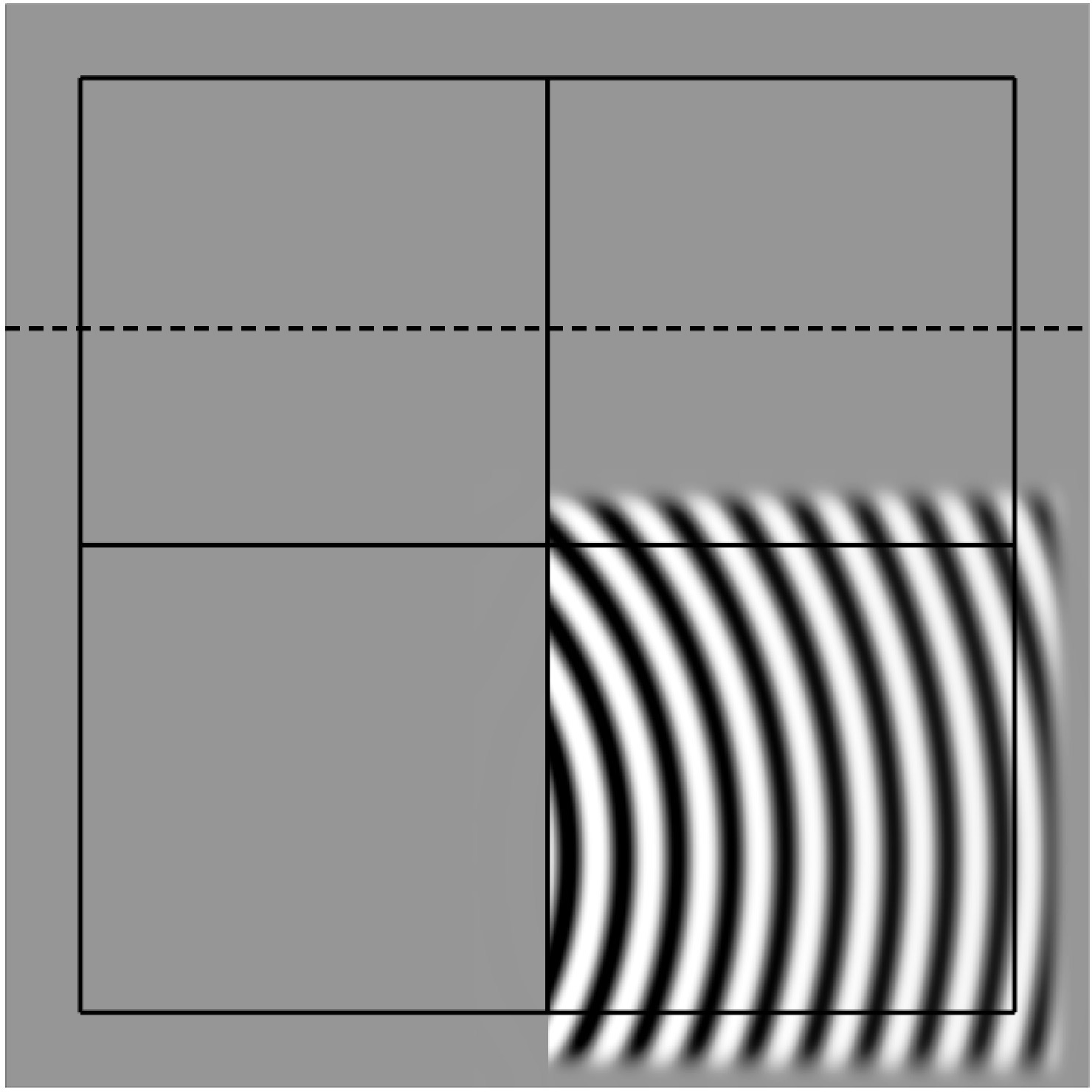}\\
                (b) $u_{2,1}^{1}$
        \end{minipage}
        \begin{minipage}[t]{\wdff\linewidth}
                \centering
                \includegraphics[width=0.9\textwidth]{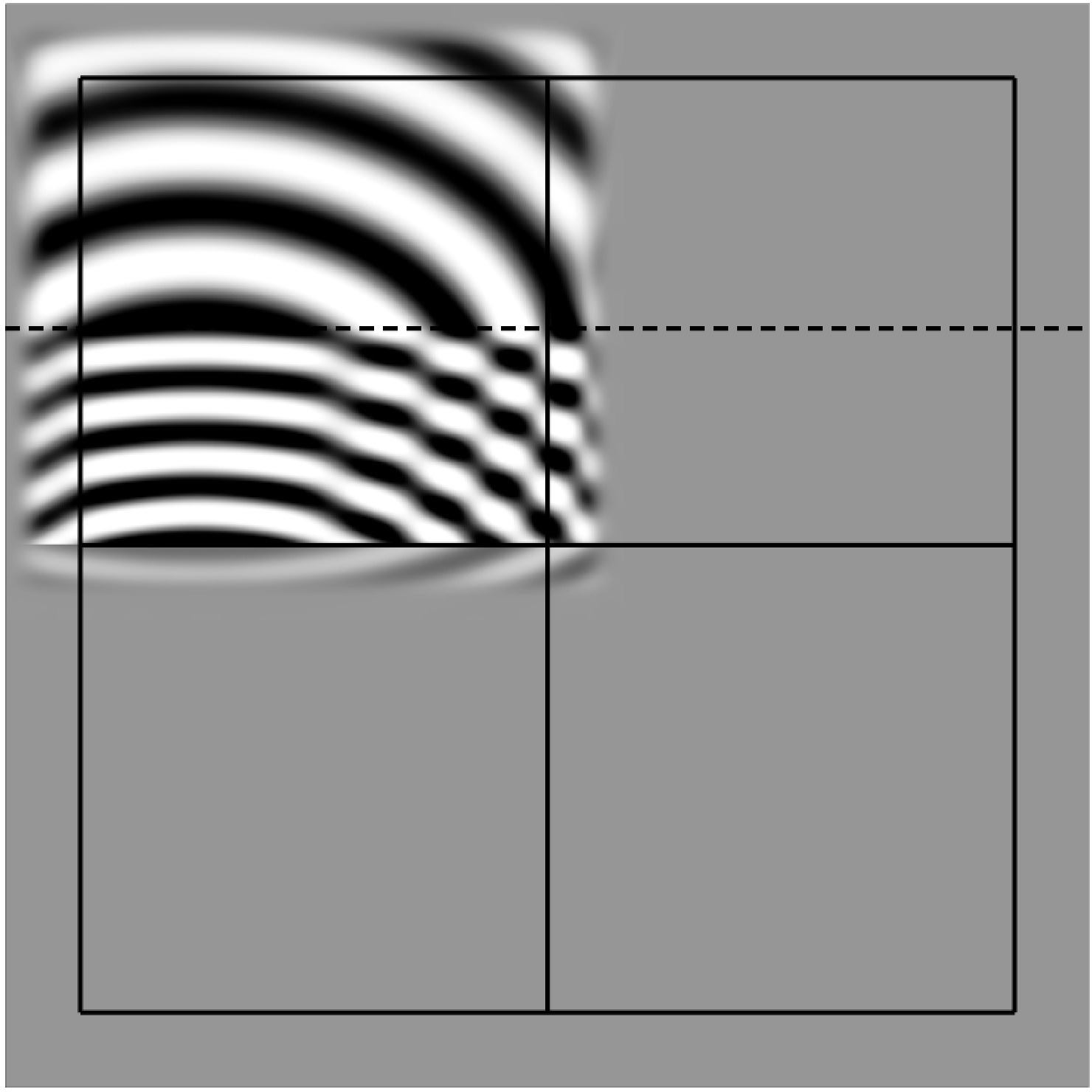}\\
                (c) $u_{1,2}^{1}$
        \end{minipage}
        \begin{minipage}[t]{\wdff\linewidth}    
                \centering
                \includegraphics[width=0.9\textwidth]{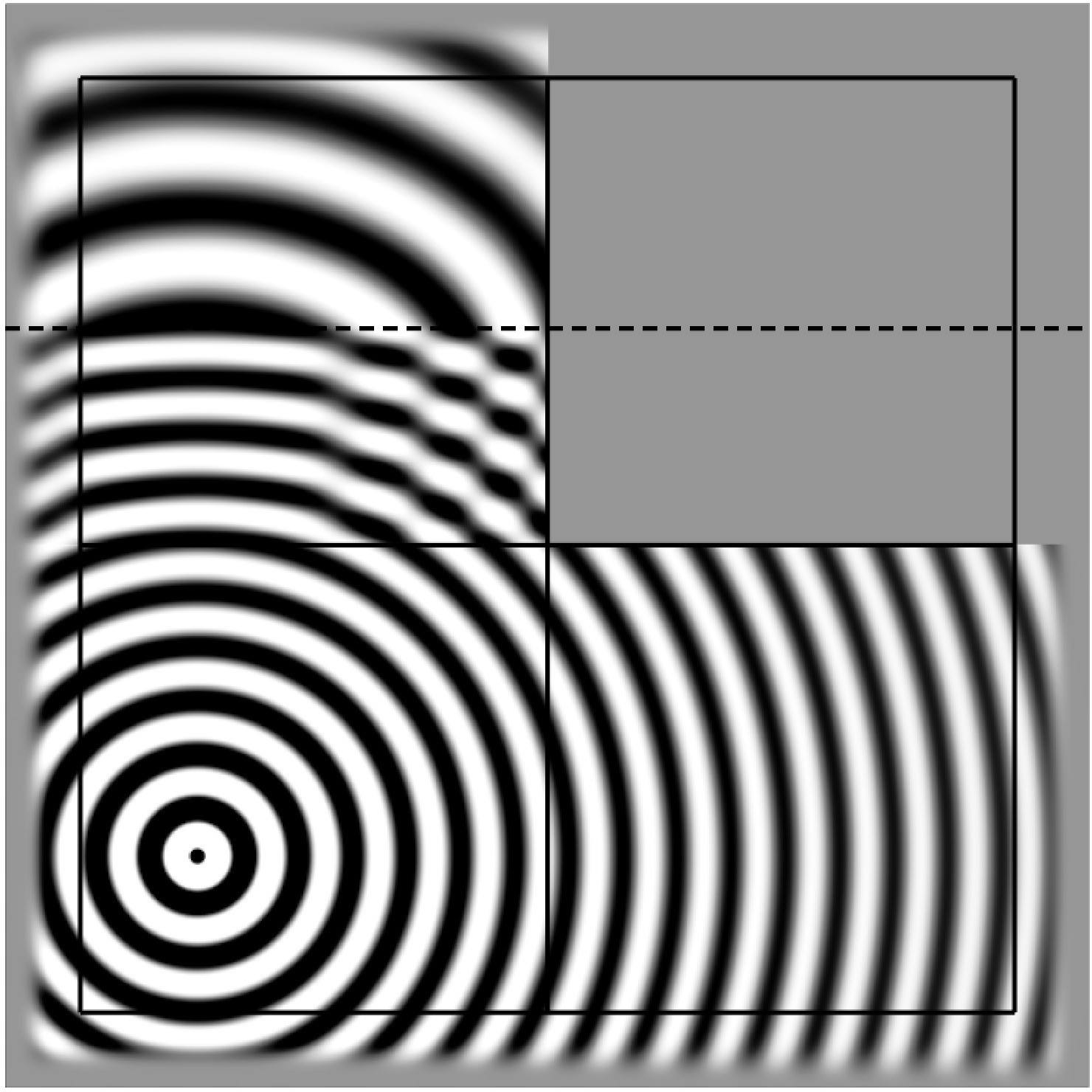}\\
                (d) $U_{2}^{(1)}$
        \end{minipage}  \\      
        \vspace{0.2cm}
        
        \begin{minipage}[t]{\wdff\linewidth}
                \centering
                \includegraphics[width=0.9\textwidth]{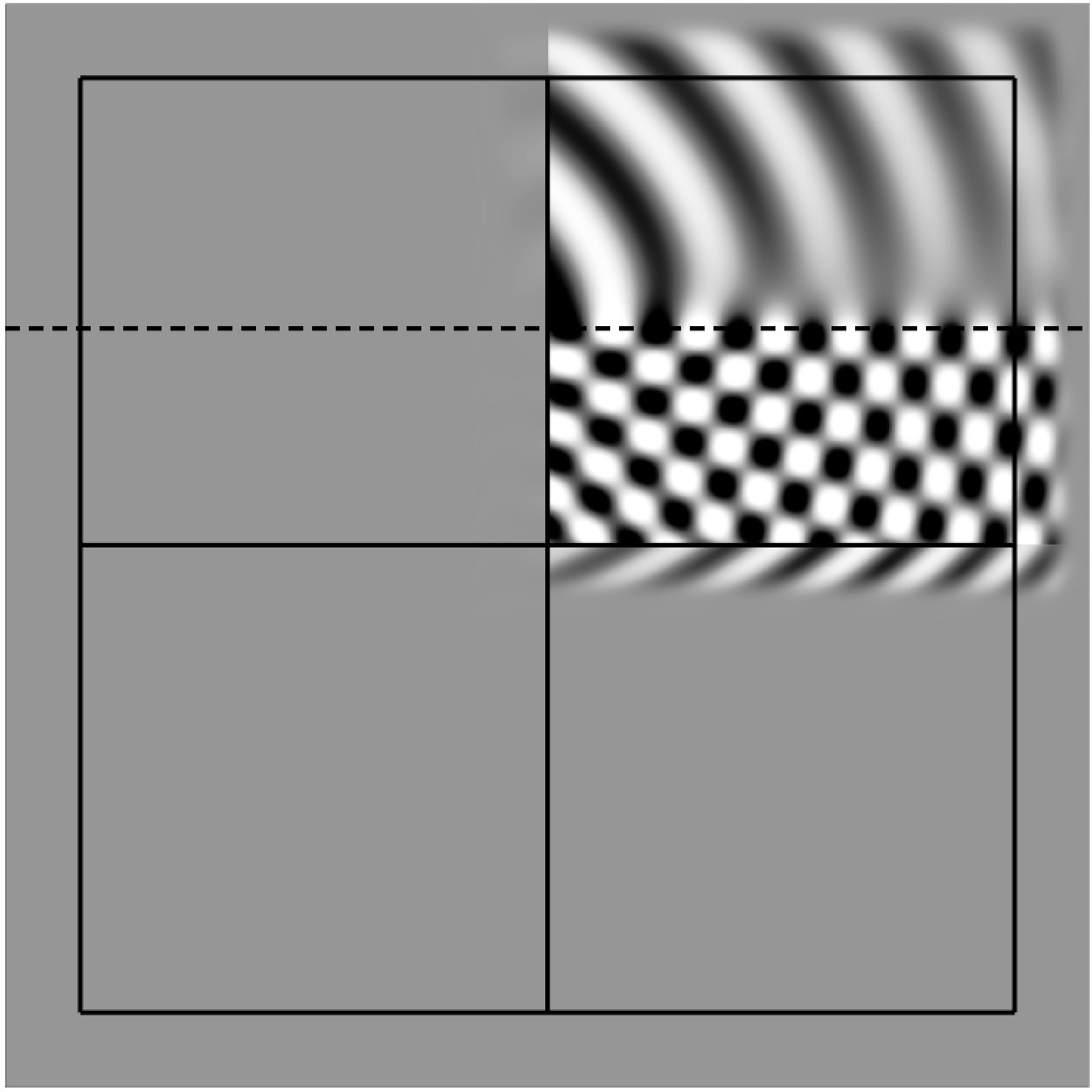}\\
                (e) $u_{2,2}^{1} $
        \end{minipage}
        \begin{minipage}[t]{\wdff\linewidth}
                \centering
                \includegraphics[width=0.9\textwidth]{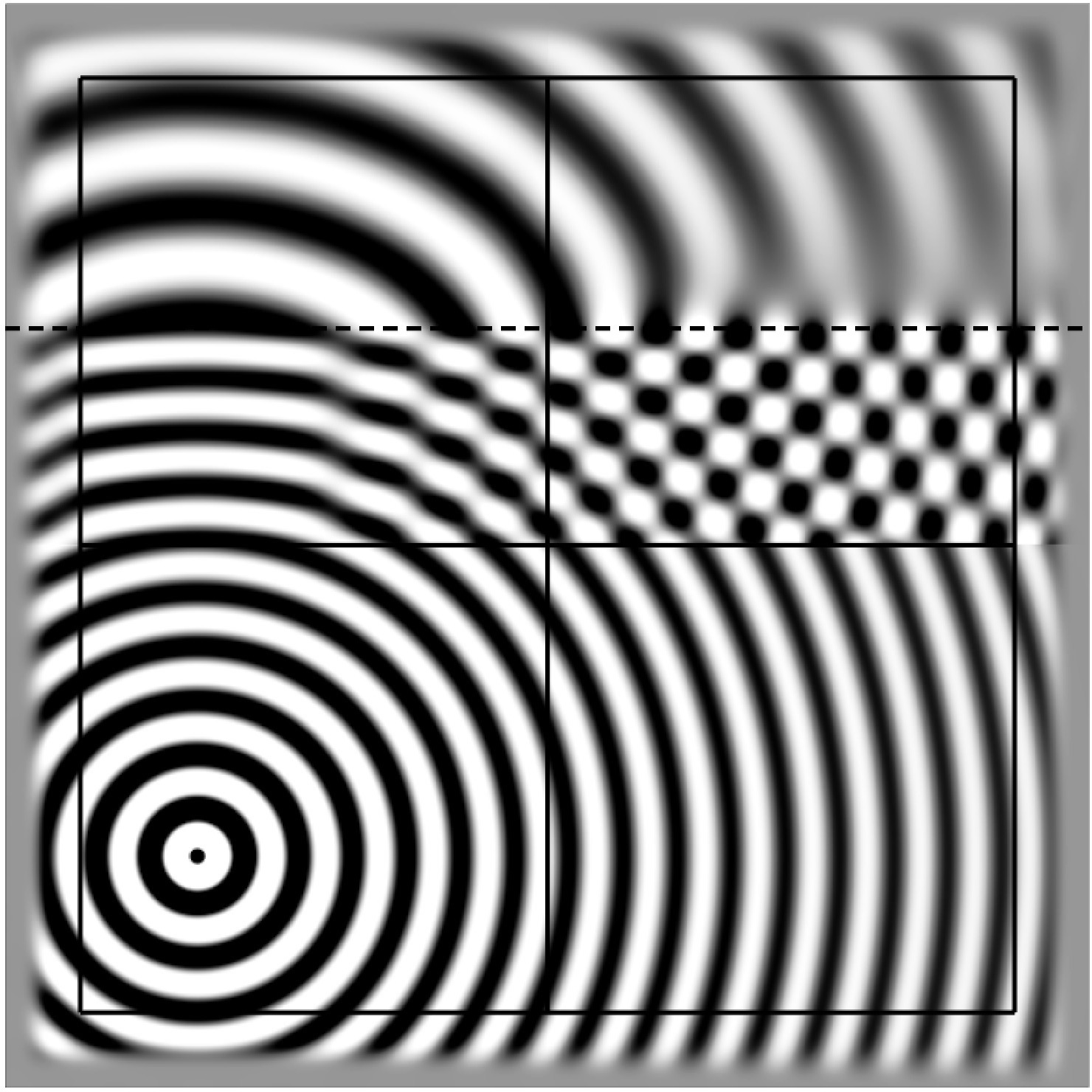}\\
                (f) $U^{(1)} = U^{(2)}$
        \end{minipage}
        \begin{minipage}[t]{\wdff\linewidth}    
                \centering
                \includegraphics[width=0.9\textwidth]{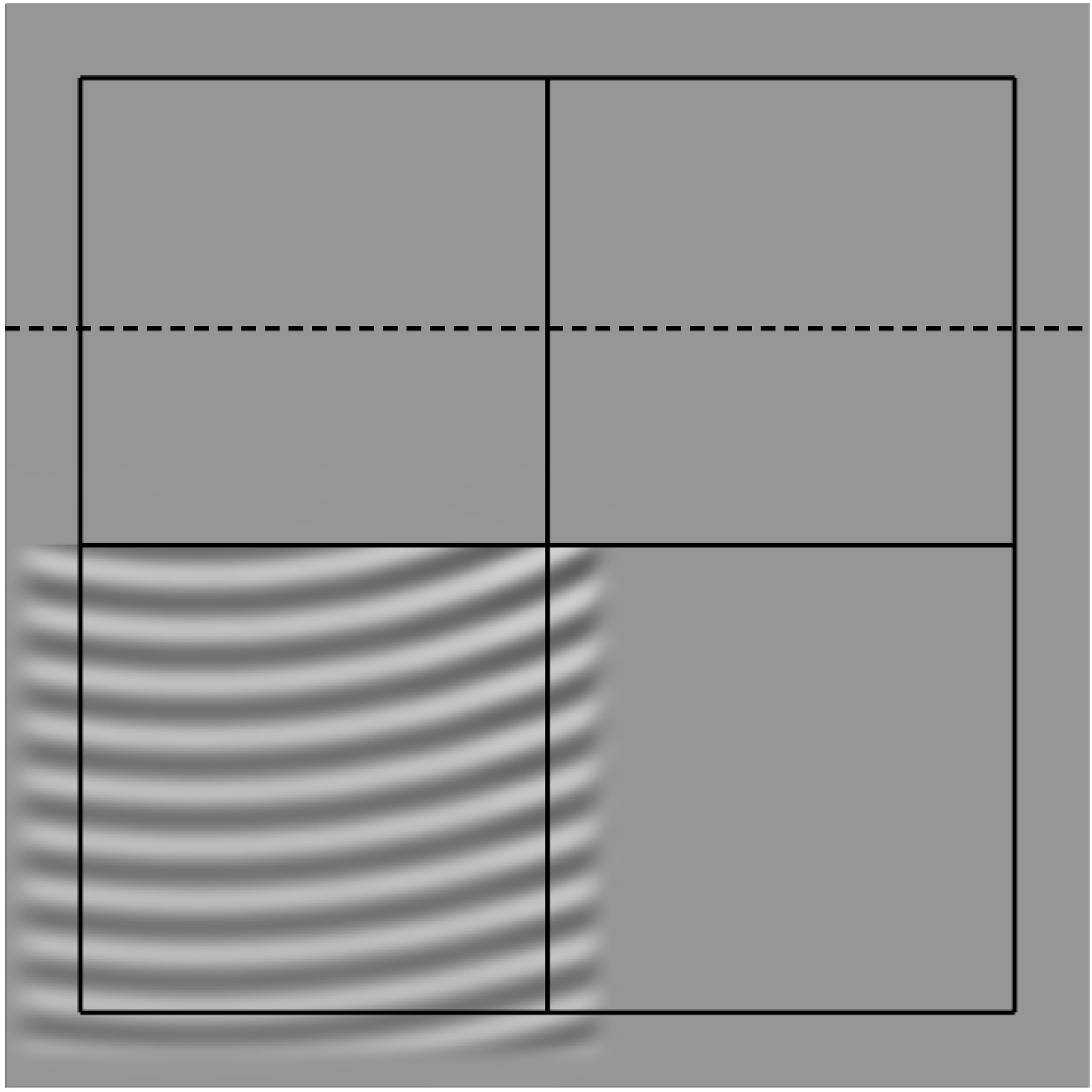}\\
                (g) $u_{1,1}^{3} $
        \end{minipage}
        \begin{minipage}[t]{\wdff\linewidth}
                \centering
                \includegraphics[width=0.9\textwidth]{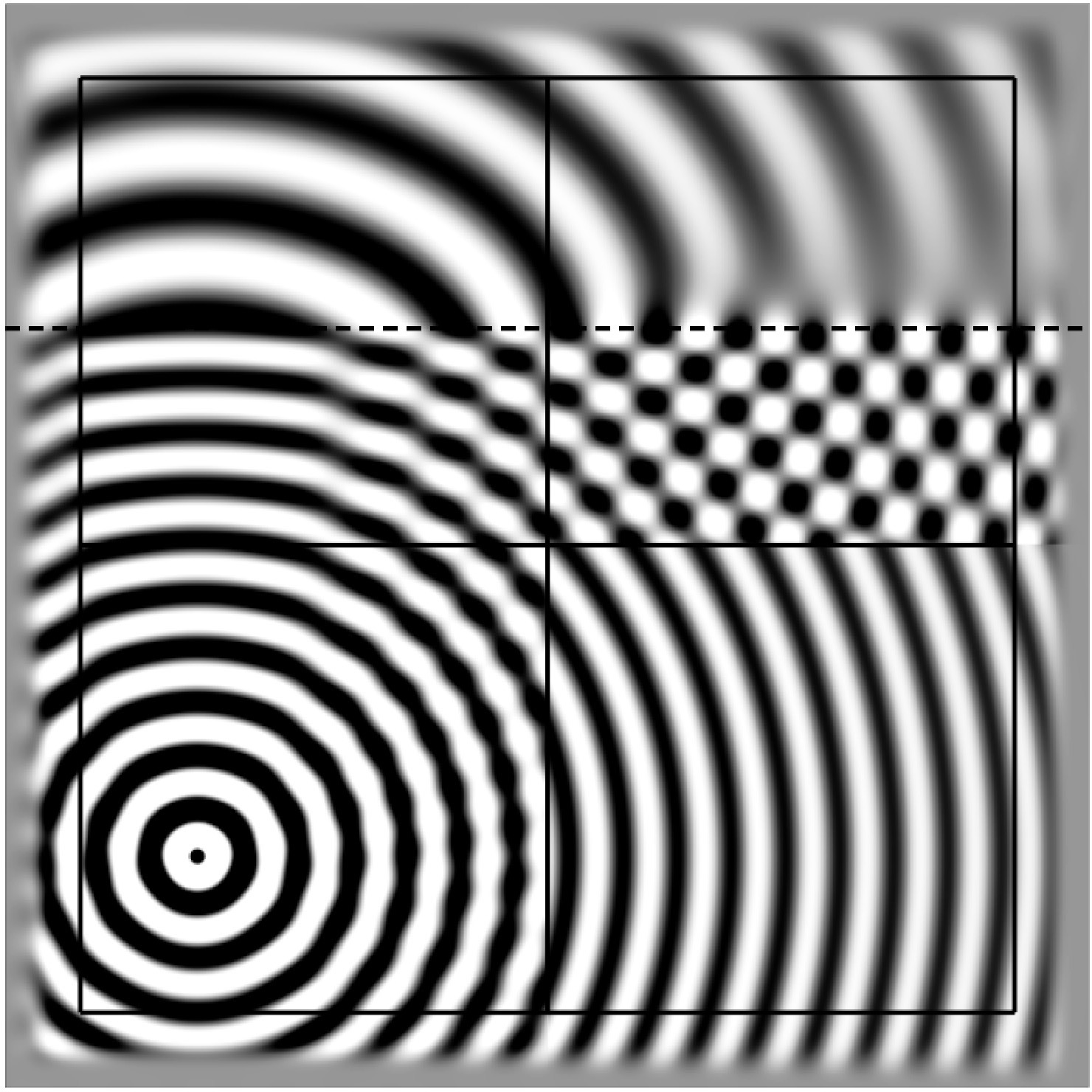}\\
                (h) $U^{(3)}_2$
        \end{minipage}          
        \vspace{0.2cm}
        
        \begin{minipage}[t]{\wdff\linewidth}
                \centering
                \includegraphics[width=0.9\textwidth]{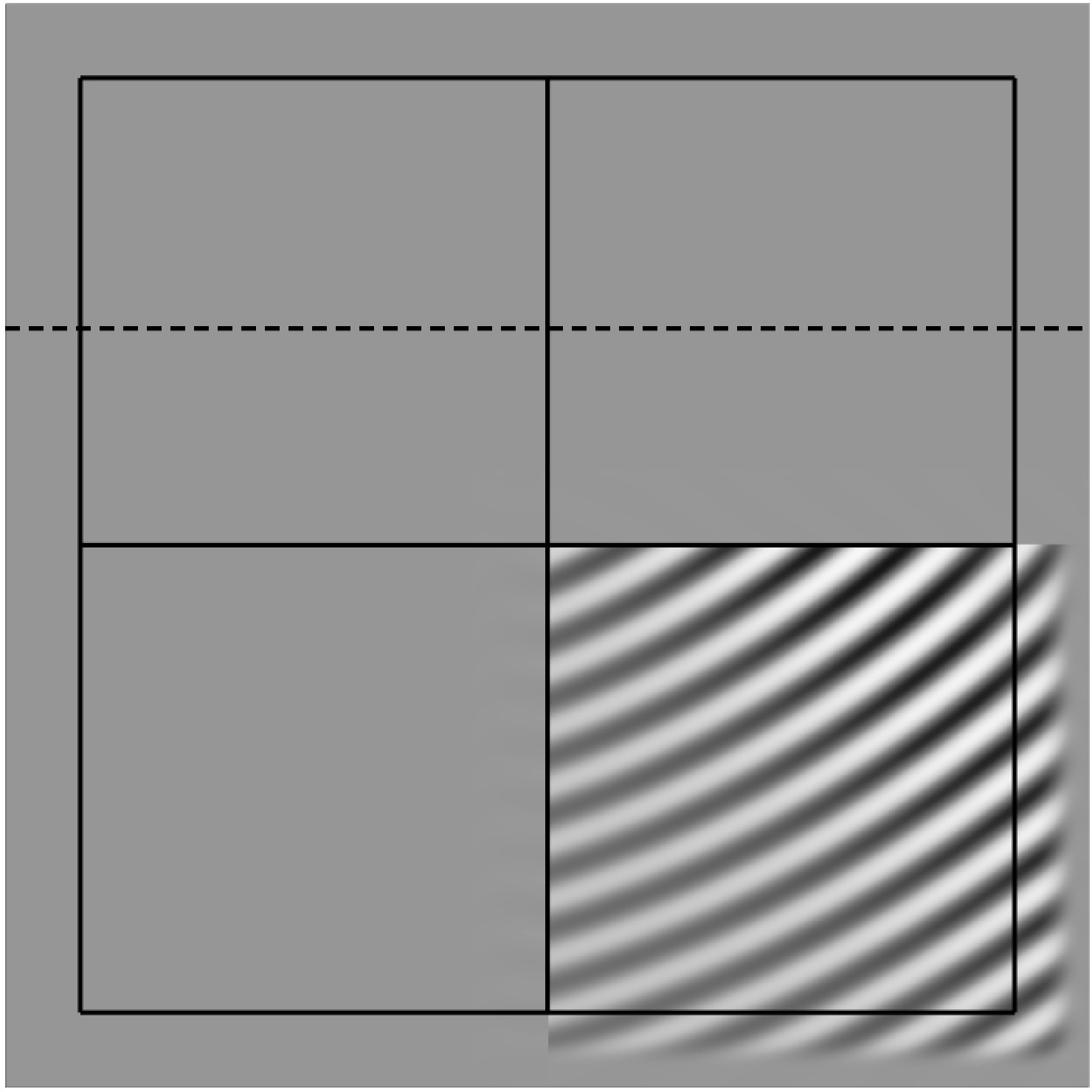}\\
                (i) $u_{2,1}^{3}$
        \end{minipage}
        \begin{minipage}[t]{\wdff\linewidth}
                \centering
                \includegraphics[width=0.9\textwidth]{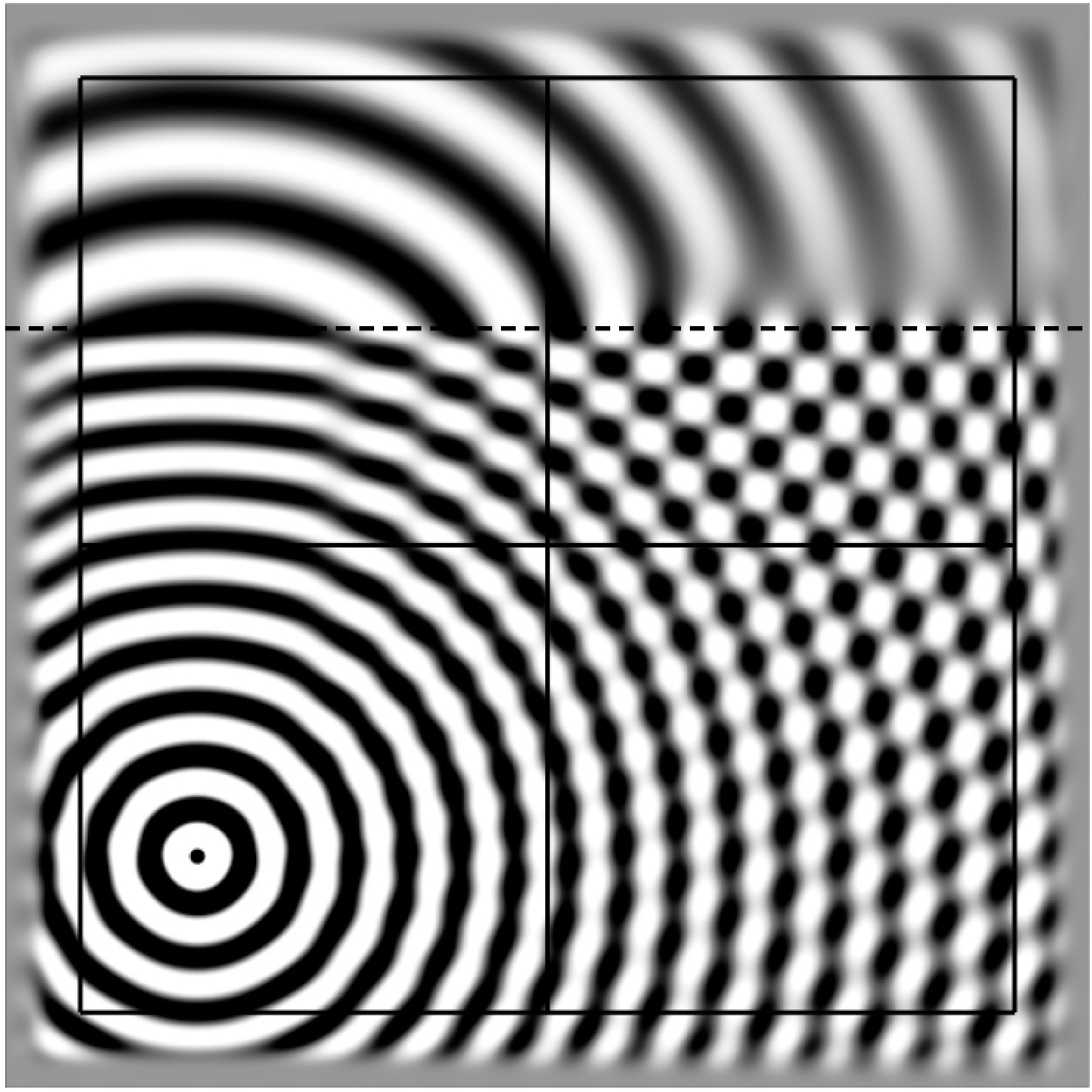}\\
                (j) $U^{(3)} =U^{(4)} = u $
        \end{minipage}
              
        \caption{{Illustration of the sweeping process of Algorithm  \ref{alg:diag2D_t}  in the two-layered media case with the $2 \times 2 $ domain partition, $\eta_L>0$ and $\supp(f) \subset \Omega_{1,1}$. $U^{(l)}$ ($l = 1,2,3,4$) denotes the DDM solution after sweep $l$, and $U^{(l)}_{i}$ denotes the DDM solution at the $i$-th step of sweep $l$.  
                } \label{fig:sf22r}}
\end{figure*}

In the \textbf{first sweep} of direction $(+1,+1)$, at step 1 the subdomain problem of $\Omega_{1,1} $ is solved with the source $f$, the solution $u_{1,1}^{1}$ is obtained, as shown in Figure \ref{fig:sf22r}-(a).
At step 2, the subdomain problem of $\Omega_{2,1} $ is solved with the transferred trace from $\Omega_{1,1}$, the solution $u_{2,1}^{ 1 }$ is obtained, as shown in Figure \ref{fig:sf22r}-(b). By applying the horizontal trace transfer on $\Omega_{1,2; 1,1}$, we can show that
\begin{align} \label{eq:2x2_u21}
u_{2,1}^{ 1 } =
\left\{
\begin{array}{ll}
u_{\abbrPred},   & \text{in} \,\, \quadFour,  \\
0,   & \text{in} \,\, \quadThree,
\end{array}
\right. 
\end{align}
thus one upward transferred traces are generated from $\Omega_{2,1}$. In addition, the subdomain problem in $\Omega_{1,2} $ is solved with the transferred trace from $\Omega_{1,1}$, the solution $u_{1,2}^{ 1 }$ is obtained, as shown in Figure \ref{fig:sf22r}-(c), and by applying the vertical  trace transfer for the two-layered media on $\Omega_{1,1; 1,2}$, we have that
\begin{align} \label{eq:2x2_u12}
u_{1,2}^{ 1 } =
\left\{
\begin{array}{ll}
u,   & \text{in} \,\, \quadTwo, \\
\widetilde{u}_{\abbrRefl},   & \text{in} \,\, \quadThree.
\end{array}
\right. 
\end{align}
%$u_{2,1}^{ 1 } =  u$ in $\Omega_{1,1;1,2}$. 
This implies that two transferred traces are generated from $\Omega_{1,2}$, one is rightwards and the other is downwards. %which is the reflection of the medium.
At step 3, the subdomain problem in $\Omega_{2,2} $ is solved with the transferred traces from $\Omega_{1,2}$ and $\Omega_{2,1}$ and the solution $u_{2,2}^{ 1 }$ is obtained, as shown in Figure \ref{fig:sf22r}-(d). We will show that $u_{2,2}^{ 1 } = u$ in $\Omega_{2,2}$.
By \eqref{eq:2x2_u21} and \eqref{eq:2x2_u12}, it is easy to see that
\begin{align} \label{eq:2x2_u22}
u_{2,2}^{ 1 } =&\; \Psi_{-1,0; 2,2} (u_{1,2}^{ 1 }) +  \Psi_{0,-1; 2,2} (u_{2,1}^{ 1 }) \nonumber \\
=&\; \Psi_{-1,0; 2,2} (u \chi_{2}^{+} + \widetilde{u}_{\abbrRefl} \chi_y^{-}  ) +  \Psi_{0,-1; 2,2} (u_{\abbrPred} \chi_{1}^{+}) \nonumber \\
=&\; \Big( \Psi_{-1,0; 2,2} (u \chi_{2}^{+}) +  \Psi_{0,-1; 2,2} (u \chi_{1}^{+}) \Big) + \Big( \Psi_{-1,0;2,2} (\widetilde{u}_{\abbrRefl} \chi_{2}^{-}) - \Psi_{0,-1; 2,2} ((u-u_{\abbrPred}) \chi_{1}^{+}) \Big).
%=& u \chi_{1}^+ \chi_{2}^+ + \Big( \Psi_{+1,0; 1,2} (\widetilde{u}_{\abbrRefl} \chi_{2}^{-}) + \Psi_{0,-1; 2,2} ((u-u_{\abbrPred}) \chi_{1}^{+}) \Big) \label{eq:2x2_u22}\\
\end{align}
By Lemma \ref{lemma:change_domain} and  applying the trace transfer on $\Omega_{2,2} $,  we have 
\begin{align} \label{eq:2x2_u22_1st}
\Psi_{-1,0; 2,2} (u \chi_{2}^{+}) +\Psi_{0,-1; 2,2} (u \chi_{1}^{+})  = u \chi_{1}^{+} \chi_{2}^{+}. % \qquad \text{in} \,\, \Omega_{}.
\end{align}
and then by using the fact that  $u_{\abbrRefl}+u_{\abbrPred} = u$  on $\gamma_1$ and Lemma \ref{lemma:change_domain}, it holds
\begin{align} \label{eq:2x2_u22_2nd}
\Psi_{-1,0; 2,2} (\widetilde{u}_{\abbrRefl} \chi_{2}^{-}) - \Psi_{0,-1; 2,2} ((u-u_{\abbrPred}) \chi_{1}^{+}) 
% =&\Psi_{-1,0; 2,1} (\widetilde{u}_{\abbrRefl} \chi_{2}^{-}) + \Psi_{0,+1; 2,1} (\widetilde{u}_{\abbrRefl} \chi_{1}^{+}) \nonumber\\
=\widetilde{u}_{\abbrRefl} \chi_{1}^{+} \chi_{2}^{-}.
\end{align}
%in $\Omega / \gamma$, 
Therefore, with \eqref{eq:2x2_u22_1st} and \eqref{eq:2x2_u22_2nd} we obtain
\begin{align} \label{eq:2x2_u22_final}
u_{2,2}^{ 1 } = u \chi_{1}^{+} \chi_{2}^{+} + \widetilde{u}_{\abbrRefl} \chi_{1}^{+} \chi_{2}^{-},  %\qquad \text{in} \,\, \Omega / \gamma
\end{align}
which implies that the exact solution in $\Omega_{2,2}$ is obtained, as shown in Figure \ref{fig:sf22r}-(e), 
and a downward transferred trace is generated on $\gamma_1$.
After the first sweep, the subdomain solutions of $\Omega_{1,2}$ and $\Omega_{2,2}$ are exact, while the reflection waves are missing in the subdomain solutions of $\Omega_{1,1}$ and $\Omega_{2,1}$, as shown in Figure \ref{fig:sf22r}-(f).
Besides, two downwards transferred traces from $\Omega_{1,2}$ and $\Omega_{2,2}$ are generated and will be used in the remaining sweeps.

In the \textbf{second sweep} of direction $(+1,-1)$, all the subdomain sources and solutions are zero  since the two transferred traces are not used in this sweep due to  Rule  \ref{rule2d_a} (similar directions).
In the \textbf{third sweep} of direction  $(-1,+1)$, the subdomain solution is zero in step 1. Then at step 2, the subdomain problem of $\Omega_{1,1}$ with the transferred trace from $\Omega_{1,2}$ is solved, and the solution $u_{1,1}^{ 3 }$ is obtained. By the previous arguments for the two-layered subdomains, we have 
\begin{align} \label{eq:2x2_u11_R}
u_{1,1}^{ 3 } =
\left\{   
\begin{array}{ll}
u_{\abbrRefl},     & \text{in} \,\, \quadThree, \\
%u_{\abbrLeft}, & \text{in} \,\, \quadFour \\
0, & \text{in} \,\, \OmegaYplus,
\end{array}
\right.
\end{align}   
as shown in Figure \ref{fig:sf22r}-(g).
This implies that the missing reflections in $\Omega_{1,1}$ is recovered, as shown in Figure \ref{fig:sf22r}-(h),   and a rightward transferred trace is generated.
At step 3, the subdomain problem of $\Omega_{2,1}$ with the transferred traces from $\Omega_{1,1}$ and $\Omega_{2,2}$ is solved, and the solution $u_{2,1}^{3}$ is obtained.  By \eqref{eq:2x2_u22_final}-\eqref{eq:2x2_u11_R} and Lemma \ref{lemma:change_domain}, we obtain that
\begin{align} \label{eq:2x2_u21_final}
u_{2,1}^{ 3 }  = &\;\Psi_{-1,0; 2,1} (u_{1,1}^{ 3 }) + \Psi_{0,+1; 2,1}
(u_{2,2}^{ 1 }) \nonumber \\
=&\;\Psi_{-1,0; 2,1} (u_{\abbrRefl} \chi_{2}^{-}) + \Psi_{0,+1; 2,1}
(u_{\abbrRefl} \chi_{1}^{+}) \nonumber\\
=&\; u_{\abbrRefl} \chi_{1}^{+} \chi_{2}^{-},
\end{align}
%in $\Omega/\gamma$, 
thus the missing reflection in $\Omega_{2,1}$ is obtained, as shown in Figure \ref{fig:sf22r}-(i), and no more transferred traces are generated.
All subdomain solutions in the  \textbf{fourth sweep} are zero since there are no transferred traces and source, and at last the total exact solution is constructed in $\Omega$, as shown in Figure \ref{fig:sf22r}-(j).
% Thus the verification of the exactness of the solution is finished.
\end{proof}

% eta_y < 0
Next we consider the case of horizontal media interface and $\eta_L < 0$. Similar to Lemma \ref{lemma:1x2_down} for the $1\times 2$ partition, when the source lies in the subdomain above the media interface, one extra round of four (i.e., $2^2 = 4$ in $\R^2$) diagonal sweeps is needed to produce the exact solution, and  we have the following Lemma.
\begin{lemmaa} \label{lemma:2x2_down}
	Assume that $f$ is a bounded smooth source, the media interface is horizontal and $\eta_L < 0$, and the domain $\Omega$ is decomposed into $2 \times 2$ subdomains. Then the DDM solution generated by Algorithm \ref{alg:diag2D_t} with one extra round of four diagonal sweeps is the exact solution to $\P_{\Omega,\kappa}$ with the source $f$ in the two-layered media case.     
\end{lemmaa}

% eta_x > 0 | eta_x < 0
In the case of  vertical media interface, similar analysis as Lemma \ref{lemma:2x2_up} and Lemma \ref{lemma:2x2_down}
for the horizontal media interface could be carried out,
and the required number of diagonal sweeps to obtain the exact solution follows a similar rule, that is,
 four diagonal sweeps are enough if no source lies in any of the subdomains in the right of the media interface, and eight sweeps are needed otherwise.
We have the following Lemma.
\begin{lemmaa} \label{lemma:2x2_left_right}
	Assume that $f$ is a bounded smooth source, the media interface is vertical, and the domain $\Omega$ is decomposed into $2 \times 2$ subdomains. Then the DDM solution generated by  Algorithm \ref{alg:diag2D_t} with one extra round of four diagonal sweeps is the exact solution to $\P_{\Omega,\kappa}$ with the source $f$ in the two-layered media case.     
\end{lemmaa}

Using the similar arguments of the Lemma \ref{lemma:single_src} for the constant medium case, we can further  extend the above analysis from the $2\times 2$ partition to the general $N_1 \times N_2$  checkerboard partition.
Again,  similar to the $1 \times N_2$ strip partition, the position of the media interface affects the number of sweeps needed to construct the exact solution, and the following result can be obtained. 
\begin{thm} \label{lemma:NxN}
	Assume that $f$ is a bounded smooth source and the domain $\Omega$ is decomposed into $N_1\times N_2$ subdomains, then the DDM solution generated by  Algorithm \ref{alg:diag2D_t} with one extra round of four diagonal sweeps is the exact solution to $\P_{\Omega,\kappa}$ with the source $f$ in the two-layered media case.        
\end{thm}

%%%%%%%%%%%%%%%%%%%%%%%%%%%%%%%%%%%%%%%%%%%%%%%%%%%%%%%%%%%%%%%%
\section{Extension to the Helmholtz problem in $\R^3$}
%%%%%%%%%%%%%%%%%%%%%%%%%%%%%%%%%%%%%%%%%%%%%%%%%%%%%%%%%%%%%%%%

%-------------------------------------------------------------------------
% domain and DDM 

We now consider the extension the trace transfer-based diagonal sweeping DDM to the Helmholtz  problem \eqref{eq:helm}-\eqref{eq:radiation} in $\R^3$. The PML equation in $\R^3$ associated with  the cuboidal box 
$\B = \{\bx=(x_1,x_2,x_3) \;|\; a_j \le x_j \le b_j, j = 1,2,3\}$ is defined as
\begin{equation} \label{eq:PML3}
J_{\B}^{-1} \nabla \cdot (A_{\B} \nabla\tilde{u}) + \kappa^2 \tilde{u} = f, \qquad \text{in} \,\, \R^3,
\end{equation}
where 
$$\DD A_{\B}(\bx) = \mbox{diag}\left(\frac{\a_2(x_2)\a_3(x_3)}{\a_1(x_1)}\right. , \frac{\a_1(x_1)\a_3(x_3)}{\a_2(x_2)}, \left. \frac{\a_1(x_1)\a_2(x_2)}{\a_3(x_3)}\right),\quad J_{\B}(\bx) = \a_1(x_1) \a_2(x_2)\a_3(x_3),$$ 
with $\alpha_3(x_3) = 1 + \ii \sigma_3(x_3) $ 
and $\sigma_3(x_3)$ being defined  in the same way as \eqref{eq:sigma}.

With the $z$-direction partition as $\zeta_k = -l_3 + (k-1) \Delta\zeta$ for $k = 1, \ldots, \Nbz+1 $, where $\Delta\zeta = 2 \l_3 / \Nbz$, the cuboidal domain $\Omega = [-l_1,l_1]\times[-l_2,l_2]\times[-l_3,l_3]$ in $\R^3$ is partitioned into $\Nbx \times \Nby \times \Nbz$ nonoverlapping subdomains $\Omega_{i,j,k}$ and the source $f$ is decomposed into $f_{i,j,k} = f\cdot \chi_{\Omega_{i,j,k}}$, $i = 1, 2, \ldots, \Nbx$, $ j = 1, 2, \ldots, \Nby$, $k = 1, 2, \ldots, \Nbz$.
%
%
%
%-------------------------------------------------------------------------
% notations
%
% In 3D,  the same notations as above for $x$ and $y$ components are used, while notations for for $z$ component are added.
Let the one-dimensional cutoff functions in the $z$-direction be defined as
\begin{align*}
\mu^{(3)}_{\ocircle, k}(x_1) &= \left\{
\begin{array}{ll}
\DD H( x_3 - \zeta_k),  & \, \ocircle = -1\,\text{and}\,k\ne1,\\
\DD H(\zeta_{k+1} - x_3 ),  & \, \ocircle = 1\,\text{and}\,k\ne N_3,\\
  1,                         & \, \text{otherwise}, 
\end{array}
\right. 
\end{align*}
and the cutoff function for the  subdomain $\Omega_{i,j,k}$ is given by
\begin{align}
\betaH_{i, j, k} =  \betaH^{(1)}_{-1, i}(x_1) \betaH^{(1)}_{+1, i}(x_1) \betaH^{(2)}_{-1, j}(x_2)  \betaH^{(2)}_{+1, j}(x_2)  \betaH^{(3)}_{-1, k}(x_3)  \betaH^{(3)}_{+1, k}(x_3). 
\end{align}

Denote by $G_{i,j,k}(\bx, \by)$ the fundamental solution of the problem $\P_{\Omega_{i,j,k}, \kappa_{i,j,k}}$,
where $\kappa_{i,j,k}$ is the extension of the wave number $\kappa$ as done in \eqref{eq:wave_number},
and denote the boundaries 
of $\Omega_{i,j,k}$ as
\begin{align*}
\gamma_{\Box,\vartriangle,\ocircle; i, j, k} &= \left\{ 
\begin{array}{ll}
\{(x_1,x_2,x_3)\;|\; x_1 = \xi_i \minusd\},  & \, \text{if} \,\,  (\Box, \vartriangle, \ocircle) = (-1,0,0) \,\,\text{and}\,\, i > 1, \\
\{(x_1,x_2,x_3)\;|\; x_1 = \xi_{i+1} \plusd\},  & \, \text{if} \,\,  (\Box, \vartriangle, \ocircle) = (+1,0,0) \,\,\text{and}\,\, i < \Nbx,\\
\{(x_1,x_2,x_3)\;|\; x_2 = \eta_j \minusd\},  & \, \text{if} \,\, (\Box, \vartriangle, \ocircle) = (0,-1,0) \,\,\text{and}\,\, j > 1,\\
\{(x_1,x_2,x_3)\;|\; x_2 = \eta_{j+1} \plusd\},  & \, \text{if} \,\, (\Box, \vartriangle, \ocircle) = (0,+1,0)  \,\,\text{and}\,\, j < \Nby, \\
\{(x_1,x_2,x_3)\;|\; x_3 = \eta_k \minusd\},  & \, \text{if} \,\, (\Box, \vartriangle, \ocircle) = (0,0,-1) \,\,\text{and}\,\, k > 1,\\
\{(x_1,x_2,x_3)\;|\; x_3 = \eta_{k+1} \plusd\},  & \, \text{if} \,\, (\Box, \vartriangle, \ocircle) = (0,0,+1)  \,\,\text{and}\,\, k < \Nbz,  \\
\varnothing , &   \text{otherwise},           
\end{array}
\right. 
\end{align*}
and the corresponding unit normal vectors associated with them are then $\n_{\Box,\vartriangle,\ocircle} = -(\Box, \vartriangle, \ocircle)$. 
Using the above fundamental solutions, we can define the potential operators  as follows:
\begin{align}
  \Potential_{\Box, \vartriangle, \ocircle;\, i, j, k} (v) :=
  \Phi_{\Box, \vartriangle, \ocircle;\, i, j, k} (  (v, A_{\Omega_{i,j,k}} \nabla v \cdot \n_{\Box,\vartriangle,\ocircle} )^T   ),
\end{align}
where
\begin{align*}
  \Phi_{\Box, \vartriangle, \ocircle;\, i, j, k} (  (v, w)^T   ) :=
  &
    \DD  \int_{\gamma_{\Box, \vartriangle, \ocircle;\, i, j, k}}
    J^{-1}_{\Omega_{i,j,k}} G_{i,j,k}(\bx, \by) w(\by) \nonumber\\  %(A_{i,j,k} \nabla_y v(\by) \cdot \n_{\Box,\vartriangle,\ocircle}) \\
  &\qquad\qquad
    - v(\by) \Big( A_{\Omega_{i,j,k}} \nabla_y \left(J^{-1}_{\Omega_{i,j,k}} G_{i,j,k}(\bx, \by)\right) \cdot \n_{\Box,\vartriangle,\ocircle}   \Big) d\by,
\end{align*}
for  $(\Box, \vartriangle, \ocircle) = (\pm 1, 0,0), (0, \pm 1, 0), (0, 0, \pm 1)$.

% Sweeping order:
Eight diagonal directions, namely $(\dirThreeD)$, $\dirThreeD = \pm 1$, are used in diagonal sweeping in $\R^3$, and the sweep along each  direction contains a total of $N_1+N_2+N_3-2$ steps.  The sweeping order is set as
\begin{align} \label{Sorder3D}
\begin{array}{llll}
(+1,+1,+1),& (-1,+1,+1),& (+1,-1,+1),& (-1,-1,+1), \\
(+1,+1,-1),& (-1,+1,-1),& (+1,-1,-1),& (-1,-1,-1).
\end{array}
\end{align}
In the $s$-th step of the sweep of direction $(d_1, d_2, d_3)$, the group of subdomains $\{ \Omega_{i,j,k} \}$ satisfying
   \begin{equation}
     \label{eq:step_subd3}    
     \widehat{i}(i)+\widehat{j}(j) +\widehat{k}(k) +1 = s
   \end{equation}
    are to be solved, where $\widehat{i}(i)$ and $\widehat{j}(j)$ are defined by \eqref{eq:step_subd_ij} and $\widehat{k}(k)$ is defined as
\begin{align} \label{eq:step_subd_k}     
\widehat{k}(k) &= \left\{
\begin{array}{ll}
k-1,  &\,\,\, \text{if} \,\,\,\, d_3 = 1,\\
\Nbz-k, & \,\,\, \text{if} \,\,\,\, d_3 = -1.
\end{array}
\right. 
\end{align}
The similar direction of vectors in $\R^3$ is defined as in \cite{Leng2020}:
%which we include here for the reader's convenience.
% similar direction 
two vector $\d_1$ and $\d_2$ in $\R^3$ are called in the {\em similar
        direction} %\cite{Leng2020}  
if $\d_1 \cdot \d_2 > 0$ and $\d_1(k) \, \d_2(k) \geq 0$ for
$k = 1,2,3$, where $\d_1(k)$ and $\d_2(k)$ are the $k$-th components of
$\d_1$ and $\d_2$, respectively. 
% Rules
The rules on the transferred traces  for sweeps in $\R^3$ are also as those in \cite{Leng2020}:   
\begin{rulee}{\rm (Similar directions in $\R^3$)} \label{rule3d_a}
        A transferred source  which is not in the similar direction of one sweep in $\R^3$, should not be used in that sweep.
\end{rulee}
\begin{rulee}{\rm (Opposite directions in $\R^3$)}      \label{rule3d_b}
        Suppose a transferred source with direction $\d_{\text{src}}$ is 
        generated in one sweep with direction $\d_1$, then it should not be used in the later sweep with direction $\d_2$, if under any of $x-y$, $x-z$,
        $y-z$ plane projections, the projection of $\d_{\text{src}}$ has exactly one zero components and the projections of $\d_1$ and $\d_2$ are opposite.
\end{rulee}

The trace transfer-based diagonal sweeping DDM  in $\R^3$ is then stated as follows.

\alglongline
\begin{algorithm_}[Trace transfer-based diagonal sweeping DDM  in $\R^3$]
        $\quad$  
        \label{alg:diag3D_t}
        \begin{algorithmic}[1]
                \parState {Set the sweep order as list (\ref{Sorder3D}) } 
                \parState {Set the local trace of each subdomain for each sweep as  $\mathbf{g}^{l}_{\Box,\vartriangle,\ocircle; i, j, k}$, $l= 1,\ldots,8$ }
                \For{Sweep $l = 1,\ldots,8$} 
                \For{Step $s = 1, \ldots, \Nbx+\Nby+\Nbz-2$ } 
                \For {each of the subdomains $\{\Omega_{i,j,k}\}$ defined by \eqref{eq:step_subd3} in Step $s$ of Sweep $l$ }
                \parState {If $f_{i,j,k} \neq 0$, solve the local solution %$u_{i, j, k}^{l}$, 
                        $\L_{\Omega_{i,j,k},\kappa_{i,j,k}} (u_{i, j, k}^{l}) =  f_{i,j, k}$,\\
                else set $u_{i, j, k}^{l} = 0$}
                \State Set $f_{i,j, k} = 0$
                \State Add potentials to the local solution
                \begin{eqnarray*}
                u_{i, j, k}^{l} \gets u_{i, j, k}^{l} \sum\limits_{\substack{(\Box,\vartriangle,\ocircle) = (\pm 1, 0, 0), \\(0, \pm 1, 0), (0, 0, \pm 1)} }  \Potential_{\Box, \vartriangle, \ocircle;\, i, j, k} (\mathbf{g}^{l}_{\Box,\vartriangle,\ocircle; i, j, k})
                \end{eqnarray*}
                \parState {Compute new transferred trace 
                        $ \left( u_{i, j, k}^{l}, A_{i,j,k} \nabla u_{i, j, k}^{l} \cdot \n_{\Box,\vartriangle,\ocircle} \right)^T \big|_{\gamma_{\Box,\vartriangle,\ocircle; i, j, k}}$,
                        $(\Box,\vartriangle,\ocircle) = (\pm 1, 0, 0)$, $(0, \pm 1, 0)$, $(0, 0, \pm 1)$} 
                      \For {each direction $(\Box,\vartriangle,\ocircle) = (\pm 1, 0, 0)$, $(0, \pm 1, 0)$, $(0, 0, \pm 1)$}
                      
                \parState{Find the smallest sweep number $l' \geq l$, such that the transferred trace could be used in Sweep $l'$, according to Rules \ref{rule3d_a} and \ref{rule3d_b} } 
                \parState {Add the transferred trace to the $l'$-th trace of the corresponding  neighbor subdomain}
                \begin{align*}
                  \qquad\qquad\qquad\mathbf{g}^{l'}_{-\Box,-\vartriangle,-\ocircle; i+\Box, j+\vartriangle, k+\ocircle}
                  \gets& \mathbf{g}^{l'}_{-\Box,-\vartriangle,-\ocircle; i+\Box, j+\vartriangle, k+\ocircle} \\
                  & +\left( u_{i, j, k}^{l}, A_{\Omega_{i,j,k}} \nabla u_{i, j, k}^{l} \cdot \n_{\Box,\vartriangle,\ocircle}  \right)^T \Big|_{\gamma_{\Box,\vartriangle,\ocircle; i, j, k}}
                  %\mathrel{{+}{=}}
                \end{align*}
                \EndFor
                \EndFor
                \EndFor
                \EndFor
                \State The DDM solution for  $\P_{\Omega,\kappa}$ with the source $f$ is then given by
                $$ u_{\text{DDM}} = \sum\limits_{l = 1, \ldots, 8 } %\sum\limits_{}
                \sum\limits_{\substack{\rangeThreeDRow}} u_{i,j,k}^{l}  \mu_{i,j,k}  $$
                
        \end{algorithmic}
\end{algorithm_}
\alglongline   
\vspace{0.2cm}

The DDM solution of the above Algorithm \ref{alg:diag3D_t} is exactly the solution to $\P_{\Omega,\kappa}$ with the source $f$, not only in the constant medium case, but also in the two-layered media case with one extra round of eight diagonal sweeps.  The verification in the constant medium case is similar to the source transfer method discussed in \cite{Leng2020}, in fact, it is even simpler since only cardinal directions need to be considered for  transferred traces. 
The verification of the two-layered media case basically follows the same line of argument used in the  $\R^2$ case in the preceding section
except the media interface now becomes a plane in $\R^3$. 

\begin{rem}

  The diagonal sweeping DDM is a very efficient method in terms of computational complexity and weak scalability.
  When we increase the number of subdomains while fixing the subdomain problem size,  the computational complexity of solving one RHS 
  Helmholtz system by using a Krylov subspace solver with the diagonal sweeping DDM  as the preconditioner is $\mathcal{O}(N n_{\text{iter}})$, where $N$ is the size of global system and $n_{\text{iter}}$ is the number of iterations used by the Krylov subspace solver.
If $n_{\text{iter}}$ grows proportionally to $\log N$, which is to be verified in the numerical experiments section, then the method is of $\mathcal{O}(N \log N)$ complexity.
For problems with many RHSs,  a pipeline can be set up to solve all the RHSs in sequence as discussed  in \cite{Leng2020},
where the number of used processors in the pipeline is chosen to be equal to the number of subdomains.
Take the 3D case for illustration, assuming that the number of RHSs, denoted $N_{\text{RHS}}$, is a multiple of
$\Nbx+\Nby+\Nbz-2$, then the average solving time for one RHS in the method is about
\begin{equation}
  8 n_{\text{iter}} T_0 + \frac{\Nbx+\Nby+\Nbz-2}{N_{\text{RHS}}}  T_0,
\end{equation}
where $T_0$ is the time cost of one subdomain solve. Thus, the pipeline is very efficient for this case in the sense that the idle of the cores in the pipeline processing only wastes a tiny amount (i.e., $\frac{N_1+N_2+N_3-2}{8 n_{\text{iter}}N_{\text{RHS}}}$) of the total computing resource.
%The weak scalability of the method for large-scale Helmholtz problem with large wave number is related to the number of Krylov iteration $n_{\text{iter}}$, which is expected to grow proportionally to $\log N$, and investigated numerically in the next section.

\end{rem}

%%%%%%%%%%%%%%%%%%%%%%%%%%%%%%%%%%%%%%%%%%%%%%%%%%%%%%%%%%%%%
\section{Numerical experiments}
%%%%%%%%%%%%%%%%%%%%%%%%%%%%%%%%%%%%%%%%%%%%%%%%%%%%%%%%%%%%%

Performance evaluation of the trace transfer-based diagonal sweeping DDM (Algorithms \ref{alg:diag2D_t}  and \ref{alg:diag3D_t}) is presented in this section. 
The finite difference  (FD) discretization on structured meshes is adopted for spatial discretization. More specifically, the five points stencil and the seven points one are employed for two and three dimensional problems respectively, both of them are of second order accurate. 
The accuracy of the FD scheme is ensured by the mesh density, which is defined as the number of nodes per wave length.
The proposed diagonal sweeping DDM algorithms are implemented in parallel using Message Passing Interface (MPI), and the spare direct solver ``MUMPS'' \cite{MUMPS} is used to solve the local subdomain problems. 
The numerical tests are all carried out on the  ``LSSC-IV'' cluster at the State Key Laboratory of Scientific and Engineering Computing, Chinese Academy of Sciences, each computing node of the cluster has two 2.3GHz Xeon Gold 6140 processors with 18 cores and 192G shared memory. The number of used cores is always chosen to be equal to the number of subdomains. 

First, the convergence of the fully discrete DDM solution is tested in the following sense. The total error of a numerical solution usually consists of the errors from the discretization, the truncation of the PML layer and the DDM algorithm. However, in the constant medium case and the two-layered media case (with an extra round of diagonal sweeps), there is no error from the proposed DDM itself as shown by convergence analysis. Therefore, using appropriate PML medium parameters to keep the PML truncation error relatively small (which is easy since the truncation error decays exponentially away from the computational domain), the total error will be dominated by the error of numerical discretization and then can be identified by the convergence order.
Second, the performance and parallel scalability of the proposed DDM as the preconditioner is also tested.  The DDM solution is an approximate solution to the problem, hence it could be used as a good preconditioner for Krylov subspace solvers (such as GMRES), and the effectiveness of the DDM preconditioner is
reflected  by the number of  iterations needed to reach certain stopping criterion. All running times are measured in seconds.

%%%%%%%%%%%%%%%%%%%%%%%%%%%%%%%%
\subsection{Convergence tests of the discrete DDM solutions}
%%%%%%%%%%%%%%%%%%%%%%%%%%%%%%%%

In this subsection, the convergence of the proposed method is tested in the constant medium case and the two-layered media case.  The error of the DDM solution is expected to be dominated by the spatial discretization error, hence we fix both the wave number and the number of subdomains, and increase the mesh resolution to check the convergence order.

%%%%%%%%%%%%%%%%%%%%%%%%%%%%%%%%
\subsubsection{2D constant medium problem}
%%%%%%%%%%%%%%%%%%%%%%%%%%%%%%%%

A 2D constant medium problem  is used to test Algorithm \ref{alg:diag2D_t}.  The computational domain is $B_L=[0,2L]\times [0,2L]$ with $L = 1/2$, and the interior domain without PML is $B_l=[L-l,L+l]\times [L-l,L+l]$ with $l = 25/56$.
For the wave number $\kappa = 50 \pi $,  a series of uniformly refined meshes are used,   of which the mesh density increases approximately from $11$ to $173$ . 
The domain partition is fixed to be $4 \times 4$  for all the meshes, and the source is defined as
\begin{equation} \label{eq:smooth_src_2d}
f(x_1, x_2) = \frac{16 \kappa^2}{\pi^3}  e^{- (\frac{4 \kappa}{\pi})^2 ((x_1-r_1)^2+(x_2-r_2)^2) },
\end{equation}
where $(r_1, r_2) = (0.28, 0.27) $, such that $\supp(f)$ extends to all four subdomains $\Omega_{i,j}$, $i,j=1,2$.
 The optimal second-order convergence of the errors in $L^2$ and $H^1$ norms along the mesh refinement are observed, as is reported in Table \ref{tab:const2D}, which implies the total error is indeed dominated by the finite difference discretization error in the constant medium case.
%
%The results of the algorithm with trace and source transfer are very close to each other, the differences in error are in the fourth digits, hence they %are not shown in the table.

\begin{table}[ht!]
        \centering
        \begin{tabular}{|c|c|r|c|r|c|}
                \hline
                Mesh     & Local Size  & \mrowcol{$L^2$ Error} & Conv. & \mrowcol{$H^1$  Error} & Conv. \\
                Size     & without PML &                       & Rate  &                         & Rate  \\
                \hline\hline
                270$^2$  & 60$^2$      & 1.38$\times 10^{-2}$  &    -   & 2.13$\times 10^{0}\,\,\,\,$     &  -     \\
                540$^2$  & 120$^2$     & 3.43$\times 10^{-3}$  & 2.0   & 5.35$\times 10^{-1}$    & 2.0      \\
                1080$^2$ & 240$^2$     & 8.51$\times 10^{-4}$  & 2.0   & 1.33$\times 10^{-1}$    & 2.0   \\
                2160$^2$ & 480$^2$     & 2.12$\times 10^{-4}$  & 2.0   & 3.33$\times 10^{-2}$    & 2.0   \\
                4320$^2$ & 960$^2$     & 5.30$\times 10^{-5}$  & 2.0   & 8.32$\times 10^{-3}$    & 2.0   \\
                \hline
        \end{tabular}
        \caption{The errors and convergence rates of numerical solutions by the   trace transfer-based diagonal sweeping  DDM for the constant medium problem in $\R^2$.} \label{tab:const2D}
\end{table}

%%%%%%%%%%%%%%%%%%%%%%%%%%%%%%%%
\subsubsection{2D  two-layered media problem}
%%%%%%%%%%%%%%%%%%%%%%%%%%%%%%%%

Next we test Algorithm \ref{alg:diag2D_t} with a two-layered problem in
$\R^2$, the problem settings are the same as the first example, except that the
wave number is changed to two-layered as follows:
\begin{align*}
  \kappa = \left\{
  \begin{array}{ll}
    \kup, & x_2 > \eta_L + \eps, \\
    \kdown (1-w)  + \kup w, & \eta_L \leq x_2 \leq \eta_L +\eps,  \\
    \kdown,  &   x_2 < \eta_L,
  \end{array}
  \right.
\end{align*}
where $\kdown = 50 \pi $, $\kup = 62.5 \pi$, $\eta_L = 0.65$ and $\eps = 0.01$, here a thin layer of width $\eps$ is used to make a smooth transition between the two media, as shown in Figure \ref{fig:two-layered-sol}-(left).
The same single source as \eqref{eq:smooth_src_2d} is used to test  the original algorithm, and additionally the combination of four symmetric sources, defined as 
\begin{equation*}
{ \tilde f}(x_1, x_2) = f(x_1, x_2) + f(1-x_1, x_2) + f(x_1, 1-x_2) + f(1-x_1, 1-x_2),
\end{equation*}
where $f(x_1, x_2)$ is defined as \eqref{eq:smooth_src_2d}, is used to test the algorithm with extra four diagonal sweeps.
The domain partition is again fixed to be $4 \times 4$  for all the meshes. Note that the DDM solution  with the $2\times2$ domain partition is still guaranteed to be the exact solution as analyzed in Section 3.2.
%such that $\supp(f)$ extends to all four subdomains $\Omega_{i,j}$, $i,j=1,2$.
The real part of the discrete DDM solution on the mesh of 1080$^2$ is shown in Figure \ref{fig:two-layered-sol}-(middle)\&(right), the optimal second-order convergence of the errors in $L^2$ and $H^1$ norms along the mesh refinement are again obtained, as shown in Tables \ref{tab:two-layered} and \ref{tab:two-layered4}, which implies the total error is still  dominated by the  finite difference discretization error in the two-layered media case.
%
%Again the results of the algorithm with trace and source transfer are very close to each other and not shown in the table.

\def\wdff{0.32}
\begin{figure*}[!ht]    
  \centering
  \begin{minipage}[t]{\wdff\linewidth}
    \centering
    \includegraphics[width=1.05\textwidth]{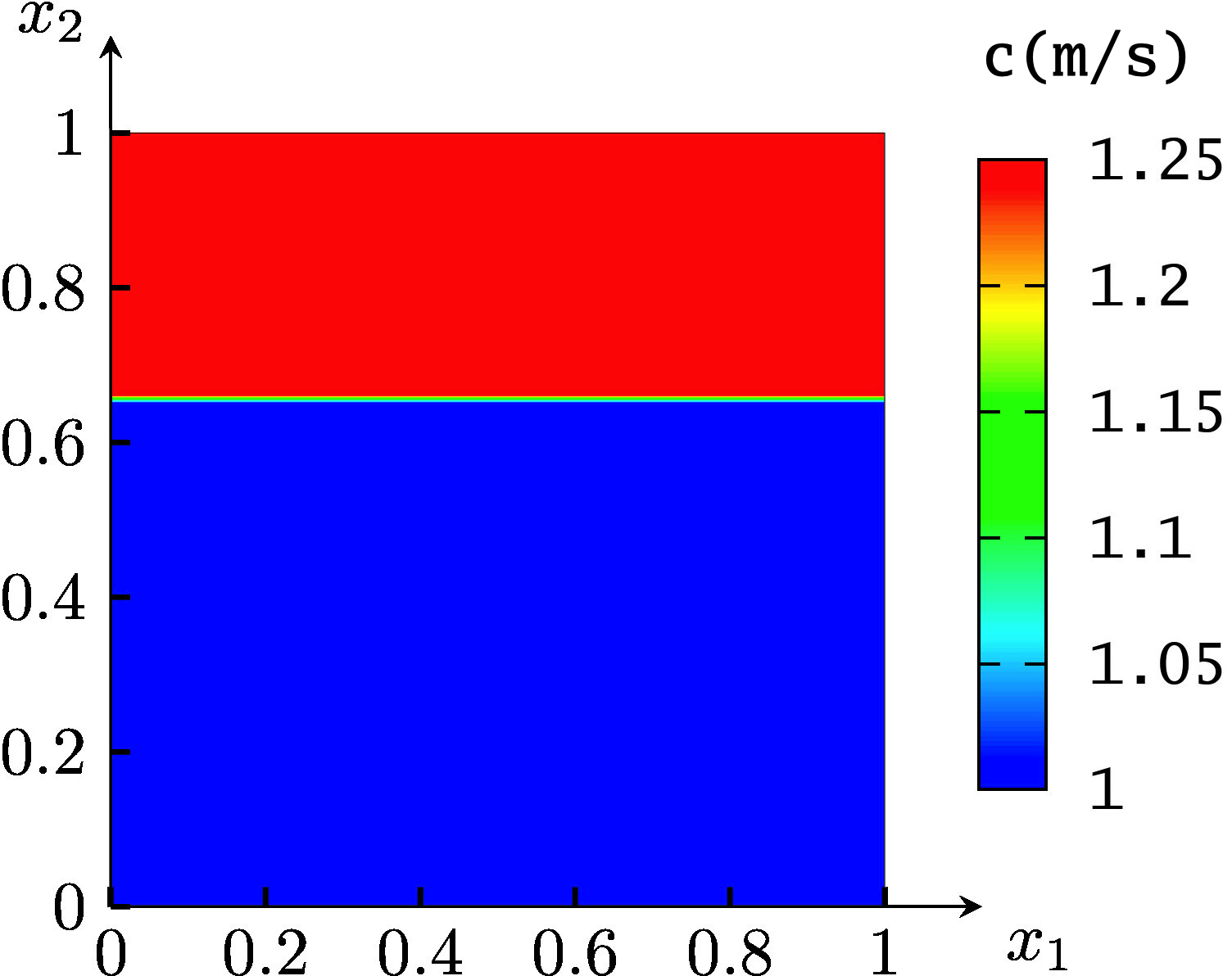}
  \end{minipage}\hspace{0.3cm}
    \begin{minipage}[t]{\wdff\linewidth}
    \centering
    \includegraphics[width=0.98\textwidth]{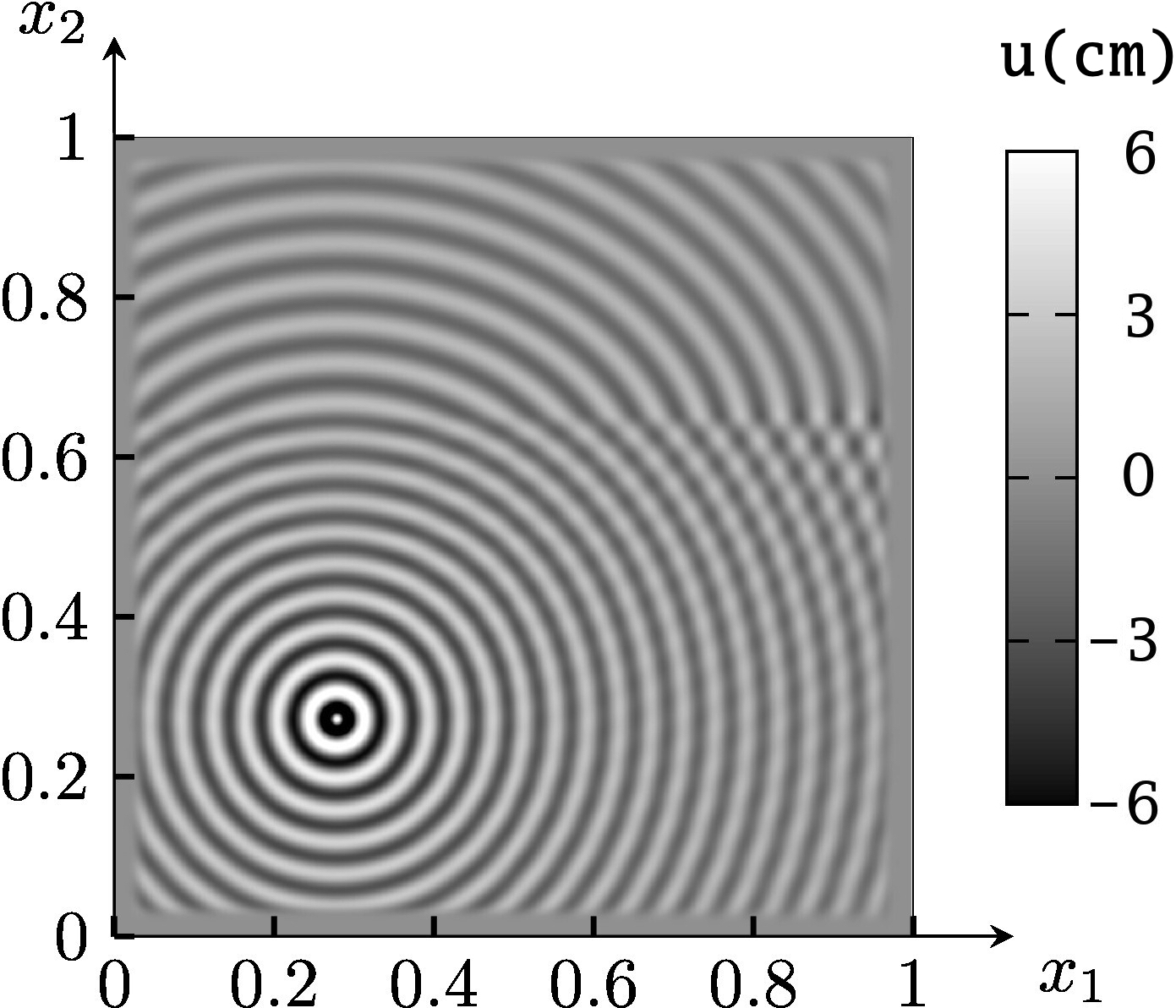}
  \end{minipage}\hspace{-0.1cm}
  \begin{minipage}[t]{\wdff\linewidth}
    \centering
    \includegraphics[width=0.98\textwidth]{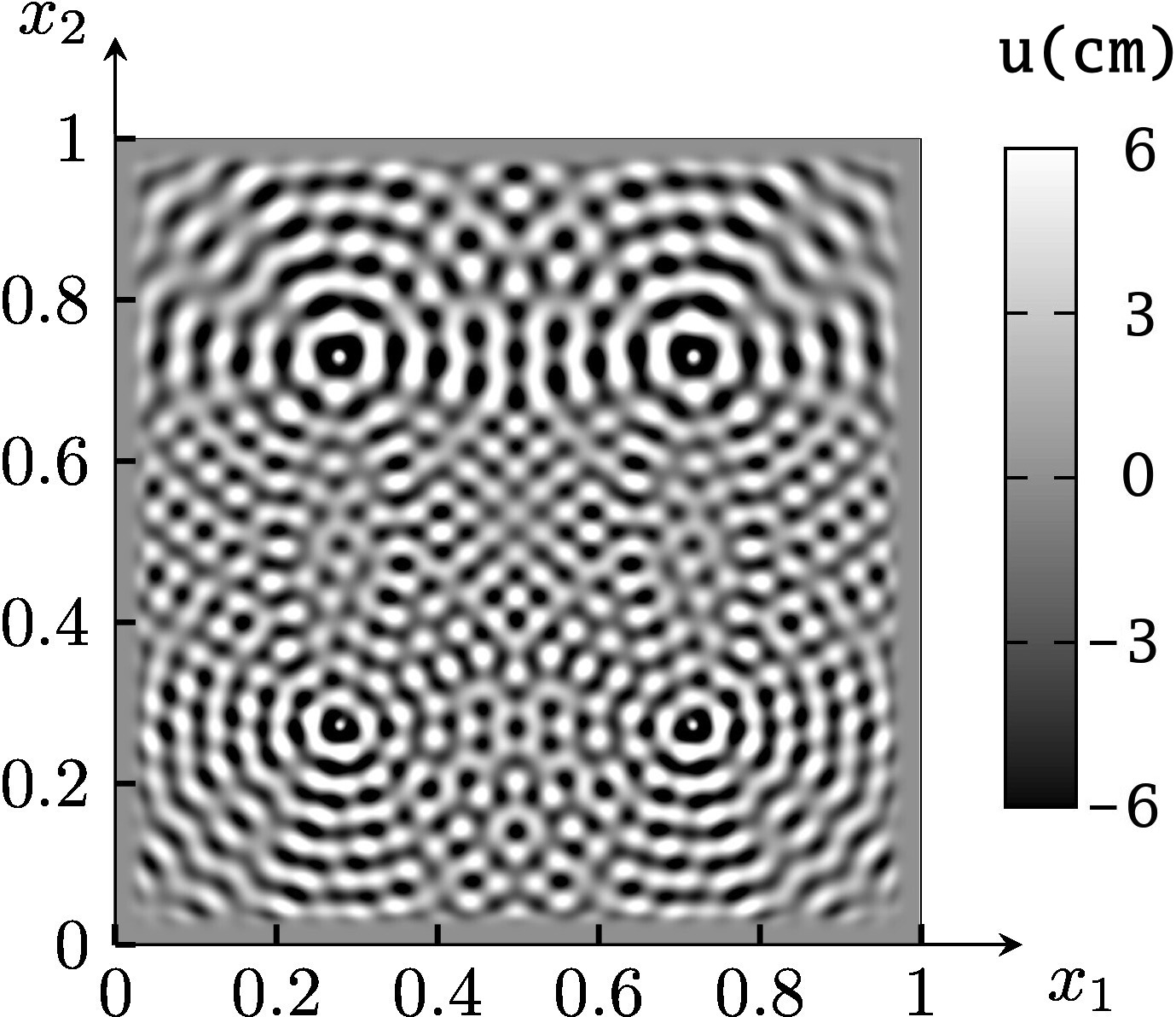}
  \end{minipage}
  \caption{Velocity profile of the 2D two-layered media (left) and the real part of the DDM solution with one source (middle) and four sources (right), respectively. \label{fig:two-layered-sol} }
\end{figure*}

\begin{table}[ht!]
        \centering
        \begin{tabular}{|c|c|r|c|r|c|}
                \hline
                Mesh     & Local Size  & \mrow{$L^2$ Error}   & Conv. & \mrowcol{$H^1$  Error} & Conv. \\
                Size     & without PML &                      & Rate  &                        & Rate  \\
                \hline\hline
                270$^2$  & 60$^2$      & 1.37$\times 10^{-2}$ &   -    & 1.97$\times 10^{0}\,\,\,\,$    &   -    \\
                540$^2$  & 120$^2$     & 3.39$\times 10^{-3}$ & 2.0   & 4.91$\times 10^{-1}$   & 2.0     \\
                1080$^2$ & 240$^2$     & 8.41$\times 10^{-4}$ & 2.0   & 1.22$\times 10^{-1}$   & 2.0   \\
                2160$^2$ & 480$^2$     & 2.10$\times 10^{-4}$ & 2.0   & 3.05$\times 10^{-2}$   & 2.0   \\
                4320$^2$ & 960$^2$     & 5.25$\times 10^{-5}$ & 2.0   & 7.63$\times 10^{-3}$   & 2.0   \\
                \hline
        \end{tabular}
        \caption{The errors and convergence rates of numerical solutions by the trace transfer-based diagonal sweeping  DDM  for the two-layered media problem in $\R^2$ with one source.} \label{tab:two-layered}
\end{table}

\begin{table}[ht!]
        \centering
        \begin{tabular}{|c|c|r|c|r|c|}
                \hline
                Mesh     & Local Size  & \mrow{$L^2$ Error}   & Conv. & \mrowcol{$H^1$  Error} & Conv. \\
                Size     & without PML &                      & Rate  &                        & Rate  \\
                \hline\hline
                270$^2$  & 60$^2$      & 2.65$\times 10^{-2}$ &   -    & 3.89$\times 10^{0}\,\,\,\,$    &   -    \\
                540$^2$  & 120$^2$     & 6.66$\times 10^{-3}$ & 2.0   & 9.89$\times 10^{-1}$   & 2.0     \\
                1080$^2$ & 240$^2$     & 1.65$\times 10^{-4}$ & 2.0   & 2.46$\times 10^{-1}$   & 2.0   \\
                2160$^2$ & 480$^2$     & 4.11$\times 10^{-4}$ & 2.0   & 6.12$\times 10^{-2}$   & 2.0   \\
                4320$^2$ & 960$^2$     & 1.03$\times 10^{-4}$ & 2.0   & 1.53$\times 10^{-2}$   & 2.0   \\
                \hline
        \end{tabular}
        \caption{The errors and convergence rates of numerical solutions by the trace transfer-based diagonal sweeping  DDM  for the two-layered media problem in $\R^2$ with four sources.} \label{tab:two-layered4}
\end{table}

%%%%%%%%%%%%%%%%%%%%%%%%%%%%%%%%
\subsubsection{3D  constant medium problem}
%%%%%%%%%%%%%%%%%%%%%%%%%%%%%%%%

Next a 3D constant medium problem  is used to test  Algorithm \ref{alg:diag3D_t}.  The computational domain is $B_L=[-L,L]^3$ with $L = 1/2$, and the interior domain without PML is $B_l=[-l,l]^3$ with $l = 3/8$.
For the wave number $\kappa / 2\pi = 11$,  a series of uniformly refined meshes are used,   of which the mesh density increases approximately from $7$ to $14$ . 
The domain partition is fixed to be $3 \times 3 \times 3$  for all the meshes, and the source is defined as, 
\begin{equation*}
f(x_1, x_2, x_3) = \frac{64 \kappa^3}{\pi^{9/2}}  e^{- (\frac{4 \kappa}{\pi})^2 ((x_1-r_1)^2+(x_2-r_2)^2+(x_3-r_3)^2) },
\end{equation*}
where $(r_1, r_2, r_3) = (0.12, 0.13, 0.14) $.
%
%
%and the source is supported in 
% four subdomains $\Omega_{3,4}$, $\Omega_{3,5}$, $\Omega_{4,4}$ and $\Omega_{4,5}$.
%region $\Omega_{3,4;4,5}$. 
%
The optimal second-order convergence of the errors in $L^2$ and $H^1$ norms along the mesh refinement are again obtained, as is reported in Table \ref{tab:const3D}. %which implies the total error is indeed dominated by the 2nd order finite difference discretization error in the constant medium case.
%
%The results of the algorithm with trace and source transfer are very close to each other, the differences in error are in the fourth digits, hence they are not shown in the table.
%Again, the differences in error between trace and source transfer are too small to show in the table.

\begin{table}[ht!]
        \centering
        \begin{tabular}{|c|c|r|c|r|c|}
                \hline
                Mesh    & Local Size  & \mrowcol{$L^2$ Error} & Conv. & \mrowcol{$H^1$  Error}       & Conv. \\
                Size    & without PML &                       & Rate  &                              & Rate  \\
                \hline\hline
                80$^3$  & 20$^2$      & 3.67$\times 10^{-2}$  &  -     & 2.43$\times 10^{0}\,\,\,\,$  &  -     \\
                96$^3$  & 24$^2$      & 2.49$\times 10^{-2}$  & 2.1   & 1.67$\times 10^{0}\,\,\,\,$  & 2.0   \\
                128$^3$ & 32$^2$      & 1.37$\times 10^{-2}$  & 2.1   & 9.33$\times 10^{-1}\,\,\,\,$ & 2.0   \\
                160$^3$ & 40$^2$      & 8.76$\times 10^{-3}$  & 2.0   & 5.99$\times 10^{-2}\,\,\,\,$ & 2.0   \\
                \hline
        \end{tabular}
        \caption{The errors and convergence rates of  numerical solutions by the trace transfer-based diagonal sweeping  DDM  for the constant medium problem in $\R^3$.} \label{tab:const3D}
\end{table}

%%%%%%%%%%%%%%%%%%%%%%%%%%%%%%%%
\subsection{Performance tests with the DDM as the preconditioner} % for the global discrete system}
%%%%%%%%%%%%%%%%%%%%%%%%%%%%%%%%

%\def\Tave{\widetilde{T}_{\text{it}}}
\def\Tave{{T}_{\text{par}}}

The trace transfer-based diagonal sweeping DDM  (Algorithms \ref{alg:diag2D_t} and \ref{alg:diag3D_t}) are now tested as the  preconditioner for the GMRES solver of  the global discretized system of the Helmholtz problem, 
each preconditioner solving consists four diagonal sweeps and $\Nbx + \Nby -1$ sequential steps per sweep for 2D problem,
or eight diagonal sweeps and $\Nbx + \Nby + \Nbz -2 $ sequential steps per sweep in 3D problem.
In all the tests, the stopping criterion is that the relative residual is less than 10$^{-6}$. 
The effectiveness and  efficiency of the proposed method are evaluated  by repeatedly increasing the number of subdomains and frequencies while the subdomain problem size and mesh density remain fixed.  The performance of the proposed method is also compared to that of the diagonal sweeping DDM with source transfer \cite{Leng2020}. In addition, we demonstrate the scalability of the proposed method when 
combined with the pipeline processing through solving  multiple RHSs problems. 
%

%%%%%%%%%%%%%%%%%%%%%%%%%%%%%%%%
\subsubsection{2D constant medium problem}
%%%%%%%%%%%%%%%%%%%%%%%%%%%%%%%%

A constant medium problem on the 
square domain $[0,1]\times[0,1]$  is first used for testing and comparison between the proposed method 
and the diagonal sweeping DDM with source transfer.
The source contains four approximated delta sources, that is
\begin{equation} \label{eq:four_src}
%\DD f_{i,j}(x_1, x_2) = \sum\limits_{p,q=1,2} \frac{1}{h_1 h_2}\delta(x_1^p - i h_1) \delta (x_2^q - j h_2 ),
\DD f(x_1, x_2) = \sum\limits_{\substack{i = 1, \ldots, N_x \\ j = 1, \ldots, N_y}} \sum\limits_{s=1,\ldots,4} \frac{1}{h_1 h_2}\delta(x_1^s - i h_1) \delta (x_2^s - j h_2 ),
\end{equation}
where $(x^s_1,x^s_2) = (1/4,1/4)$, $(1/4,3/4)$, $(3/4,1/4)$ and $(3/4,3/4)$, for $s=1,\ldots,4$, and $h_{1}$ and $h_{2}$ are the grid spacing in $x$ and $y$ directions, respectively.
The frequency and number of subdomains increase as the mesh resolution increases, while the size of the subdomain problems is fixed to be $300 \times 300$ and the mesh density is kept to be 10. The PML layer is of 30 points (i.e., 3 wave length).
% \blue{$\Tfact$}
The factorization time of the subdomains is around $2.3 \sim 2.4$ seconds for all partitions.  
The  performance results of diagonal sweeping DDMs with trace transfer or source transfer as the preconditioner  are reported and compared in Table \ref{tab:iter_const}.   It is observed that the number of needed GMRES iterative steps (denoted by $n_{\text{iter}}$) grows roughly proportionally to $\log(\Nbx + \Nby)$ or $\log(\omega)$ for both methods, which implies that  they  are highly efficient. In addition, 
 the method based on trace transfer does relatively better in efficiency than the one based on source transfer  as expected, i.e., the GMRES iteration times (denoted by $\Tsolve$) are  around 6.7\%--14.6\% smaller.
   
\begin{table}[ht!]
  \centering
\begin{tabular}{|r|r|r|c|c|}
	\hline
	
	\mcol{Mesh} & \mrow{$N_1 \times N_2$} & Freq.          & GMRES & \mrowcol{$\Tsolve$} \\
	\mcol{Size} &                       & $\omega/2\pi$    & $n_{\text{iter}}$ &           \\
	
	\hline
	\hline
  600$^2$  & 2 $\times$ 2   & 65.9 & 3 (3) & 29.3 (32.0) \\
  1200$^2$ & 4 $\times$ 4   & 126  & 3 (3) & 64.4 (73.8) \\
  2400$^2$ & 8 $\times$ 8   & 246  & 3 (3) & 123.5 (136.7) \\
  4800$^2$ & 16 $\times$ 16 & 486  & 3 (3) & 268.4 (285.9) \\
  9600$^2$ & 32 $\times$ 32 & 966  & 4 (4) & 715.3 (766.1) \\
	\hline
\end{tabular}

  \caption{Performance comparison of the  diagonal sweeping DDMs based on trace transfer vs. source transfer  (shown in brackets), as an iterative solver  for the 2D constant medium problem when the subdomain problem size is fixed.} \label{tab:iter_const}
\end{table}

\subsubsection{The BP-2004 model}

The performance of the proposed method is then tested with the 2D BP-2004 velocity model \cite{Billette2004}, which is a popular benchmark problem in seismic exploration to study reverse time migration.
The middle part of the model that contains the slat body is used, which is $24$ km in length and $12$ km in depth,
and the velocity range is $[1, 5]$ km/s, as shown in Figure \ref{fig:seg04_vel}.
The efficiency and parallel scalability of the proposed method with pipeline processing is tested with a multiple RHSs problem,
where a set of randomly distributed shots at a depth of one-fourth of subdomain height are used as sources, and the total number of shots is $N_{\text{RHS}} = 2 (\Nbx + \Nby - 1)$ for a $\Nbx \times \Nby $ domain partition.
A series of uniformly refined meshes are used, and the frequencies increase as the mesh refines, so that the minimum mesh density is kept to be 12.
The number of subdomain also increases as the mesh refines, while the size of subdomain is kept $180  \times 180$ without PML.
An approximate solution to the problem is presented in Figure \ref{fig:seg04_sol}, where the angular frequency is $\omega/2\pi = 14.6$ and the source is one of the random shots.
%\blue{$\Tfact$}
The factorization time of the subdomains is around $2.3 \sim 2.7$ seconds for all partitions.    
The performance results of diagonal sweeping DDMs with trace transfer or source  transfer as the preconditioner are reported in Table \ref{tab:iter_BP2004}.
It is easy to see that
the number of GMRES iterative steps still grows roughly  proportional to  $\log(\omega)$ and  so does the average solving time $\Tave := \frac{\Tsolve}{N_{\text{RHS}}}$. 
%which demonstrates high efficiency of  the diagonal sweeping DDMs with the pipeline processing. 
The weak scaling efficiency is around 86\%--87\% with 1024 processors, which shows that the method  with the pipeline processing is highly scalable.
The   trace transfer-based method  again does relatively better than the source transfer-based one  as expected, i.e., the GMRES iteration times are  around 2.6\%--4.1\% smaller.

\begin{figure}[!ht]
 \centerline{
   \includegraphics[width=.65\textwidth]{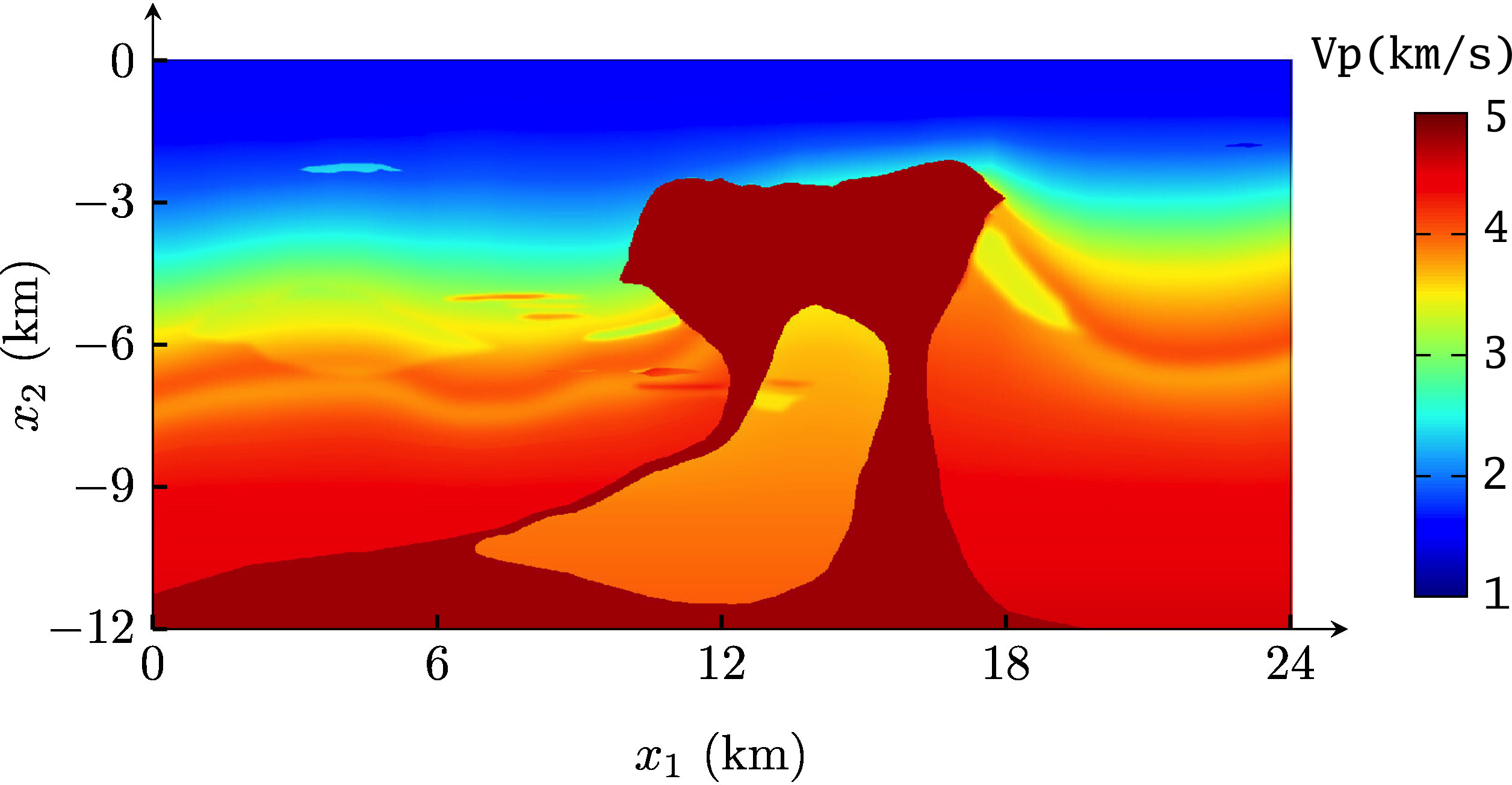} 
 }
 \vspace{-0.4cm}
 \caption{The velocity profile of the BP-2004 model.
 }
 \label{fig:seg04_vel}
\end{figure}

\begin{figure}[!ht]
 \centerline{
   \includegraphics[width=.65\textwidth]{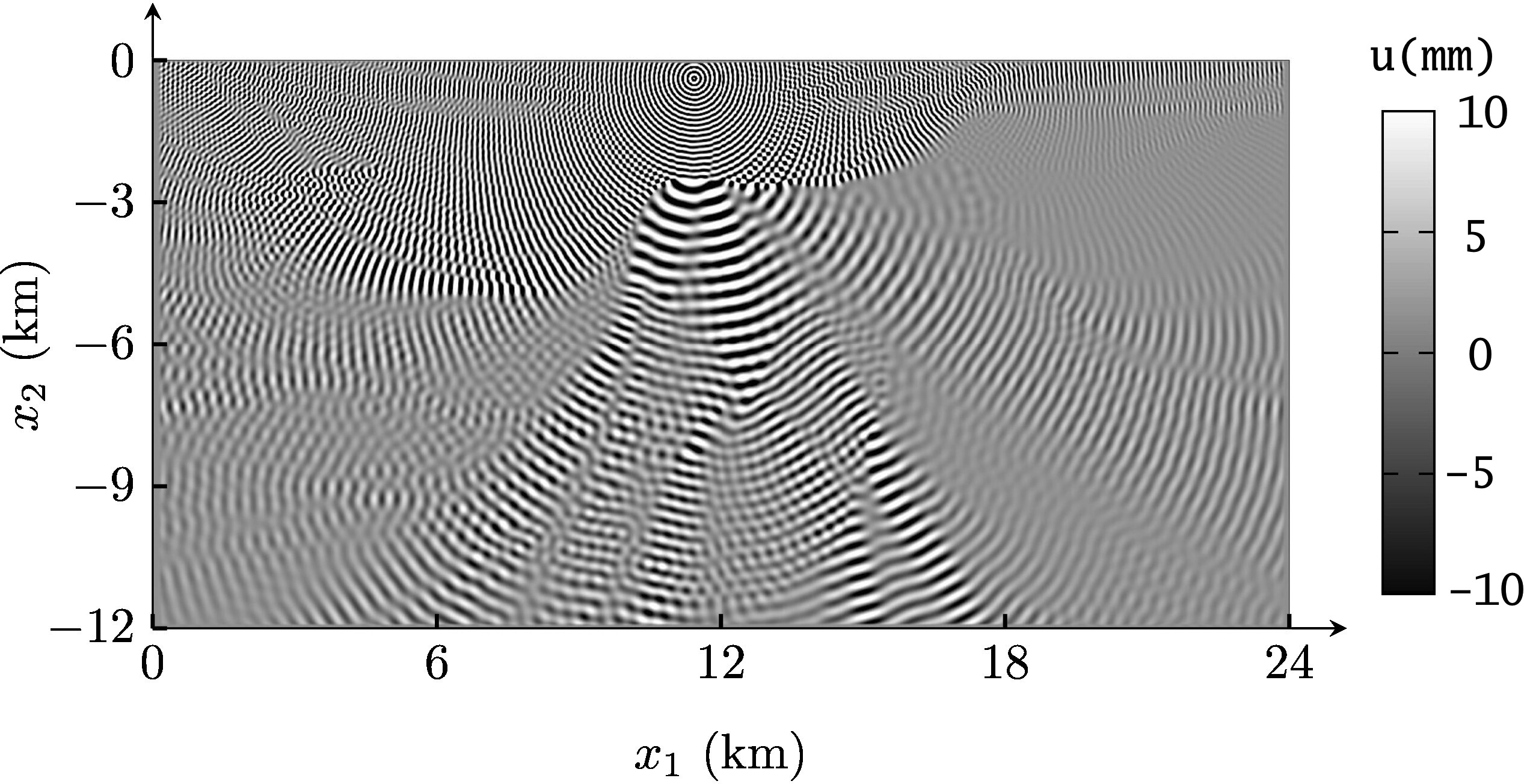} 
 }
 \vspace{-0.2cm}
 \caption{
 The real part of the approximate solution to the problem of the shot located at (11.399 km, -0.3675 km) with the angular frequency $\omega/2\pi = 14.6$ on the mesh of size $2880^2$  in the BP-2004 model.}
 \label{fig:seg04_sol}
\end{figure}

\begin{table}[ht!]
 \centering 
\begin{tabular}{|r|r|r|r|c|c|c|c|}
\hline
  \mcol{Mesh} & \mrow{$N_1 \times N_2$} & Freq.         & \mrowcol{$N_{\rm RHS}$} & GMRES & \mrowcol{$\Tsolve$} & \mrowcol{$\Tave$} & \mcol{Weak scaling}\\
  \mcol{Size} &                       & $\omega/2\pi$ &                     & $n_{\text{iter}}$ &                          &                            &  \mcol{efficiency}     \\

\hline
\hline
% 360$^2$  & 2 $\times$ 2   & 2.08 & 6   & 4(4) & 48.4 (48.7) & 8.07 (8.28) & -(-)\\ %4 & 48.4 & 8.07 \\
 720$^2$  & 4 $\times$ 4   & 3.87 & 14  & 6 (6) & 148.8 (153.1)   & 10.6 (10.9) & 100 \%(100\%)\\ 
 1440$^2$ & 8 $\times$ 8   & 7.44 & 30  & 6 (6) & 318.8 (328.2)  & 10.6 (10.9) & 100 \%(100\%)\\ 
 2880$^2$ & 16 $\times$ 16 & 14.6 & 62  & 7 (7) & 747.7 (774.7)   & 12.1 (12.5) & 88\% (88\%)\\ 
 5760$^2$ & 32 $\times$ 32 & 28.9 & 126 & 7 (7) & 1537.7 (1603.5) & 12.2 (12.7) & 87\% (86\%)\\ 
 \hline
\end{tabular}
\caption{The performance of the diagonal sweeping DDM based on trace transfer vs. source transfer  (shown in brackets)  as the preconditioner  for the BP-2004 problem when the subdomain problem size is fixed.
} \label{tab:iter_BP2004}
\end{table}

%%%%%%%%%%%%%%%%%%%%%%%%%%%%%%%%
\subsubsection{3D layered media problem}
%%%%%%%%%%%%%%%%%%%%%%%%%%%%%%%%
We now test a 3D layered media problem on the 
cuboidal domain $[0,1]\times[0,1]\times[0,1]$, in which the five-layered media is given as that in \cite{Leng2020}, see Figure \ref{fig:layer-3d_vel}-(left).
The efficiency and parallel scalability of the proposed method with pipeline processing is again tested with a multiple RHSs problem,
where a set of randomly distributed shots on the square $[0.15, 0.85]\times[0.15, 0.85]\times[0.85, 0.85]$ are tested as sources, and the total number of shots is $N_{\text{RHS}} = 2 (\Nbx + \Nby + \Nbz -2)$ for a $\Nbx \times \Nby \times \Nbz$ domain partition.
The frequencies and number of subdomains increase as the mesh resolution increase, while the mesh density is kept 9 and the size of the subdomain problems is fixed to be $32^3$. The PML layer is 10 points wide, which is approximately of 1 wave length. 
An approximate solution to the problem is presented in Figure  \ref{fig:layer-3d_vel}-(right), where the angular frequency is $\omega/2\pi = 25.5$ and the source is one of the random shots.
%\blue{$\Tfact$} 
The factorization time of the subdomains is around $60.4 \sim 61.2$ seconds for all partitions.  
The performance results of the diagonal sweeping DDMs with  trace transfer or source transfer are reported in Table \ref{tab:iter_const_3D}. 
We observe that both the number of GMRES iterative steps and the average solving time $\Tave$ both grow roughly proportionally to $\log(\Nbx+\Nby+\Nbz)$ or $\log(\omega)$.
%which demonstrates that the diagonal sweeping DDMs with pipeline processing are highly efficient.
 The weak scaling efficiency is 63\% with 1000 cores, which shows that the proposed method is still well scalable. Furthermore the  trace transfer-based method  again does relatively better in efficiency  than the source transfer-based one on as expected, i.e., the GMRES iteration times are  around 1.2\%--3.1\% smaller.

\begin{figure}[!ht]
	\centerline{
          \includegraphics[width=.4\textwidth]{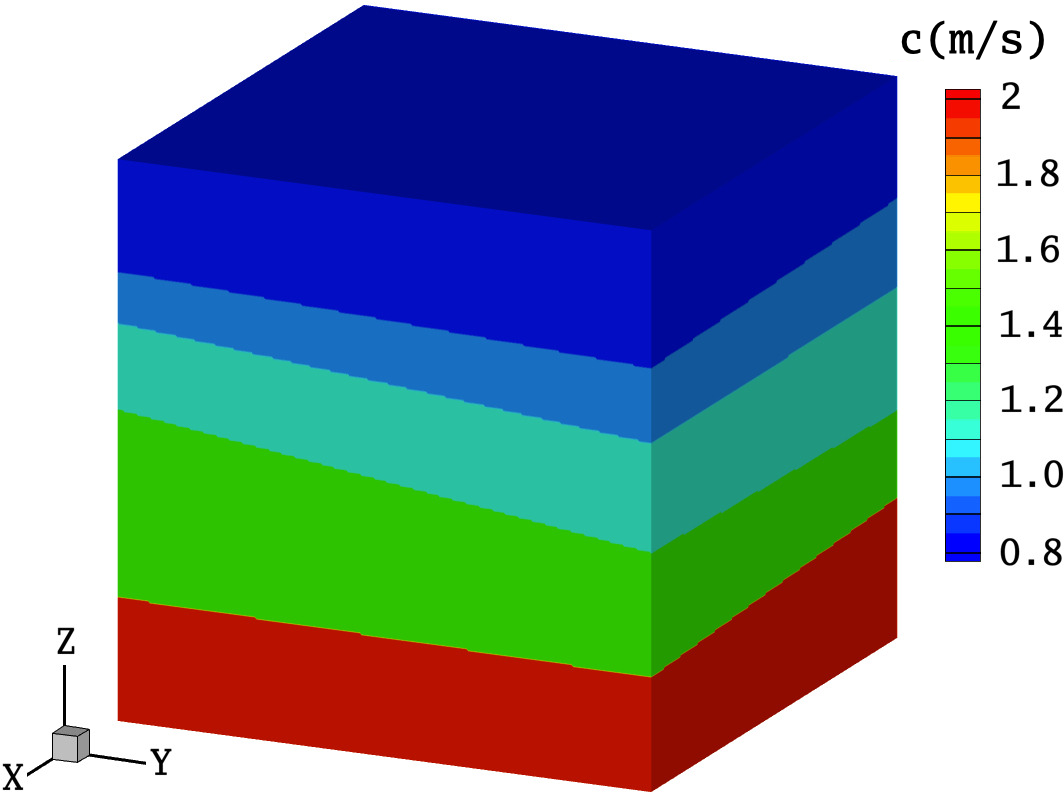}\hspace{1cm}
          \includegraphics[width=.4\textwidth]{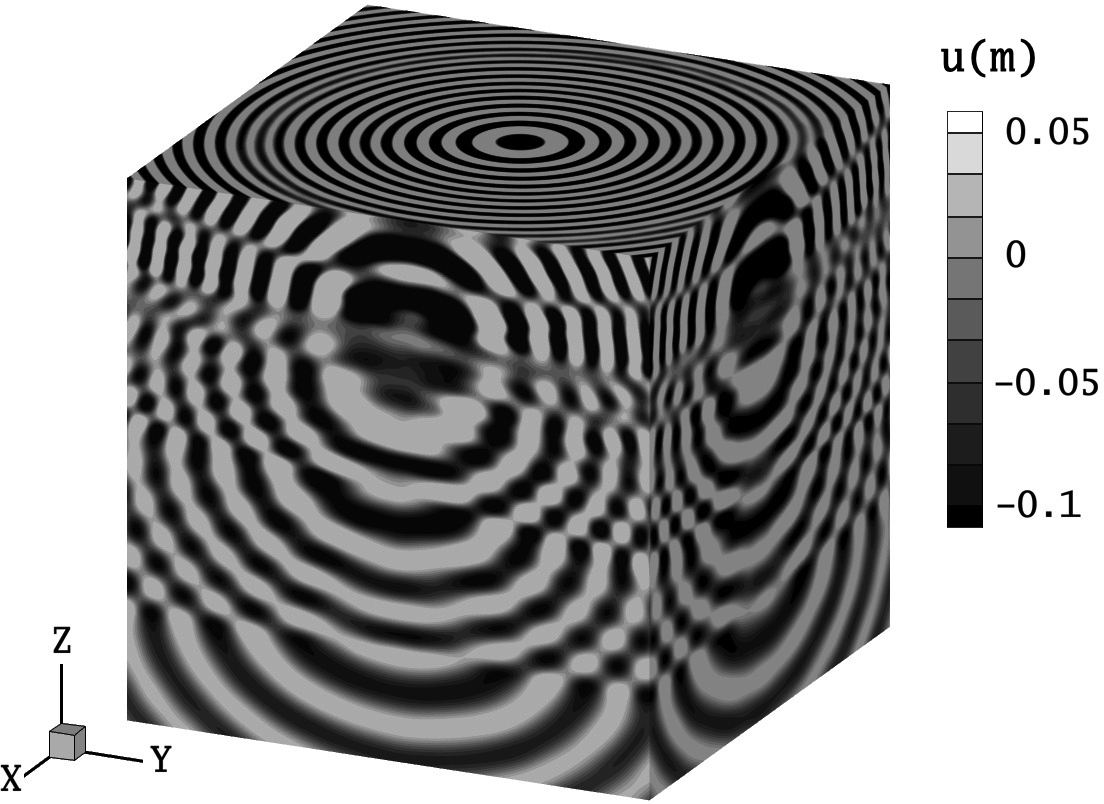} 
	}
	\caption{The velocity profile of the 3D layered media problem (left) and the real part of the approximate solution (right) to the problem of the shot located at (0.539, 0.539, 0.85) with the angular frequency $\omega/2\pi = 25.5 $ on the mesh of size $256^3$.
	}
	\label{fig:layer-3d_vel}
\end{figure}

\begin{table}[ht!]
 \centering 
\begin{tabular}{|r|r|r|r|c|c|c|c|}
\hline

  \mcol{Mesh} & \mrow{$N_1 \times N_2 \times N_3$} & Freq.         & \mrowcol{$N_{\rm RHS}$} & GMRES & \mrowcol{$\Tsolve$ } & \mrowcol{$\Tave$} & \mcol{Weak scaling}\\
  \mcol{Size} &                       & $\omega/2\pi$ &                     & $n_{\text{iter}}$ &                          &                            &  \mcol{efficiency}     \\

  \hline
  \hline 
  64$^3$  & 2 $\times$ 2 $\times$ 2    & 7.69 & 8  & 4 (4) & 261.1 (269.0)   & 32.6 (33.6) & 100\% (100\%) \\
  128$^3$ & 4 $\times$ 4 $\times$ 4    & 13.6 & 20 & 5 (5) & 767.2 (792.4)   & 38.4 (39.6) & 85\% (85\%) \\ 
  192$^3$ & 6 $\times$ 6 $\times$ 6    & 19.5 & 32 & 5 (5) & 1388.3 (1396.2) & 43.1 (43.6) & 75\% (77\%) \\
  256$^3$ & 8 $\times$ 8 $\times$ 8    & 25.5 & 44 & 5 (5) & 1915.1 (1981.5) & 43.5 (45.0) & 75\% (75\%) \\
  320$^3$ & 10 $\times$ 10 $\times$ 10 & 31.4 & 56 & 6 (6) & 2899.0 (2984.1) & 51.8 (53.3) & 63\% (63\%) \\
  \hline 

\end{tabular}
  
  \caption{The performance of the  diagonal sweeping  DDM based on trace transfer vs. source transfer  (shown in brackets) as the  precondtioner for the 3D layered media problem
    when the subdomain problem size is fixed. } \label{tab:iter_const_3D}
\end{table}

%%%%%%%%%%%%%%%%%%%%%%%%%%%%%%%%%%%%%%%%%%%%%%%%%%%%%%%%%%%%%%%%
\section{Conclusions}
%%%%%%%%%%%%%%%%%%%%%%%%%%%%%%%%%%%%%%%%%%%%%%%%%%%%%%%%%%%%%%%%

In this paper, we propose and analyze a  trace transfer-based diagonal sweeping DDM for solving the high-frequency Helmholtz equation. Compared 
to the source transfer-based  diagonal sweeping DDM  in \cite{Leng2020}, the basic transfer method is changed from source transfer to trace transfer, and consequently the number of directions of transferred information is greatly reduced, which makes the proposed method easier to analyze and implement and more efficient.
We prove that the proposed method not only produces the exact PML solution in the constant medium case, but also does it with at most one extra round of diagonal sweeps, which lay down the theoretical foundation of the method as direct solver or preconditioner. Extensive numerical experiments in two and three dimensions  are also carried out to demonstrate the performance and parallel scalability (with the pipeline processing) of the proposed method.
On the other hand, there still remain lots of problems worthy of further investigation. Implementations of higher-order discretizations for the diagonal sweeping DDM with source transfer are quite straightforward, but that for the  trace transfer-based one  is quite different, for example, the high-order numerical quadratures for the potential computation with  singular integrands  need to be studied, and special treatments of the quadratures around the crossings of subdomain interfaces also need to be developed. Also
the PML boundary for sweeping DDMs has been shown to be effective, however, the absorption deteriorates for the near-grazing incident waves that come from the source near the boundary, which may cause the efficiency of the sweeping DDMs deteriorates. Another difficulty involving PML emerges when extending the method to Maxwell's equation and the elastic wave equation, that is, the required number of grid points in the PML to reach satisfactory absorption effect is usually too many, which may bring much extra computation for the subdomain problems. Therefore, the development of more effective PML is also a vital direction of our future research.

%\section*{Acknowledgments}
%W. Leng's research is partially supported by
%National Key R\&D Program of China under grant number 2020YFA0711904,
%National Natural Science Foundation of China under grant number 11771440, and National Center for Mathematics and Interdisciplinary Sciences of Chinese Academy of Sciences (NCMIS). L. Ju's research is partially supported by US National Science Foundation under grant number DMS-1818438.

%%%% Bibliography  %%%%%%%%%%

\end{document}